\documentclass[10pt,a4paper]{amsart}

\usepackage[subsec,skipmbb,skipfonts]{rune}

\usepackage[full]{textcomp} 
\usepackage[p,osf]{fbb} 
\usepackage[scaled=.95,type1]{cabin} 
\usepackage[varqu,varl]{zi4}
\usepackage[T1]{fontenc} 
\usepackage[libertine]{newtxmath}
\usepackage[bb=boondox,frak=boondox]{mathalfa}
\DeclareMathAlphabet{\mathbfsf}{\encodingdefault}{\sfdefault}{bx}{sl}

\usepackage{runediag}

\renewcommand{\catname}[1]{\textup{\textmd{\textsf{#1}}}}
\renewcommand{\Cat}{\catname{Cat}}
\renewcommand{\Mon}{\catname{Mon}}
\renewcommand{\Seg}{\catname{Seg}}

\renewcommand{\Span}{\catname{Span}}

\renewcommand{\Fun}{\catname{Fun}}
\renewcommand{\Map}{\catname{Map}}

\newcommand{\pr}{\txt{pr}}

\newcommand{\ev}{\txt{ev}}

\newcommand{\cocabb}{cocart}

\newcommand{\coc}{\txt{\cocabb}}

\newcommand{\dcoc}{\txt{-\cocabb}}

\newcommand{\CatIsl}[1]{\Cat_{\infty/#1}}

\newcommand{\CatIzcoc}[1]{\CatIsl{#1}^{#1_{0}\dcoc}}

\newcommand{\CatILcoc}[1]{\CatIsl{#1}^{L\dcoc}}
\newcommand{\CatIBL}{\CatILcoc{\mathcal{B}}}

\newcommand{\Funzcoc}[1]{\Fun_{/#1}^{#1_{0}\dcoc}}

\newcommand{\CatIcoc}[1]{\CatIsl{#1}^{\coc}}

\renewcommand{\itcat}{($\infty$,2)-category}
\renewcommand{\itcats}{($\infty$,2)-categories}
\newcommand{\xF}{\mathbb{F}}

\newcommand{\Lfbrs}{L_{\txt{fbrs}}}
\newcommand{\Lrseg}{L_{\txt{rseg}}}

\theoremstyle{plain}
\newtheorem{thmA}{Theorem}
\newtheorem{corA}[thmA]{Corollary}

\makeatletter
\renewcommand{\to}{%
  \if@display\longrightarrow%
  \else\rightarrow%
  \fi%
}
\renewcommand{\xto}[1]{%
  \if@display\xlongrightarrow{#1}%
  \else\xrightarrow{#1}%
  \fi%
}
\renewcommand{\from}{%
  \if@display\longleftarrow%
  \else\leftarrow%
  \fi%
}
\renewcommand{\xfrom}[1]{%
  \if@display\xlongleftarrow{#1}%
  \else\xleftarrow{#1}%
  \fi%
}
\renewcommand{\isoto}{\xto{\sim}}
\makeatother

\newcommand{\frXb}{\frX^{b}}
\newcommand{\frXf}{\frX^{f}}

\newcommand{\Spanbf}{\Span_{b,f}}

\theoremstyle{definition}

\newcommand{\ostar}{\circledast}

\usepackage{extarrows}

\newcommand{\Mnd}{\txt{Mnd}}

\newcommand{\act}{\txt{act}}
\newcommand{\el}{\txt{el}}

\crefformat{equation}{(#2#1#3)}

\let\lim\relax \DeclareMathOperator{\lim}{lim}
\let\colim\relax \DeclareMathOperator{\colim}{colim}

\usepackage[scaled=.98]{BOONDOX-uprscr}

\DeclareSymbolFont{bbold}{U}{bbold}{m}{n}
\DeclareSymbolFontAlphabet{\mathbbold}{bbold}

\renewcommand{\Map}{\textup{\textsf{Map}}}
\renewcommand{\Fun}{\textup{\textsf{Fun}}}

\renewcommand{\Cat}{\textup{\textsf{Cat}}}

\renewcommand{\CatI}{\Cat_{\infty}}

\newcommand{\SpF}{\Span(\xF)}

\newcommand{\AlgPatt}{\catname{AlgPatt}}

\newcommand{\bfone}{\textbf{\textup{\textlf{1}}}}
\newcommand{\fset}[1]{\textbf{\textup{\textlf{#1}}}}

\newcommand{\Env}{\catname{Env}}
\newcommand{\actto}{\rightsquigarrow}
\newcommand{\intto}{\rightarrowtail}
\newcommand{\xint}{\txt{int}}
\newcommand{\xactto}[1]{\overset{#1}{\actto}}
\newcommand{\xintto}[1]{\overset{#1}{\intto}}
\usetikzlibrary{decorations.pathmorphing}  
  \tikzcdset{
    active/.style={rightsquigarrow},
    lactive/.style={leftsquigarrow},    
    inert/.style={tail}
  }

\DeclareMathOperator{\Ar}{Ar}

\newcommand{\ArB}{\Ar(\mathcal{B})}

\newcommand{\ArR}{\Ar_{R}}
\newcommand{\ArL}{\Ar_{L}}
\newcommand{\ArRB}{\ArR(\mathcal{B})}
\newcommand{\ArLB}{\ArL(\mathcal{B})}
\newcommand{\Aract}{\Ar_{\mathrm{act}}}
\newcommand{\AractO}{\Aract(\mathcal{O})}
\newcommand{\Fbrs}{\catname{Fbrs}}

\newcommand{\intcoc}{\txt{int-cocart}}

\newcommand{\adj}{\rightleftarrows}

\newcommand{\ns}{\mathrm{ns}}
\newcommand{\gen}{\mathrm{gen}}

\newcommand{\Dbl}{\catname{DblCat}_{\infty}}

\newcommand{\Orb}{\mathrm{Orb}}

\newcommand{\CMon}{\catname{CMon}}

\newcommand{\pt}{\ast}

\newcommand{\frX}{\mathfrak{X}}
\newcommand{\frY}{\mathfrak{Y}}

\newcommand{\calA}{\mathcal{A}}
\newcommand{\calB}{\mathcal{B}}
\newcommand{\calC}{\mathcal{C}}
\newcommand{\calD}{\mathcal{D}}
\newcommand{\calE}{\mathcal{E}}
\newcommand{\calF}{\mathcal{F}}

\newcommand{\calP}{\mathcal{P}}

\newcommand{\calO}{\mathcal{O}}
\newcommand{\calR}{\mathcal{R}}

\newcommand{\calK}{\mathcal{K}}

\newcommand{\calS}{\mathcal{S}}
\newcommand{\calQ}{\mathcal{Q}}

\newcommand{\rmE}{\mathrm{E}}
\newcommand{\rmQ}{\mathrm{Q}}

\newcommand{\calV}{\mathcal{V}}

\newcommand{\Sub}{\mathrm{Sub}}

\newcommand{\ulF}{\underline{\xF}}
\newcommand{\OpdG}{\Opd_{G,\infty}}
\newcommand{\CatG}{\Cat_{G,\infty}}

\newcommand{\sliceEnv}[1]{\Env_{#1}^{/\calA_{#1}}}
\newcommand{\St}[2]{\mathrm{St}_{#1}^{#2}}
\newcommand{\StBL}{\St{\calB}{L}}
\newcommand{\StOint}{\St{\calO}{\xint}}
\newcommand{\Un}[2]{\mathrm{Un}_{#1}^{#2}}
\newcommand{\UnBL}{\Un{\calB}{L}}
\newcommand{\UnOint}{\Un{\calO}{\xint}}
\newcommand{\slSeg}[1]{\Seg_{#1}^{/\Afun{#1}}}
\newcommand{\slSegC}[1]{\slSeg{#1}(\CatI)}

\newcommand{\slSegCO}{\slSegC{\mathcal{O}}}

\newcommand{\CatIOintcoc}{\Cat_{\infty/\mathcal{O}}^{\intcoc}}
\newcommand{\CatIOcoc}{\Cat_{\infty/\mathcal{O}}^{\txt{cocart}}}
\newcommand{\Afun}[1]{\mathcal{A}_{#1}}
\newcommand{\Agpd}[1]{\mathcal{A}_{#1}^{\simeq}}

\newcommand{\Spansitdeg}{\Span_{\txt{si},\txt{tdeg}}}

\renewcommand{\phi}{\varphi}

\title{Envelopes for Algebraic Patterns}

\author{Shaul Barkan}
\address{Hebrew University of Jerusalem, Israel}

\author{Rune Haugseng}
\address{Norwegian University of Science and Technology (NTNU),
  Trondheim, Norway}
\urladdr{http://folk.ntnu.no/runegha}

\author{Jan Steinebrunner}
\address{Gonville \& Caius College, Cambridge, UK}
\urladdr{http://jan-steinebrunner.com}

\date{\today}

\begin{document}

\begin{abstract}
  We generalize Lurie's construction of the symmetric monoidal
  envelope of an $\infty$-operad to the setting of algebraic
  patterns. 
  This envelope becomes fully faithful when sliced over the envelope of the terminal object, and we characterize its essential image.
  Using this, we prove a comparison result
  that allows us to compare analogues of $\infty$-operads over various algebraic patterns.
  In particular, 
  we show that the $G$-$\infty$-operads of Nardin-Shah 
  are equivalent to ``fibrous patterns'' over the $(2,1)$-category $\Span(\xF_G)$ of spans of finite $G$-sets.
  When $G$ is trivial this means that Lurie's $\infty$-operads can equivalently be defined over $\Span(\xF)$ instead of $\xF_*$.
\end{abstract}

\maketitle

\tableofcontents

\section{Introduction}
\label{sec:introduction}

In Lurie's seminal work on homotopy-coherent algebra \cite{HA}, the
main objects used to encode algebraic structures are (symmetric)
\emph{\iopds{}}, which are defined as a certain type of functor of
\icats{} $\mathcal{O} \to \xF_{*}$, where $\xF_{*}$ is the
category of finite pointed sets. However, as illustrated already in
\cite{HA}, it can sometimes be useful to consider variants of this
notion, for instance because they give a combinatorially simpler
description of some structure. For example, Lurie also considers
\emph{planar} (or non-symmetric) \iopds{}, where the category
$\xF_{*}$ is replaced by the simplex category $\simp^{\op}$. As a
special case of a general comparison theorem \cite{HA}*{Theorem
  2.3.3.26} using the theory of \emph{approximations} to \iopds{},
Lurie proves that there is an equivalence of \icats{} between planar
\iopds{} and \iopds{} over the (symmetric) associative operad
$\catname{Ass}$, given by pulling back along an explicit map
$\simp^{\op} \to \catname{Ass}$.

Our main goal in this paper is to prove a more general version of such
comparisons. Before we explain this result in more detail, let us motivate
it by (informally) stating the two main new comparisons we will apply it to:
\begin{itemize}
\item In the definition of symmetric \iopds{}, we can equivalently
  replace the category $\xF_{*}$ of finite pointed sets by the
  (2,1)-category $\Span(\xF)$ of \emph{spans} of finite sets.
\item For $G$ a finite group, the $G$-equivariant \iopds{} of Nardin
  and Shah~\cite{NS} can equivalently be described as \iopds{} over
  the (2,1)-category $\Span(\xF_{G})$ of spans of finite $G$-sets.
\end{itemize}

\subsubsection*{Fibrous patterns}
The general version of our main result is in the setting of
\emph{algebraic patterns} in the sense of Chu and
Haugseng~\cite{patterns1}, which is a general framework for algebraic
structures described by ``Segal conditions''. More precisely, an
algebraic pattern is an \icat{} $\calO$ equipped with a factorization
system $(\calO^\xint, \calO^\act)$ of ``inert'' and ``active''
morphisms and a full subcategory $\calO^\el \subset \calO^{\xint}$ of
``elementary'' objects. 
This data lets one define \emph{Segal $\mathcal{O}$-objects} in a complete \icat{}
$\mathcal{C}$ as functors $F \colon \mathcal{O} \to \mathcal{C}$ such
that for any object $O \in \mathcal{O}$ the natural map
\[ 
    F(O) \to \lim_{E \in \mathcal{O}^{\el}_{O/}} F(E) 
\] 
is an equivalence, where
$\mathcal{O}^{\el}_{O/} := \mathcal{O}^{\el}
\times_{\mathcal{O}^{\xint}} \mathcal{O}^{\xint}_{O/}$ consists of
inert morphisms from $O$ to elementary objects.
We can then consider a version of \iopds{} where the category $\xF_{*}$
is replaced by an arbitrary algebraic pattern $\mathcal{O}$;
we will refer to them as \emph{fibrous
  $\mathcal{O}$-patterns}\footnote{Under the mild technical assumption
that $\mathcal{O}$ is \emph{sound}, our definition of fibrous
$\mathcal{O}$-patterns agrees with the definition of \emph{weak Segal
  $\mathcal{O}$-fibrations} studied in \cite{patterns1}; see
\cref{propn:fibrous=WSF}.
However, we prove some results beyond this case, and
   here the notion of fibrous $\mathcal{O}$-pattern we introduce is
   better behaved for our purposes.
}.
Such a fibrous $\calO$-pattern can be defined as a functor $\pi \colon
  \mathcal{P} \to \mathcal{O}$ such that:
  \begin{enumerate}
  \item
    $\mathcal{P}$ has all $\pi$-cocartesian lifts of inert morphisms in $\mathcal{O}$.
  \item 
  For all $O \in \mathcal{O}$, the commutative square of \icats{} 
  \[
    \begin{tikzcd}
      \mathcal{P} \times_{\mathcal{O}} \mathcal{O}^{\act}_{/O}
      \arrow{r} \arrow{d} & \lim_{E \in \mathcal{O}^{\el}_{O/}}
      \mathcal{P} \times_{\mathcal{O}} \mathcal{O}^{\act}_{/E}
      \arrow{d} \\
      \mathcal{O}^{\act}_{/O}
      \arrow{r}  & \lim_{E \in \mathcal{O}^{\el}_{O/}}
      \mathcal{O}^{\act}_{/E}
    \end{tikzcd}
  \]
  is cartesian.
  This square is constructed in \cref{defn:fibrous} using the factorization system and the cocartesian lifts from (1).%
  \footnote{
    The bottom horizontal functor is induced by the functors 
    $\alpha_!: \calO^\act_{/O} \to \calO^\act_{/E}$ 
    that are defined for an inert map $\alpha:O \intto{} E$ by sending 
    $\omega:X \to O$ to the active part of the factorization 
    $\alpha \circ \omega: X \intto \alpha_!X \actto E$.
    The top horizontal functor is defined similarly,
    by using the cocartesian lifts for inerts.
  }
\end{enumerate}
The \icat{} $\Fbrs(\mathcal{O})$ of fibrous $\mathcal{O}$-patterns is
then defined as the subcategory of $\CatIsl{\mathcal{O}}$ whose
objects are the fibrous $\mathcal{O}$-patterns and whose morphisms are
required to preserve cocartesian morphisms over inert maps in
$\mathcal{O}$. 

Let us mention a few examples of algebraic patterns where the
corresponding notion of fibrous pattern has already been studied:
\begin{itemize}
    \item If we take $\xF_{*}$ with the classes of inert and active
      maps defined as in \cite{HA} (see \cref{ex:xF*}) and $\angled{1}
      := (\{0,1\},0)$ as the only elementary object, then
      a fibrous $\xF_*$-pattern is 
      a functor $\pi:\calP \to \xF_*$ that has cocartesian lifts for inerts and for which the functor
    \[
      \calP^\act \times_{\xF} \xF_{/\angled{n}} \simeq \calP \times_{\xF_*} (\xF_*)_{/\angled{n}}^\act 
      \longrightarrow \prod_{\langle n \rangle \intto{} \angled{1}} \calP \times_{\xF_*} \xF
      \simeq (\calP^{\rm act})^n,
    \]
    is an equivalence. 
    We will show in \cref{propn:fibrous=WSF} that this is precisely equivalent to $\calP$ being a (symmetric) \iopd{} in the sense of Lurie. 
\item If $\calO \to \xF_*$ is an $\infty$-operad in the sense of Lurie, 
    then it has a canonical pattern structure for which
    a fibrous $\calO$-pattern $\pi\colon \calP \to \calO$
    is simply an $\infty$-operad over $\calO$:
    \[
        \Fbrs(\calO) \simeq \Fbrs(\xF_*)_{/\calO} = (\OpdI)_{/\calO}.
    \]
  \item Let $\xF_*^\natural$ denote the algebraic pattern with
    underlying category $\xF_*$ and the same factorization system as before, but
    with both $\angled{0}$ and $\angled{1}$ as elementary objects.  Then a
    fibrous $\xF_*^\natural$-pattern is a \emph{generalized
      $\infty$-operad} in the sense of \cite{HA}.
  \item If we equip $\Dop$ with the usual inert--active factorization
    system (see \cref{ex:Dop}) and take $[1]$ as the only elementary
    object, then a fibrous $\Dop$-pattern is precisely a planar
    or non-symmetric \iopd{} as in \cite{HA}. If we instead take both
    $[0]$ and $[1]$ as elementary we get \emph{generalized
      non-symmetric \iopds{}} as in \cite{enr}.
\item
    For a finite group $G$, the $G$-$\infty$-operads of \cite{NS} are
    precisely fibrous $\ulF_{G,*}$-patterns for a certain pattern
    $\ulF_{G,*}$ (see \S\ref{subsec:G-operads}).
\end{itemize}

\subsubsection*{Comparing fibrous patterns}
Our first main theorem allows us to compare fibrous patterns over various bases:
\begin{thmA}\label{introthm:comp}
  Let $f \colon \mathcal{O} \to \mathcal{P}$ be a
  morphism of algebraic patterns (\ie{} a functor that preserves
  active and inert morphisms and elementary objects).
  Suppose furthermore that:
  \begin{enumerate}[(i)]
  \item The induced functors $\mathcal{O}^{\el}_{O/} \to
    \mathcal{P}^{\el}_{f(O)/}$ are coinitial for all $O \in
    \mathcal{O}$.
  \item The pattern $\mathcal{P}$ is \emph{sound} in
    the sense of \cref{defn:sound}.
  \item The pattern $\calP$ is \emph{extendable}:
  for all $P \in \mathcal{P}$ the canonical functor
    \[ 
      \mathcal{P}^{\act}_{/P} \to \lim_{E \in \mathcal{P}^{\el}_{P/}}
      \mathcal{P}^{\act}_{/E},
    \]
    is an equivalence.
  \item The restriction $f^{\el}\colon \mathcal{O}^{\el} \to
    \mathcal{P}^{\el}$ of $f$ is an
    equivalence of \icats{},
  \item The functor $(\mathcal{O}^{\act}_{/O})^{\simeq}\to
    (\mathcal{P}^{\act}_{/f(O)})^{\simeq}$ induced by $f$ is an
    equivalence for all $O \in \mathcal{O}$.
  \end{enumerate}
  Then pullback along $f$ gives an equivalence
  \[ f^{*} \colon \Fbrs(\mathcal{P}) \isoto \Fbrs(\mathcal{O}).\]
\end{thmA}
Here the condition of \emph{soundness} is a mild but rather technical assumption, which is satisfied in almost all
examples of algebraic patterns we are aware of.
We can now state the applications of \cref{introthm:comp} that we mentioned above more precisely:

\begin{corA}\label{introcor:Gopd}
  Let  $G$ be a finite group and $\Span(\xF_{G})$ the $(2,1)$-category
  of spans of finite $G$-sets; we regard this as an algebraic pattern where the inert and active
  maps are the backwards and forwards maps, respectively, and the
  elementary objects are the orbits $G/H$ for $H$ a subgroup of $G$.
  There is a functor $\ulF_{G,*} \to \Span(\xF_G)$ such that pullback
  along it gives an equivalence
    \[
        \Fbrs(\Span(\xF_G)) \xrightarrow{\simeq}
        \Fbrs(\ulF_{G*}) = \OpdG. 
    \]
\end{corA}
If we restrict to those fibrous patterns that are also cocartesian fibrations (these are Segal fibrations, or equivalently Segal objects in $\CatI$) then we recover \cite[Theorem 2.3.9]{NS} of Nardin--Shah, which says that the $\infty$-category of $G$-symmetric monoidal $\infty$-categories is equivalent to the $\infty$-category of product-preserving functors $\Span(\xF_G) \to \CatI$.

In the case of the trivial group $G = \{e\}$, \cref{introcor:Gopd} yields an equivalence
   \[ \Fbrs(\Span(\xF)) \isoto \Fbrs(\xF_{*}) = \OpdI\]
between fibrous $\Span(\xF)$-patterns and $\infty$-operads in the sense of Lurie, given by pulling back along the inclusion of $\xF_{*}$ in $\Span(\xF)$ as the wide subcategory containing the spans whose backwards map is injective.

\subsubsection*{Segal envelopes}
The crux of our strategy for proving \cref{introthm:comp} is a reduction to a
comparison between Segal objects in $\CatI$ for the two
patterns. 
For this purpose we need to develop an analogue of
Lurie's \emph{symmetric monoidal envelope} for \iopds{} over a general
algebraic pattern.

A symmetric monoidal \icat{} can be viewed both as a commutative
monoid in $\CatI$ (\ie{} a Segal object for $\xF_{*}$) and as an
\iopd{} that is a cocartesian fibration; we thus have a (non-full)
subcategory inclusion $\CMon(\CatI) \to \OpdI$. In
\cite{HA}*{\S 2.2.4}, Lurie shows that this functor has a left
adjoint, the symmetric monoidal envelope, which admits a very explicit
description as a cocartesian fibration: the envelope of an \iopd{} $\mathcal{O}$ is simply the
fiber product $\mathcal{O} \times_{\xF_{*}} \Aract(\xF_{*})$ where
$\Aract(\xF_{*})$ is the full subcategory of the arrow category of
$\xF_{*}$ on the active morphisms and the fiber product is over the
source functor $\xF^\amalg :=\Aract(\xF_{*}) \to \xF_{*}$, while the projection to
$\xF_{*}$ giving the symmetric monoidal \icat{} is by the target
functor. 
Moreover, it was observed in \cite{iopdprop} that if we
instead regard the envelope as a functor to symmetric monoidal
\icats{} over $(\xF, \amalg)$ 
(that is, finite sets with the disjoint union as symmetric monoidal structure) 
then it is fully faithful.
We want to generalize these results to fibrous $\mathcal{O}$-patterns
for a general algebraic pattern $\mathcal{O}$. 
To simplify exposition we assume here that $\calO$ is both sound and extendable. 
For such $\calO$, unstraightening restricts to give a functor $\Seg_{\calO}(\CatI) \to \Fbrs(\calO)$ 
analogous to the inclusion $\CMon(\CatI) \to \OpdI$.
Our second main result is a description of the left adjoint of this functor.
\begin{thmA}\label{thm:introenv}
  Let $\mathcal{O}$ be a sound and extendable pattern. Then:
  \begin{enumerate}
  \item The unstraightening functor
    $\catname{Seg}_{\calO}(\CatI) \to \Fbrs(\mathcal{O})$ 
    has a left
    adjoint $\Env_{\mathcal{O}}$ whose value on a fibrous $\calO$-pattern $\calP$ 
    is given by the functor $O \longmapsto \calP \times_\calO \calO^\act_{/O}$.
  \item 
  Slicing $\Env_{\calO}$ over $\calA_\calO  := \Env_{\calO}(\calO)$ 
  yields a fully faithful embedding
    \[ \sliceEnv{\mathcal{O}} \colon \Fbrs(\mathcal{O}) \hookrightarrow 
    \Seg_\calO(\CatI)_{/\calA_\calO}\]
    which admits both a left and a right adjoint.
  \item An object $\calC \to \calA_\calO$ in $\Seg_\calO(\CatI)_{/\calA_\calO}$ 
    lies in the essential image of $\sliceEnv{\calO}$ \IFF{} 
    it is \emph{$\Aract(\calO)$-equifibered}, \ie{}
    for every active map $O \actto O'$ in  $\mathcal{O}$, 
    the square
  \begin{equation*}\label{eq:segequifib}
    \begin{tikzcd}
     \calC(O) \arrow{r}{\calC(\omega)} \arrow{d} &
     \calC(O') \arrow{d} \\
      \mathcal{O}^{\act}_{/O} \arrow{r}{\omega_{*}} & \mathcal{O}^{\act}_{/O'}
    \end{tikzcd}
    \end{equation*}
    is cartesian.
  \end{enumerate}
\end{thmA}

In \S\ref{sec:segenv} we actually prove more general (but weaker)
versions of this statement that do not require $\calO$ to be sound or
extendable. 
The comparison of \cref{introthm:comp} can now be shown
by recalling a (simpler) comparison theorem for Segal objects from \cite{ShaulThesis},
passing to slices and then showing that the equivalence 
restricts to the essential image of the envelope.

In \S\ref{subsec:ex-of-envelopes} we spell out \cref{thm:introenv} in several examples. 
In particular, for $\calO = \xF_*$, \cref{thm:introenv} recovers a result of \cite{iopdprop},
though with an alternative characterization of the image:%
\footnote{See \cref{obs:HKcomp} for a comparison.}
\begin{corA}
    The left adjoint to the forgetful functor $\CMon(\CatI) \to \OpdI$
    lifts to a fully faithful functor:
    \[
        \Env\colon \OpdI \hookrightarrow \CMon(\CatI)_{/(\xF, \amalg)}
    \]
    This functor has adjoints on both sides.
    A symmetric monoidal functor $\pi\colon (\calC, \otimes) \to (\xF, \amalg)$ is in the essential image of $\Env$ \IFF{} the square
    \[
    \begin{tikzcd}
      \calC \times \calC \ar[r, "\otimes"] \ar[d, "\pi \times \pi"'] & 
      \calC \ar[d, "\pi"] \\
      \xF \times \xF \ar[r, "\amalg"] &
      \xF
    \end{tikzcd}
    \]
    is a pullback square in $\CatI$.
\end{corA}

In \S\ref{subsec:G-operads} we also give a similar characterisation of the essential image of the
envelope for $G$-$\infty$-operads, though in that case one has to
require additional pullback squares involving the norm maps
$\mathrm{Nm}_K^H\colon \calC^K \to \calC^H$.

\subsubsection*{Organization}
In \S\ref{sec:factenv} we prove a key part of \cref{thm:introenv},
which only depends on the factorization system on an algebraic
pattern:
\begin{thmA}\label{thm:partial-SU-introduction}
    Let $\calB$ be an \icat{} with a factorization system $(\calB_L, \calB_R)$.
    \begin{enumerate}
        \item The forgetful functor $\CatIcoc{\mathcal{B}} \to
          \CatIBL$ has a left adjoint, which takes $\mathcal{E} \to
          \mathcal{B}$ to $\mathcal{E} \times_{\mathcal{B}}
          \ArRB$, where $\ArRB$ is the full subcategory of
          $\Ar(\mathcal{B}) := \Fun([1],\mathcal{B})$ spanned by the
          morphisms in $\mathcal{B}_{R}$, the fiber product is over
          evaluation at $0 \in [1]$, and the projection to
          $\mathcal{B}$ uses evaluation at $1$.
    \item 
          The induced functor $\CatIBL \to
          (\CatIcoc{\mathcal{B}})_{/\ArRB}$ is fully
          faithful, and a morphism $\mathcal{E} \to \ArRB$ in
          $\CatIcoc{\mathcal{B}}$ lies in the image of $\CatIBL$
          \IFF{} it is $\ArRB$-\emph{equifibered}, meaning that
          for every $\phi \colon a \to b$ in $\ArRB$ the commutative square
          \[\begin{tikzcd}
              \mathcal{E}_{a} \arrow{d} \arrow{r}{\phi_{!}} &
              \mathcal{E}_{b} \arrow{d} \\
              (\mathcal{B}_{R})_{/a} \arrow{r}{\phi_{!}} &
              (\mathcal{B}_{R})_{/b}
            \end{tikzcd}\]
          is cartesian.
    \end{enumerate}
\end{thmA}
We emphasize that only the second point here is actually new ---
the first point has already been proved by both
Ayala, Mazel-Gee, and Rozenblyum~\cite{AyalaMazelGeeRozenblyum} and
Shah~\cite{Shah}.

We then review algebraic patterns in \S\ref{sec:patterns}, where we
also introduce the condition of soundness for patterns.
In \S\ref{sec:fbrs} we define fibrous patterns, specialize \cref{thm:partial-SU-introduction}
to this context to prove \cref{thm:introenv}, and explore several examples.
We are then ready to prove \cref{introthm:comp} in \S\ref{sec:comp},
where we also discuss the applications and an $(\infty,2)$-categorical version of \cref{introthm:comp}.

\subsubsection*{Acknowledgments}
\label{sec:acknowledgments}
We thank Manuel Krannich and the anonymous referee for helpful
corrections.  SB: I would like to thank Lior Yanovski for helpful
conversations.  RH: Substantial parts of this paper were written
during visits to the University of Regensburg in December 2021 and the
Centre de Recerca Matem\`atica at the Universitat Aut\`onoma de
Barcelona in May--June 2022, and I thank both for providing a very
pleasant working environment.  JS: I would like to thank Clark Barwick
and Jay Shah for helpful conversations during the Mathematical
Sciences Research Institute program (supported by NSF grant
no. DMS-1928930) at the Universidad Nacional Autónoma de México
(Cuernavaca campus) in June 2022.

\section{Envelopes for factorization systems}\label{sec:factenv}

Our goal in this section is to prove
\cref{thm:partial-SU-introduction}. We begin in \S\ref{sec:cocsubcat} by
explicitly describing the general procedure of freely adding cocartesian morphisms
over a wide subcategory $\mathcal{B}_{0}$ of $\mathcal{B}$ to a functor $p
\colon \mathcal{E} \to \mathcal{B}$, and then in \S\ref{sec:freefibfact} we specialize
this to the situation where $\mathcal{B}_{0}$ is the right class of a
factorization system and $\mathcal{E}$ already has $p$-cocartesian
morphisms over the left class. As already mentioned, these results are
not new, but we include complete proofs to make the paper
more self-contained. In \S\ref{sec:fullfaithsl} we then prove the new part
of \cref{thm:partial-SU-introduction}: we observe that for the induced
adjunction on slices the left adjoint is fully faithful, and identify
its image.

\subsection{Adding cocartesian morphisms over a subcategory}\label{sec:cocsubcat}

Let $\calB$ be an \icat{} equipped with a wide subcategory
$\calB_{0}$, and write $\Cat_{\infty/\calB}^{\calB_0-\mathrm{cocart}}$ for
the subcategory of $\Cat_{\infty/\calB}$ whose objects have
all cocartesian lifts of morphisms in $\calB_{0}$ and whose
morphisms preserve these. The aim of this subsection is to show that the forgetful functor
\[ \Cat_{\infty/\calB}^{\calB_0-\mathrm{cocart}} \to
  \Cat_{\infty/\calB} \]
admits an (explicitly defined) left adjoint. Before explaining the construction of the left
adjoint, let us first fix some notation: We let $\Ar(\calB):=\Fun([1],\calB)$ denote the arrow \icat{} of $\calB$, and write $\Ar_0(\calB)$ for
the full subcategory of $\ArB$ spanned by morphisms in
$\calB_{0}$. The left adjoint of the forgetful functor above is then given by
\[\left(\mathcal{E} \to \calB \right) \longmapsto \left(\mathcal{E} \times_{\calB} \Ar_0(\calB) \to \calB \right).\]
where the fiber product is over $\ev_{0} \colon
\Ar_0(\calB) \to \calB$, and the map $\mathcal{E} \times_{\calB}
\Ar_0(\calB) \to \calB$ is given by $\ev_{1}$. We will prove this by showing that for
any $\mathcal{E} \in \Cat_{\infty/\calB}$ and $\mathcal{F} \in
\Cat_{\infty/\calB}^{\calB_0-\mathrm{cocart}}$, 
restriction yields a natural equivalence:
\[ \Fun^{\calB_0-\mathrm{cocart}}_{/\calB}(\mathcal{E}
  \times_{\calB} \Ar_0(\calB), \mathcal{F}) \isoto
  \Fun_{/\calB}(\mathcal{E}, \mathcal{F}),\]
where the left-hand side consists of functors that preserve
cocartesian morphisms over $\calB_{0}$. This result is by
no means new, and has already appeared in
\cite{AyalaMazelGeeRozenblyum} and \cite{Shah}, but we include a
proof for completeness, as this is the key input needed for our work
in this paper.

\begin{notation}
  Since $\calB_0$ is a wide subcategory, the degeneracy map
  $s_{0}^{*} \colon \calB\to \Ar(\mathcal{B})$ restricts to a
  functor $i \colon \calB\to \Ar_0(\calB)$, taking an object of
  $\calB$ to its identity map. We also have evaluation maps
  $\ev_{0},\ev_{1}\colon \Ar_0(\calB) \to \calB$, and natural
  transformations
  $\sigma \colon i\circ \ev_{0} \to \id$ and $\tau \colon \id \to i \circ \ev_{1}$, given for an
  object $x \xto{\phi} y$ by the squares
  \[
    \begin{tikzcd}
      x \arrow[equals]{r} \arrow[equals]{d} & x \arrow{d}{\phi} \\
      x \arrow{r}{\phi} & y,
    \end{tikzcd}
    \qquad \qquad
    \begin{tikzcd}
      x \arrow{r}{\phi} \arrow{d}{\phi} & y \arrow[equals]{d} \\
      y \arrow[equals]{r} & y,
    \end{tikzcd}
  \]
  respectively. For any functor $p \colon \mathcal{E} \to
  \calB$, the functor $i$ induces a section $i_{\mathcal{E}}
  \colon \mathcal{E} \to \mathcal{E} \times_{\calB}\Ar_0(\calB)$
  of the projection $\pr_{\mathcal{E}} \colon \mathcal{E}
  \times_{\calB}\Ar_0(\calB) \to \mathcal{E}$, and $\sigma$
  induces a natural transformation $\sigma_{\mathcal{E}}\colon
  i_{\mathcal{E}}\pr_{\mathcal{E}} \to \id$.
\end{notation}

\begin{observation}
  Suppose $p \colon \mathcal{E} \to \calB$ is cocartesian over
  $\calB_0$. Then $i_{\mathcal{E}} \colon \mathcal{E} \to
  \mathcal{E} \times_{\calB} \Ar_0(\calB)$ has a left adjoint $\pi_{\mathcal{E}}$:
  Such an adjoint exists \IFF{}, given an object $(x, \phi \colon px
  \to b)$, there is an initial object in
  the \icat{}
  \[ \mathcal{E}_{(x,\phi)/} := \mathcal{E} \times_{\mathcal{E} \times_{\calB} \Ar_0(\calB)} \left(\mathcal{E} \times_{\calB} \Ar_0(\calB)\right)_{(x,\phi)/}
   \simeq \mathcal{E}_{x/}
    \times_{\calB_{px/}} \calB_{b/} \] with the functor
  $\calB_{b/} \to \calB_{px/}$ given by composition with $\phi$.
  A cocartesian morphism $x \to \phi_{!}x$ is precisely an
  initial object in the right-hand side that maps to the identity in
  $\calB_{b/}$. Thus $\pi_{\mathcal{E}}$ takes
  $(x, \phi \colon px \to b)$ to the target $\phi_{!}x$ of the
  cocartesian morphism over $\phi$. Note that we have
  $\pi_{\mathcal{E}} i_{\mathcal{E}} \simeq \id$, and the unit
  transformation $\id \to i_{\mathcal{E}} \pi_{\mathcal{E}}$ is given
  at $(x, \phi)$ by
  \[
    \begin{tikzcd}
      x \arrow{d} &  px \arrow{r}{\phi} \arrow{d}{\phi} & b
      \arrow[equals]{d} \\
      \phi_{!}x, &  b \arrow[equals]{r} & b.
    \end{tikzcd}
  \]
  Moreover, this is an adjunction over $\calB$ in the sense that we have commutative squares
  \[
    \begin{tikzcd}
        \calE \ar[r, hook, "\iota_\calE"] \ar[d, "p"] & \calE \times_{\calB} \Ar_0(\calB) \ar[d, "\ev_1"] \\
        \calB \ar[r, equal] & \calB
    \end{tikzcd}
    \qquad \text{and} \qquad
    \begin{tikzcd}
        \calE \times_\calB \Ar_0(\calB) \ar[d, "\ev_1"] \ar[r, "\pi_\calE"] & \calE \ar[d, "p"] \\
        \calB \ar[r, equal] & \calB.
    \end{tikzcd}
  \]
  For the left square this holds by construction and for the right square we have a transformation $p \circ \pi_\calE = \ev_1 \circ \iota_\calE \circ \pi_\calE \to \ev_1$ coming from the counit of $\pi_\calE \dashv \iota_\calE$.
  This is an equivalence by the point-wise description of $\pi_\calE$ given above.
\end{observation}

\begin{observation}
  Given $p \colon \mathcal{E} \to \calB$, observe that
  $\mathcal{E} \times_{\calB} \Ar_0(\calB)$ is cocartesian over
  $\calB_0$, with cocartesian morphisms given by composition
  in $\Ar_0(\calB)$. (For instance, we can write $\mathcal{E}
  \times_{\calB} \Ar_0(\calB)$ as a pullback $(\mathcal{E} \times
  \calB) \times_{(\calB\times \calB)} \Ar_0(\calB)$
  over $\calB$, where all three \icats{} appearing are
  cocartesian over $\calB_0$.)   
\end{observation}

\begin{propn}\label{propn:freefibdesc}
  If $q \colon \mathcal{F} \to
  \calB$ is cocartesian over $\calB_0$, composition with
  $i_{\mathcal{E}}$ gives a functor
  \[ \Fun^{\calB_0-\mathrm{cocart}}_{/\calB}(\mathcal{E}
    \times_{\calB}\Ar_0(\calB), \mathcal{F}) \to
    \Fun_{/\calB}(\mathcal{E}, \mathcal{F}).\]
  This is an equivalence, with inverse given by taking $F \colon
  \mathcal{E} \to \mathcal{F}$ to the composite
  \[ \mathcal{E} \times_{\calB} \Ar_0(\calB) \xto{F
      \times_{\calB} \Ar_0(\calB)} \mathcal{F}
    \times_{\calB} \Ar_0(\calB) \xto{\pi_{\mathcal{F}}} \mathcal{F}.\]
\end{propn}
\begin{proof}
  Given $G \colon \mathcal{E} \to \mathcal{F}$, the
  definition of the sections $i_{\mathcal{E}}$ and $i_{\mathcal{F}}$
  give
  \[(G \times_{\calB} \Ar_0(\calB)) \circ
    i_{\mathcal{E}} \simeq i_{\mathcal{F}}\circ G, \]
  and so we have
  \[ \pi_{\mathcal{F}} \circ (G \times_{\calB} \Ar_0(\calB)) \circ
    i_{\mathcal{E}} \simeq \pi_{\mathcal{F}}\circ
    i_{\mathcal{F}}\circ G \simeq G. \]
  In the other direction, given $F \colon \mathcal{E}
  \times_{\calB} \Ar_0(\calB) \to \mathcal{F}$ that preserves
  cocartesian morphisms over $\calB_0$, we have to show that
  $F$ is naturally equivalent to $\pi_{\mathcal{F}} \circ
  (Fi_{\mathcal{E}} \times_{\calB} \Ar_0(\calB))$. Here we can
  write $\pr_{\mathcal{F}} \circ (Fi_{\mathcal{E}} \times_{\calB} \Ar_0(\calB))$ as the
  composite
  \[ \mathcal{E} \times_{\calB} \Ar_0(\calB)
    \xto{\pr_{\mathcal{E}}} \mathcal{E} \xto{i_{\mathcal{E}}}
    \mathcal{E} \times_{\calB} \Ar_0(\calB) \xto{F}
    \mathcal{F},\]
  so that $\sigma_{\mathcal{E}}$ induces a natural transformation
  \[ \alpha \colon \pr_{\mathcal{F}} \circ (Fi_{\mathcal{E}} \times_{\calB}
    \Ar_0(\calB))  \to F. \]
  Note that this is given at $(e, \phi\colon p(e) \to b)$ by the image
  $F(e,\id_{p(e)}) \to F(e,\phi)$ of a cocartesian morphism in
  $\mathcal{E}\times_{\calB}\Ar_0(\calB)$, and so is cocartesian
  in $\mathcal{F}$ since by assumption $F$ preserves cocartesian
  morphisms over $\calB_0$. Projecting to $\calB$, we
  see that $q\alpha$ factors as the projection to $\Ar_0(\calB)$
  followed by the evaluation map $\Ar_0(\calB) \times [1] \to \calB$.
  We can therefore define a natural transformation
  \[ \beta \colon \mathcal{E} \times_{\calB} \Ar_0(\calB) \times [1] \to
    \mathcal{F} \times_{\calB} \Ar_0(\calB)\]
  via the commutative diagram
  \[
    \begin{tikzcd}
    \mathcal{E} \times_{\calB} \Ar_0(\calB) \times [1] \arrow{r}{\alpha} \arrow{d} &
    \mathcal{F} \arrow{dd}{q} \\
    \Ar_0(\calB) \times [1] \arrow{d}{\tau} \arrow{dr}{\ev} \\
    \Ar_0(\calB) \arrow{r}{s} & \calB.
    \end{tikzcd}
  \]
  Here $\beta$ is a natural transformation $(Fi_{\mathcal{E}} \times_{\calB}
    \Ar_0(\calB))  \to i_{\mathcal{F}}F$, and
  takes $(e, \phi \colon p(e) \to b)$ to
  $(F(e,\id_{p(e)}), \phi) \to (F(e, \phi), \id_{c})$. Composing with
  $\pi_{\mathcal{F}}$ we get a natural transformation
  $\pi_{\mathcal{F}}\beta \colon \pi_{\mathcal{F}} \circ (Fi_{\mathcal{E}} \times_{\calB}
  \Ar_0(\calB))  \to \pi_{\mathcal{F}}i_{\mathcal{F}}F \simeq
  F$. This is given at $(e, \phi)$ by the canonical morphism
  $\phi_{!}F(e,\id) \to F(e,\phi)$. Since $F$ preserves cocartesian
  morphisms over $\calB_0$, this is an equivalence, and so we
  have obtained the natural equivalence we required.
\end{proof}

\begin{cor}\label{cor:freecocadj}
  The forgetful functor
  \[ \CatIzcoc{\mathcal{B}} \to
    \CatIsl{\mathcal{B}}\]
  has a left adjoint given by
  \[ (\mathcal{E} \to \calB) \quad\mapsto\quad \mathcal{E}
    \times_{\calB} \Ar_0(\calB) = s^{*}\mathcal{E} \to \Ar_0(\calB)
    \xto{t} \calB,\]
  and unit given by $i_{\mathcal{E}} \colon \mathcal{E} \to \mathcal{E}
  \times_{\calB} \Ar_0(\calB).$
\end{cor}
\begin{proof}
  By \cref{propn:freefibdesc}, for
   $\mathcal{E} \in \Cat_{\infty/\calB}$
  and $\mathcal{F} \in \Cat_{\infty/\calB}^{\calB_0-\mathrm{cocart}}$
  the composite 
  \[ \Map_{\Cat_{\infty/\calB}^{\calB_0-\mathrm{cocart}}}(\mathcal{E} \times_{\calB}
    \Ar_0(\calB), \mathcal{F}) \to
    \Map_{\Cat_{\infty/\calB}}(\mathcal{E} \times_{\calB}
    \Ar_0(\calB), \mathcal{F}) \xto{i_{\mathcal{E}}^{*}}
    \Map_{\Cat_{\infty/\calB}}(\mathcal{E}, \mathcal{F})
  \]
  is an equivalence, hence this natural transformation is indeed the
  unit of an adjunction.
\end{proof}

\begin{observation}
  The forgetful functors
  $\CatIzcoc{\mathcal{B}} \to
  \CatIsl{\mathcal{B}} \to \CatI$ detect pullbacks; in particular, the \icat{}
  $\CatIzcoc{\mathcal{B}}$ has all pullbacks.
  Indeed, given morphisms $\mathcal{E}_{1} \to \mathcal{E}_{0} \from
  \mathcal{E}_{2}$ in $\CatIzcoc{\mathcal{B}}$, it is easy to see that
  a morphism in
  the fiber product $\mathcal{E}_{1}\times_{\mathcal{E}_{0}}
  \mathcal{E}_{2}$ is cocartesian over $\mathcal{B}_{0}$ \IFF{} its
  images in $\mathcal{E}_{1}$ and $\mathcal{E}_{2}$ are cocartesian.
\end{observation}

\begin{observation}\label{obs:freesubcatftr}
  Suppose $\mathcal{A}$ and $\mathcal{B}$ are \icats{} equipped with
  wide subcategories $\mathcal{A}_{0}$ and $\mathcal{B}_{0}$,
  respectively, and that $f \colon \mathcal{A} \to \mathcal{B}$ is a
  functor that takes $\mathcal{A}_{0}$ into
  $\mathcal{B}_{0}$. Pullback along $f$ clearly gives a commutative
  diagram
  \[
    \begin{tikzcd}
      \CatIzcoc{\mathcal{B}} \arrow{d}\arrow{r}{f^{*}} &
      \CatIzcoc{\mathcal{A}} \arrow{d} \\
      \CatIsl{\mathcal{B}} \arrow{r}{f^{*}} & \CatIsl{\mathcal{A}}.
    \end{tikzcd}
  \]
  We then have an induced Beck--Chevalley transformation between
  the left adjoints of the vertical maps, given for $p \colon \mathcal{E} \to
  \mathcal{B}$ by the natural map
  \[ (\mathcal{E} \times_{\mathcal{B}} \mathcal{A})
    \times_{\mathcal{A}} \Ar_{0}(\mathcal{A}) \to (\mathcal{E}
    \times_{\mathcal{B}} \Ar_{0}(\mathcal{B})) \times_{\mathcal{B}}
    \mathcal{A},\]
  which takes $(e \in \mathcal{E}, a\in \mathcal{A}, p(e) \simeq f(a),
  \phi \colon a \to a' \in \Ar_{0}(\mathcal{A}))$ to $(e, f(a),
  f(\phi),a)$. Note, however, that this is typically not
  an equivalence.
\end{observation}

\subsection{Free fibrations for factorization systems}\label{sec:freefibfact}
In this subsection we specialize our previous results to the case of
an \icat{} equipped with a factorization system. We again emphasize
that this result already appears in \cite{AyalaMazelGeeRozenblyum} and
\cite{Shah}.

\begin{notation}
  In this section we fix an \icat{} $\mathcal{B}$ with a factorization
  system $(\mathcal{B}_{L}, \mathcal{B}_{R})$; we write $\ArLB$
  and $\ArRB$ for the full subcategories of $\ArB$
  spanned by the morphisms in $\mathcal{B}_{L}$ and $\mathcal{B}_{R}$,
  respectively. We also abbreviate
  \[ \CatIBL := \CatIsl{\mathcal{B}}^{\mathcal{B}_{L}\dcoc}\]
\end{notation}

\begin{propn}[\cite{patterns1}*{Proposition 7.3}]\label{propn:free-fibration-is-cocartesian}
    Let $(q \colon \mathcal{C} \to \mathcal{B}) \in \CatIBL$.
    Then:
    \begin{enumerate}[(1)]
        \item The functor $q'\colon\mathcal{C} \times_{\mathcal{B}}
          \ArRB \to \mathcal{B}$ given by evaluation at the target is a cocartesian fibration.
        \item\label{item:envcoc} A morphism $(\alpha, \beta)\colon (c_0,\phi_0) \to (c_1,\phi_1)$ in $\mathcal{C} \times_{\mathcal{B}} \ArRB$ represented by the following diagram
        \[
        \left(
        \begin{tikzcd}
        c_0 \arrow{d}{\alpha} \\
        c_1
        \end{tikzcd},
        \begin{tikzcd}
        q(c_0) \arrow{r}{\phi_0} \arrow{d}{q(\alpha)}
        & b_0 \arrow{d}{\beta} \\
        q(c_1) \arrow{r}{\phi_1} & b_1
        \end{tikzcd}
        \right),
        \]
        is a $q'$-cocartesian lift of $\beta \colon b_0 \to b_1$ if and only if
        $q(\alpha)$ is in $\mathcal{B}_{L}$ and $\alpha$ is $q$-cocartesian.
    \end{enumerate}
\end{propn}
\begin{proof}
  We first show that $q'$ is a locally cocartesian fibration. A
  locally $q'$-cocartesian morphism over $\beta \colon b_{0} \to b_{1}$ with
  source $(c_{0}, \phi_{0} \colon q(c_{0}) \to b_{0})$ is an
  initial object in the \icat{} $(\mathcal{C} \times_{\mathcal{B}}
  \ArRB)_{(c_{0},\phi_{0})/} \times_{\mathcal{B}_{b_{0}/}}
  \{\beta\}$. We can identify this \icat{} as the fiber product
  \[ \mathcal{C}_{c_{0}/} \times_{\mathcal{B}_{qc_{0}/}}
    \left( \mathcal{B}^{R}_{/b_{1}} \times_{\mathcal{B}_{/b_{1}}}
      (\mathcal{B}_{/b_{1}})_{\beta\phi_{0}/} \right), \]
  where $\mathcal{B}^{R}_{/b_{1}}$ denotes the full subcategory of
  $\mathcal{B}_{/b_{1}}$ spanned by morphisms in
  $\mathcal{B}_{R}$.

  We first observe that here $\mathcal{B}^{R}_{/b_{1}} \times_{\mathcal{B}_{/b_{1}}}
  (\mathcal{B}_{/b_{1}})_{\beta\phi_{0}/}$ has an initial object,
  given by
  \[
    \begin{tikzcd}
      qc_{0} \arrow{rr}{\lambda} \arrow{dr}[swap]{\beta\phi_{0}} & & b'
      \arrow{dl}{\rho} \\
       & b_{1},
    \end{tikzcd}
  \]
  where $(\lambda,\rho)$ is the $(L,R)$-factorization of
  $\beta\phi_{0}$ --- this follows from \cite{HTT}*{Lemma 5.2.8.19}.

  The projection $\mathcal{B}^{R}_{/b_{1}} \times_{\mathcal{B}_{/b_{1}}}
  (\mathcal{B}_{/b_{1}})_{\beta\phi_{0}/} \to
  \mathcal{B}^{R}_{/b_{1}}$ is a left fibration, since it is a base
  change of the left fibration
  $(\mathcal{B}_{/b_{1}})_{\beta\phi_{0}/} \to
  \mathcal{B}_{/b_{1}}$. The initial object of  $\mathcal{B}^{R}_{/b_{1}} \times_{\mathcal{B}_{/b_{1}}}
  (\mathcal{B}_{/b_{1}})_{\beta\phi_{0}/}$, which maps to $\rho$ in
  $\mathcal{B}^{R}_{/b_{1}}$, therefore gives an equivalence
  \[ \mathcal{B}^{R}_{/b_{1}} \times_{\mathcal{B}_{/b_{1}}}
    (\mathcal{B}_{/b_{1}})_{\beta\phi_{0}/} \simeq
    (\mathcal{B}^{R}_{/b_{1}})_{\rho/} \] by \cite{Kerodon}*{Tag 0199}.
  We can therefore rewrite our expression for the \icat{}
  $(\mathcal{C} \times_{\mathcal{B}} \ArRB)_{(c_{0},\phi_{0})/}
  \times_{\mathcal{B}_{b_{0}/}} \{\beta\}$ as
  \[ \left( \mathcal{C}_{c_{0/}} \times_{\mathcal{B}_{qc_{0}/}}
      \mathcal{B}_{b'/}\right) \times_{\mathcal{B}_{b'/}}
    (\mathcal{B}^{R}_{/b_{1}})_{\rho/}.\] A $q$-cocartesian morphism
  over $\lambda$ with source $c_{0}$, which exists by assumption since
  $\lambda$ is in $\mathcal{B}_{L}$, is precisely an initial object of
  $\mathcal{C}_{c_{0/}} \times_{\mathcal{B}_{qc_{0}/}}
  \mathcal{B}_{b'/}$ that maps to the initial object in
  $\mathcal{B}_{b'/}$. We thus have initial objects in 
  $\mathcal{C}_{c_{0/}} \times_{\mathcal{B}_{qc_{0}/}}
  \mathcal{B}_{b'/}$ and $(\mathcal{B}^{R}_{/b_{1}})_{\rho/}$ that
  both map to the initial object in $\mathcal{B}_{b'/}$, and these
  thus give an initial object in the fiber product $(\mathcal{C} \times_{\mathcal{B}}
  \ArRB)_{(c_{0},\phi_{0})/}$. This shows that if $\alpha \colon c_{0}
  \to c_{1}$ is a $q$-cocartesian lift of $\lambda$, then
  \[
    \left(
      \begin{tikzcd}
        c_0 \arrow{d}{\alpha} \\
        c_1
      \end{tikzcd},
      \begin{tikzcd}
        q(c_0) \arrow{r}{\phi_0} \arrow{d}{\lambda}
        & b_0 \arrow{d}{\beta} \\
        b' \arrow{r}{\rho} & b_1
      \end{tikzcd}
    \right)
  \]
  is a locally $q$-cocartesian lift of $\beta$ with source $(c_{0},\phi_{0})$.

  We have thus shown that $q'$ is a locally cocartesian fibration, and
  the locally $q'$-cocartesian morphisms are precisely those in \ref{item:envcoc}.
  To see that $q'$ is a cocartesian fibration it then suffices by
  \cite{HTT}*{Proposition 2.4.2.8} to check
  that the locally $q'$-cocartesian morphisms are closed under
  composition, which in our case is clear.
\end{proof}

\begin{notation}\label{notation:rmE}
      It follows from \cref{propn:free-fibration-is-cocartesian} that the construction $\calE \mapsto \calE \times_\calB \ArRB$ 
 restricts to a well-defined functor
  \[
    \rmE \colon  \CatIBL \to \CatIcoc{\calB}, \quad
    (\calE \to \calB) \mapsto (\calE \times_\calB \ArRB \to \calB).
  \]
\end{notation}

\begin{propn}\label{propn:envfuneq}
   Let $p \colon \mathcal{E} \to \mathcal{B}$ be functor admitting cocartesian lifts for all arrows in $\mathcal{B}_L$ and let $q \colon \mathcal{F} \to \mathcal{B}$
   be a cocartesian fibration. Then the equivalence of
   \cref{propn:freefibdesc} restricts to an equivalence
   \[ \Fun^{\txt{cocart}}_{/\mathcal{B}}(\rmE(\mathcal{E}), \mathcal{F}) \isoto
    \Fun^{L\txt{-cocart}}_{/\mathcal{B}}(\mathcal{E}, \mathcal{F}).\] 
\end{propn}
\begin{proof}
  We must show that these full subcategories are identified under the
  equivalence
   \[ \Fun^{R-\txt{cocart}}_{/\mathcal{B}}(\mathcal{E}
    \times_{\mathcal{B}}\ArRB, \mathcal{F}) \isoto
    \Fun_{/\mathcal{B}}(\mathcal{E}, \mathcal{F})\] 
  of \cref{propn:freefibdesc}.
  Given a functor $F
  \colon \mathcal{E} \times_{\mathcal{B}} \ArRB \to \mathcal{F}$ that
  preserves cocartesian morphisms over $\mathcal{B}_{R}$, we must thus
  check that $F$ preserves all cocartesian morphisms \IFF{} $F \circ
  i_{\mathcal{E}}$ preserves cocartesian 
  morphisms over $\mathcal{B}_{L}$.
  We write $p' \colon \mathcal{E} \times_{\mathcal{C}} \ArRB \to
  \mathcal{B}$ for the map induced by $\ev_{1}$.
  
  First, assume that $F \colon\calE \times_\calB \ArRB \to \calF$
  preserves all cocartesian edges. For a $p$-cocartesian lift
  $\alpha \colon c_{0} \to c_{1}$ of an edge
  $\beta \colon b_0 \to b_1$ in $\mathcal{B}_{L}$, its image under
  $i_\calE$ is the edge
    \[
    \left(
      \begin{tikzcd}
        c_0 \arrow{d}{\alpha} \\
        c_1
      \end{tikzcd},
      \begin{tikzcd}
        b_{0} \arrow{r}{=} \arrow{d}{\beta}
        & b_{0} \arrow{d}{\beta} \\
        b_{1} \arrow{r}{=} & b_1
      \end{tikzcd}
    \right)
  \]
  in $\calE \times_\calB \ArRB$,
  which is $p'$-cocartesian by \cref{propn:free-fibration-is-cocartesian}.
  In other words, $i_\calE \colon \calE \to \calE \times_\calB \ArRB$ 
  preserves cocartesian lifts over $\mathcal{B}_{L}$, and hence so does $F \circ i_\calE$.
    
  For the converse assume that $F$ preserves cocartesian lifts of
  edges in $\mathcal{B}_{R}$ and $F \circ i_\calE$ preserves cocartesian lifts
  of edges in $\mathcal{B}_{L}$.  We would like to show that a general
  $p'$-cocartesian morphism
  $(\alpha, \beta)\colon (c_0,\phi_0) \to (c_1,\phi_1)$
  is sent to a $q$-cocartesian
  morphism in $\calF$.  According to
  \cref{propn:free-fibration-is-cocartesian}, the morphism $p(\alpha)$
  is in $\mathcal{B}_{L}$ and $\alpha$ is $p$-cocartesian.
  We can
  fit this morphism into the following diagram by applying the natural
  transformation $\sigma_\calE\colon i_\calE \pr_\calE \to \id$:
    \[\begin{tikzcd}[column sep = 5pc]
	    {(c_0,\id)} & {(c_0,\phi_0) } \\
	    {(c_1,\id)} & {(c_1,\phi_1)}
	    \arrow["{(\id,\phi_0) = (\sigma_\calE)_{(c_0,\phi_0)}}", from=1-1, to=1-2]
	    \arrow["{(\alpha, q(\alpha))}"', from=1-1, to=2-1]
	    \arrow["{(\id,\phi_1) = (\sigma_\calE)_{(c_1,\phi_1)}}"', from=2-1, to=2-2]
	    \arrow["{(\alpha, \beta)}", from=1-2, to=2-2]
	\end{tikzcd}\] 
	Both horizontal morphisms are cocartesian edges over $\mathcal{B}_{R}$ 
	(by \cref{propn:free-fibration-is-cocartesian})
	and the left-hand vertical morphism is the image under $i_\calE$
	of a $p$-cocartesian morphism over $\mathcal{B}_{L}$.
	Hence $F$ sends three of the morphisms in the above square 
	to cocartesian edges in $\calF$ and it follows by composition
	and right-cancellation for cocartesian edges that 
	$F(\alpha, \beta)$ is cocartesian too.
\end{proof}

\begin{cor}\label{cor:factadjn}
  The adjunction of \cref{cor:freecocadj} restricts to an adjunction
  \[ \rmE\colon \CatIBL \rightleftarrows \CatIcoc{\mathcal{B}} :\! \mathrm{forget}.\]
\end{cor}

\begin{observation}
  Suppose $(\mathcal{A}, \mathcal{A}_{L}, \mathcal{A}_{R})$ and
  $(\mathcal{B}, \mathcal{B}_{L}, \mathcal{B}_{R})$ are \icats{}
  equipped with factorization systems, and that $f \colon \mathcal{A}
  \to \mathcal{B}$ is a functor that preserves both classes of maps in
  these. Pullback along $f$ then gives a commutative
  diagram
  \[
    \begin{tikzcd}
      \CatILcoc{\mathcal{B}} \arrow{d}\arrow{r}{f^{*}} &
      \CatILcoc{\mathcal{A}} \arrow{d} \\
      \CatIcoc{\mathcal{B}} \arrow{r}{f^{*}} & \CatIcoc{\mathcal{A}}.
    \end{tikzcd}
  \]
  As in
  \cref{obs:freesubcatftr}, this induces a Beck--Chevalley
  transformation, but this is typically not an equivalence.
\end{observation}

\subsection{Full faithfulness on slices}\label{sec:fullfaithsl}
In this subsection we prove the main new result of this section: We
observe that the adjunction of \cref{cor:factadjn} induces an
adjunction
\[ \CatIBL \rightleftarrows (\CatIcoc{\mathcal{B}})_{/\ArRB}\]
where the left adjoint is fully faithful, and characterize its image
as in \cref{thm:partial-SU-introduction}.

To construct this adjunction, we recall the general construction of
adjunctions on slices:
\begin{observation}\label{rmk:sliceadj}
  Given an adjunction
  \[ L\colon \mathcal{C} \rightleftarrows \mathcal{D} :\! R \]
  where $\mathcal{C}$ admits pullbacks, we have (by
  \cite{HTT}*{Proposition 5.2.5.1}) for any $c$ in $\mathcal{C}$ an
  induced adjunction
  \[ L_{c} \colon \mathcal{C}_{/c} \rightleftarrows \mathcal{D}_{/Lc} :\!
    R_{c} \]
  where $L_{c}$ is simply given by applying $L$, while $R_{c}$ is
  defined at $f \colon d \to Lc$ by the natural pullback square
  \[
    \begin{tikzcd}
      R_{c}d \arrow{d} \arrow{r} & Rd \arrow{d}{Rf} \\
      c \arrow{r}{\eta_{c}} & RLc
    \end{tikzcd}
  \]
  over the unit map $\eta_{c}$. The unit for the new adjunction is
  then given at $c'\to c$ by the canonical map $c' \to R_{c}L_{c}c'$
  obtained by factoring the square
  \[
    \begin{tikzcd}
      c' \arrow{d} \arrow{r}{\eta_{c'}} & RLc' \arrow{d} \\
      c \arrow{r}{\eta_{c}} & RLc
    \end{tikzcd}
  \]
  through the pullback, while the counit $L_{c}R_{c}d \to d$ is given
  by the outer square in the diagram
  \[
    \begin{tikzcd}
      L R_{c}d \arrow{d} \arrow{r} & LRd \arrow{d} \arrow{r}{\epsilon_{d}} & d \arrow{d} \\
      Lc \arrow[bend right,equals]{rr} \arrow{r}{L\eta_{c}} & LRLc \arrow{r}{\epsilon_{Lc}} & Lc,
    \end{tikzcd}
  \]
  where $\epsilon$ is the counit of the original adjunction.
\end{observation}

\begin{propn}\label{propn:lifting-Env}
    By applying the construction of \cref{rmk:sliceadj} to the adjunction
    of \cref{cor:factadjn} at the terminal object 
    $(\mathcal{B} \xto{=} \mathcal{B}) \in \CatIBL$ we obtain an adjunction
    \begin{equation}
      \label{eq:sliceenvadj}
        \rmE \colon  \CatIBL \rightleftarrows
        \left(\CatIcoc{\mathcal{B}}\right)_{/\ArRB} :\! \rmQ.
    \end{equation}
    The left adjoint in this adjunction is fully faithful.
\end{propn}
\begin{proof}
    Here $\rmE$ sends $\calE \to \calB$ to the cocartesian fibration
    $\calE \times_\calB \ArRB \to \calB$, equipped with the canonical
    projection to $\ArRB \to \calB$.
    The right adjoint $\rmQ$ is given by
    \[
         \mathcal{E} \to \ArRB \quad \longmapsto \quad
        i^{*}\mathcal{E} = \calB \times_{\Ar_R(\calB)} \calE  \to \mathcal{B}
    \]
    where the pullback is taken along the inclusion of the identities 
    $i \colon \mathcal{B} \to \ArRB$.
    The unit of this adjunction is then the map
    $\mathcal{E} \to \rmQ(\rmE(\mathcal{E}))$ obtained from the
    commutative square of units for the adjunction $\rmE \dashv \mathrm{forget}$
    (from \cref{cor:factadjn}) as the canonical map from $\calE$ to the pullback. 
    This square of units is the left hand square in the following commutative diagram:
    \[
      \begin{tikzcd}
        \mathcal{E}\ar[r, "{i_{\mathcal{E}}}", near end] \arrow{d} \arrow[bend left,equals]{rr} &
        \rmE(\mathrm{forget}(\mathcal{E})) \arrow{d} \arrow{r} & 
        \mathcal{E} \arrow{d}\\
        \mathcal{B} \arrow{r}{i} \arrow[bend right,equals]{rr} & \ArRB \arrow{r}{\ev_{0}} & \mathcal{B},
      \end{tikzcd}
    \]
    where the right-hand square is cartesian by construction of 
    $\rmE(\calE)$ in \cref{notation:rmE}.
    Hence the left-hand square is also cartesian and
    thus the unit $\mathcal{E}\to \rmQ(\rmE(\mathcal{E}))$ is an equivalence,
    and so $\rmE$ is indeed fully faithful.
\end{proof}

Now that we have the fully faithful envelope functor
all that is left to do to prove \cref{thm:partial-SU-introduction}
is to characterize its essential image:

\begin{propn}\label{propn:image-of-Env}
    A morphism $\mathcal{D} \to \ArRB$ of cocartesian fibrations over $\mathcal{B}$
    is in the essential image of the left adjoint $\rmE$ from
    \cref{propn:lifting-Env} if and only if it is \emph{equifibered},
    meaning that for every object $\phi \colon a \to b$ in $\ArRB$, the natural square
	\[\begin{tikzcd}
	{\calD_{a}} & {\calD_b} \\
	{\ArRB_{a}} & {\ArRB_b}
	\arrow["{\phi \circ (\blank)}", from=2-1, to=2-2]
	\arrow[from=1-2, to=2-2]
	\arrow[from=1-1, to=2-1]
	\arrow["{\phi_!}", from=1-1, to=1-2]
    \end{tikzcd}\]  
    is cartesian.
\end{propn}
\begin{proof}
  We begin with the ``only if'' direction for
  $(\mathcal{E} \to \mathcal{B}) \in \CatIBL$ and
  $(\phi \colon a \to b) \in \ArRB$. We need to show that the left square
  of the following diagram is cartesian:
	\[\begin{tikzcd}
	{(\mathcal{E} \times_{\mathcal{B}} \ArRB)_{a}} \arrow[bend left]{rr}{\mathrm{pr}_{\mathcal{E}}} & {(\mathcal{E} \times_{\mathcal{B}} \ArRB)_b} & \mathcal{E} \\
	{\ArRB_{a}} \arrow[bend right]{rr}{s} & {\ArRB_b} & \mathcal{B},
	\arrow["{\phi \circ (\blank)}", from=2-1, to=2-2]
	\arrow[from=1-3, to=2-3]
	\arrow[from=1-2, to=2-2]
	\arrow[from=1-1, to=2-1]
	\arrow["{\phi_!}", from=1-1, to=1-2]
	\arrow["{\mathrm{pr}_\mathcal{E}}", from=1-2, to=1-3]
	\arrow["s", from=2-2, to=2-3]
      \end{tikzcd}\]
    where the identification of the composite in the top row uses the
    description of cocartesian morphisms in $(\mathcal{E}
    \times_{\mathcal{B}} \ArRB)$ from \cref{propn:free-fibration-is-cocartesian}.
    This follows since the right-hand square and the outer rectangle are both cartesian.

	For the ``if'' direction we must show that the counit
        $\rmE(\rmQ(\mathcal{D})) \to \mathcal{D}$ is an equivalence if $\mathcal{D}$ is equifibered.
    By \cref{rmk:sliceadj} this counit can be factored as the composite of the
    top horizontal maps in the following diagram:
	\[\begin{tikzcd}
	{\rmE(\rmQ(\mathcal{D}))} & {\rmE(\mathcal{D})} & \mathcal{D} \\
	{\ArRB} & {\rmE(\ArRB)} & \ArRB
	\arrow[from=1-2, to=1-3]
	\arrow[from=2-2, to=2-3]
	\arrow[from=2-1, to=2-2]
	\arrow[from=1-1, to=1-2]
	\arrow[from=1-1, to=2-1]
	\arrow[from=1-3, to=2-3]
	\arrow[from=1-2, to=2-2]
    \end{tikzcd}\label{rectangle-1}\tag{$\star$}\]
    Here the right-hand horizontal maps come from the counit of the adjunction from \cref{cor:factadjn}.
    The bottom horizontal composite is an equivalence, so it will
    suffice to show that the composite rectangle is cartesian.
    Since the left-hand square is given by $\rmE$ applied to the
    cartesian square defining $\rmQ$ (as $\ArR(\mathcal{B})$ is $\rmE(\mathcal{B})$), and $\rmE$ preserves weakly contractible limits, it suffices to show that the right-hand square is cartesian.

    By assumption, the functor $\mathcal{D} \to \mathcal{B}$ is a cocartesian
    fibration, and so the projection $\rmE(\mathcal{D}) \simeq
    \mathcal{D} \times_{\mathcal{B}} \ArRB \to \ArRB$ is
    also a cocartesian fibration, with cocartesian morphisms exactly
    those that project to cocartesian morphisms in $\mathcal{D}$. Consider now the following square
    \[
      \begin{tikzcd}
        \rmE(\mathcal{D}) \arrow{d} \arrow{r} &
        \mathcal{D} \arrow{d}{\pi} \\
        \ArRB \arrow{r}{t} & \mathcal{B},
      \end{tikzcd}
    \]
    in which the top map is
    the counit for the adjunction of \cref{cor:factadjn}. The top map in the square takes cocartesian morphisms over $\ArRB$ to
    $\pi$-cocartesian morphisms in $\mathcal{D}$. To see this, note that a cocartesian
    morphism in $\rmE(\mathcal{D})$ over $\ArRB$ is of the form
    \[\left(
      \begin{tikzcd}
        d \arrow{d} & \pi(d) \arrow{d}{\phi} \arrow{r}{\alpha} & b \arrow{d}{\beta} \\
        \phi_{!}d, &  a \arrow{r}{\gamma} & b'
      \end{tikzcd}
      \right),
      \]
    and this is by construction sent to the canonical map $\alpha_{!}d
    \to \gamma_{!}\phi_{!}d$, which is indeed cocartesian over
    $\beta$.

    Consequently the top right square of (\ref{rectangle-1}) sits as the top face in the following cube
    \[
      \begin{tikzcd}[row sep=tiny,column sep=tiny]
        \rmE(\mathcal{D}) \arrow{dr} \arrow{rr}
        \arrow{dd} & &
        \mathcal{D} \arrow{dd} \arrow{dr} \\
         & \rmE(\ArRB) \arrow[crossing over]{rr} & &\ArRB
         \arrow{dd} \\
         \ArRB \arrow{rr} \arrow[equals]{dr} & & \mathcal{B} \arrow[equals]{dr}
         \\
         & \ArRB \arrow[leftarrow,crossing over]{uu} \arrow{rr} & & \mathcal{B},
      \end{tikzcd}
    \]
    in which the vertical maps are cocartesian fibrations and the maps in
    the top square preserve cocartesian morphisms. Since the bottom
    square is obviously cartesian, to show that the top square is
    cartesian it suffices to check that taking fibers over any $\phi \in \ArRB$ yields a cartesian square.
    We thus want to show that the following square is cartesian:
    \[\begin{tikzcd}
        {\rmE(\mathcal{D})_{\phi}} & {\calD_b} \\
        {\rmE(\ArRB)_{\phi}} & {\ArRB_b}.
        \arrow[from=1-1, to=1-2]
        \arrow[from=1-1, to=2-1]
        \arrow[from=2-1, to=2-2]
        \arrow[from=1-2, to=2-2]
      \end{tikzcd}\]
    Here there is a canonical equivalence $\rmE(\mathcal{D})_{\phi} \simeq (\mathcal{D}
    \times_{\mathcal{B}} \ArRB_{\phi}) \simeq \mathcal{D}_a$ and similarly $\rmE(\ArRB)_{\phi} \simeq (\ArRB)_a$. Via these equivalences the horizontal maps are identified  with the cocartesian pushforward along
    $\phi$. The resulting square is then precisely one of the squares that are
    cartesian by the assumption that $\mathcal{D}$ is equifibered.
\end{proof}

In \S\ref{sec:segenv} it will be notationally convenient to use a
``straightened'' version of the adjunction \cref{eq:sliceenvadj}; to
state this we first introduce some notation:

\begin{notation}\label{defn:partial-s/u}
  Let $\calB$ be an $\infty$-category equipped with a factorization
  system $(\calB_{L},\calB_{R})$, and let $\calR\colon \calB \to
  \CatI$ be the
  straightening of the cocartesian fibration $\Ar_R(\calB) \to \calB$.
  We define the functor
  \[\StBL \colon \CatIBL \to \Fun(\calB,\CatI)_{/\calR}, \]
  which we think of as a form of ``straightening relative to the
  factorization system'', as the composite
  \[ \CatIBL \xto{\rmE} \left(\CatIcoc{\mathcal{B}}\right)_{/\ArRB}
    \xto{\St{\mathcal{B}}{}} \Fun(\calB,\CatI)_{/\calR},\]
   sending $(p\colon\calE \to \calB)$ to the straightening of
   $\calE \times_{\calB} \ArRB \to \calB$. 
   Dually, we define
   \[ \Un{\calB}{L} \colon \Fun(\calB,\CatI)_{/\calR} \to \CatIBL\]
   as the composite
   \[ \Fun(\calB,\CatI)_{/\calR}
     \xto{\Un{\mathcal{B}}{}}\left(\CatIcoc{\mathcal{B}}\right)_{/\ArRB}
     \xto{\rmQ} \CatIBL. \]
   For a functor $F \colon \calB \to \Cat$ together with natural
   transformation $\alpha \colon F \to \calR$ we then have that
   $\Un{\calB}{L}(\alpha)$ is the pullback
    \[
        \begin{tikzcd}
            \Un{\calB}{L}(\alpha) \ar[r] \ar[d] \ar[dr, phantom, very near start, "\lrcorner"] & 
            \Un{\calB}{}(F) \ar[d, "\Un{\calB}{}(\alpha)"] \\
            \calB \ar[r] &
            \ArRB.
        \end{tikzcd}
    \]
\end{notation}

This yields the following reformulation of \cref{thm:partial-SU-introduction}:
\begin{thm}\label{thm:partial-SU}
    The functors $\StBL$ and $\UnBL$ give an adjunction
    \[\StBL \colon \CatIBL \adj \Fun(\calB,\CatI)_{/\calR}  :\! \UnBL.\]
    The functor $\StBL$ is fully faithful and
    a natural transformation $F \to \calR$ is in the essential image 
    of $\StBL$ if and only if it is \emph{equifibered}, meaning that
    for every object $a \xto{\phi} b$ in $\ArRB$, the natural square
    \[\begin{tikzcd}
        {F(a)} & {F(b)} \\
        {\calR(a)} & {\calR(b)}
        \arrow["{F(\phi)}", from=1-1, to=1-2]
        \arrow[from=1-1, to=2-1]
        \arrow["{\calR(\phi)}", from=2-1, to=2-2]
        \arrow[from=1-2, to=2-2]
      \end{tikzcd}\]
    is cartesian.
\end{thm}

A pleasant consequence of \cref{thm:partial-SU} is that $\StBL$
also has a \emph{left} adjoint and that $\CatIBL$ is
presentable. To see this, we use the following observation:

\begin{observation}\label{obs:rightortholoc}
  Let $\mathcal{C}$ be a presentable \icat{}, and $S$ a set of
  morphisms in $\mathcal{C}$. Recall that a morphism $\phi \colon X
  \to Y$ in $\mathcal{C}$ is \emph{right orthogonal} to $S$ if there
  exists a unique filler in every commutative square
  \[
    \begin{tikzcd}
      A \arrow{r} \arrow{d}[swap]{f} & X \arrow{d}{\phi} \\
      B \arrow{r} \arrow[dotted]{ur} & Y
    \end{tikzcd}
  \]
  where $f$ is in $S$. Equivalently, $\phi$ is right orthogonal to $S$
  \IFF{} the commutative square
  \[
    \begin{tikzcd}
      \Map_{\mathcal{C}}(B,X) \arrow{r}{f^{*}} \arrow{d}{\phi_{*}} &
      \Map_{\mathcal{C}}(A,X) \arrow{d}{\phi_{*}} \\
      \Map_{\mathcal{C}}(B,Y) \arrow{r}{f^{*}} & \Map_{\mathcal{C}}(A,Y)
    \end{tikzcd}
  \]
  is cartesian for all $f \colon A \to B$ in $S$. This square is in
  turn cartesian \IFF{} for all maps $B \to Y$, the map on fibers
  \[ \Map_{/Y}(B,X) \to \Map_{/Y}(A, X)\]
  is an equivalence. Thus the map $\phi$ is right orthogonal to $S$
  \IFF{} as an object of $\mathcal{C}_{/Y}$ it is local with respect
  to the set of maps
  \[ \left\{
      \begin{tikzcd}
        A \arrow{rr}{f} \arrow{dr} & & B \arrow{dl} \\
         & Y
       \end{tikzcd}
       : f \in S    
     \right\}.\]
   In particular, the full subcategory of $\mathcal{C}_{/Y}$ spanned
   by the objects that are right orthogonal to $S$ is an accessible
   localization of $\mathcal{C}_{/Y}$, and so is also presentable.
\end{observation}

\begin{propn}\label{lem:Env-admits-left-adjoint}
  Let $(\mathcal{B}, \mathcal{B}_{L}, \mathcal{B}_{R})$ be a small
  \icat{} equipped with a factorization system. The functor
  $\St{\calB}{L}$ has a left adjoint, which exhibits $\CatIBL$ as an
  accessible localization of $\Fun(\calB,\CatI)_{/\calR}$. In
  particular, $\CatIBL$ is a presentable $\infty$-category.
\end{propn}
\begin{proof}
  The \icat{} $\Fun(\calB,\CatI)_{/\calR}$ is clearly presentable, and
  we know that the functor $\StBL$ is fully faithful, with its essential image
    given by functors equifibered over $\calR$. It therefore suffices
    to show that this is the full subcategory of objects in
    $\Fun(\calB,\CatI)_{/\calR}$ that are local with respect to a set
    of morphisms.

    Let $S$ be the collection of morphisms of the form
    \[ (y(\phi) \times \id) \colon y(b) \times [\epsilon]
      \to y(a) \times [\epsilon] \]
    for $\epsilon \in \{0,1\}$ and $(\phi \colon a \to b) \in \ArRB$,
    where $y(a)(\blank) := \Map_\calB(a,\blank)$ is the Yoneda
    embedding of $\calB^{\op}$; this is a set since $\mathcal{B}$ is
    by assumption a small \icat{}.
    An object $\gamma \colon F \to \calR$ in $\Fun(\calB, \CatI)_{/\calR}$ 
    is then equifibered if and only if it is right orthogonal to $S$:
    The latter means that the commutative squares
    \[
      \begin{tikzcd}
      \Map(y(a) \times [\epsilon], F) \arrow{r} \arrow{d} &
      \Map(y(b) \times [\epsilon], F) \arrow{d} \\
      \Map(y(a) \times [\epsilon], \mathcal{R}) \arrow{r} &
      \Map(y(b) \times [\epsilon], \mathcal{R})
      \end{tikzcd}
    \]
    are cartesian; by the Yoneda lemma this square can be identified
    with
    \[
      \begin{tikzcd}
      \Map([\epsilon], F(a)) \arrow{r} \arrow{d} &
      \Map([\epsilon], F(b)) \arrow{d} \\
      \Map([\epsilon], \mathcal{R}(a)) \arrow{r} &
      \Map([\epsilon], \mathcal{R}(b)),
      \end{tikzcd}
    \]
    which is cartesian for $\epsilon=0,1$ \IFF{} the square
    \[
      \begin{tikzcd}
        F(a) \arrow{r} \arrow{d} & F(b) \arrow{d} \\
        \mathcal{R}(a) \arrow{r} & \mathcal{R}(b)
      \end{tikzcd}
    \]
    is cartesian, since the objects $[0],[1]$ generate
    $\CatI$ under colimits. The result then follows from \cref{obs:rightortholoc}.
\end{proof}

\begin{observation}\label{obs:Lcoclim}
  It is easy to see (using the mapping space criterion for cocartesian
  morphisms) that the forgetful functor
  $\CatIBL \to \CatIsl{\mathcal{B}}$ preserves limits and filtered
  colimits. Since both \icats{} are presentable by
  \cref{lem:Env-admits-left-adjoint}, it follows by the adjoint
  functor theorem that this functor has a left adjoint.
\end{observation}

\begin{observation}\label{obs:basechange-envelope}
  Let $(\mathcal{A}, \mathcal{A}_{L}, \mathcal{A}_{R})$ and
  $(\mathcal{B}, \mathcal{B}_{L}, \mathcal{B}_{R})$ be \icats{}
  equipped with factorization systems, and let $f \colon \mathcal{A}
  \to \mathcal{B}$ be a functor that preserves both classes of maps in
  these.

  The functor $f$ then induces a commutative diagram
  \[
    \begin{tikzcd}
      \mathcal{A} \arrow{r}{i_{\mathcal{A}}} \arrow{dd}[swap]{f} &
      \Ar_{R}(\mathcal{A}) \arrow{d}{q} \arrow{r}{\ev_{1}} & \mathcal{A} \arrow[equals]{d}\\
      & f^{*}\ArRB \arrow{d} \arrow{r}{\ev_{1}} \ar[dr, phantom, near start, "\lrcorner"] &
      \mathcal{A} \arrow{d}{f} \\
      \mathcal{B} \arrow{r}{i_{\mathcal{B}}} & \ArRB \arrow{r}{\ev_{1}} & \mathcal{B}.
    \end{tikzcd}
  \]
  From this we get the following commutative diagram of \icats{}:
  \begin{equation}\label{eq:ostarradjsq}
    \begin{tikzcd}
      (\CatIcoc{\mathcal{B}})_{/\ArRB} \arrow{r}{f^{*}} \arrow{d} \arrow[bend
      right=70]{dd}[swap]{\rmQ_{\mathcal{B}}} &
      (\CatIcoc{\mathcal{A}})_{/f^{*}\ArRB} \arrow{r}{q^{*}} \arrow{d} &
      (\CatIcoc{\mathcal{A}})_{/\Ar_{R}(\mathcal{A})} \arrow{d}
      \arrow[bend left=70]{dd}{\rmQ_{\mathcal{A}}}\\
      (\CatILcoc{\mathcal{B}})_{/\ArRB} \arrow{r}{f^{*}} \arrow{d}{\iota_{\calB}^*} &
      (\CatILcoc{\mathcal{A}})_{/f^{*}\ArRB} \arrow{r}{q^{*}}  &
      (\CatILcoc{\mathcal{A}})_{/\Ar_{R}(\mathcal{A})} \arrow{d}{\iota_{\calA}^*} \\
      \CatILcoc{\mathcal{B}} \arrow{rr}{f^{*}} & & \CatILcoc{\mathcal{A}}.
    \end{tikzcd}
  \end{equation}
  Let us write $f^{\ostar}$ 
  for the composite in the top row, which takes
  $\mathcal{E} \to \ArRB$ to the fiber product
  $\mathcal{E} \times_{\ArRB} \Ar_{R}(\mathcal{A}) \to
  \Ar_{R}(\mathcal{A})$.  Passing to vertical left adjoints now yields a Beck--Chevalley
  transformation
  \[ \rmE_{\mathcal{A}} f^{*} \to f^{\ostar}\rmE_{\mathcal{B}};\]
  Unwinding the definitions, this is given at $\mathcal{E} \to
  \mathcal{B}$ in $\CatIBL$  by the natural map
  \[ \left(\mathcal{E} \times_{\mathcal{B}} \mathcal{A}
    \right)\times_{\mathcal{A}} \Ar_{R}(\mathcal{A}) \to \left(\mathcal{E}
      \times_{\mathcal{B}} \ArRB\right) \times_{\ArRB}
    \Ar_{R}(\mathcal{A}). \]
  This is an equivalence, so that we also have a commutative square
  \begin{equation}
    \label{eq:fostarEsq}
    \begin{tikzcd}
      \CatILcoc{\mathcal{B}} \arrow{d}{\rmE_{\mathcal{B}}} \arrow{r}{f^{*}} & \CatILcoc{\mathcal{A}} \arrow{d}{\rmE_{\mathcal{A}}} \\
      (\CatIcoc{\mathcal{B}})_{/\ArRB} \arrow{r}{f^{\ostar}} & (\CatIcoc{\mathcal{A}})_{/\Ar_{R}(\mathcal{A})}.
    \end{tikzcd}
  \end{equation}
\end{observation}

\section{Algebraic patterns}\label{sec:patterns}

In this section we will first review the basic definitions related to
algebraic patterns and Segal objects in
\S\ref{sec:algebraic-patterns}, and then look at some examples thereof
in \S\ref{sec:examples}. We then introduce the condition of
\emph{soundness} for algebraic patterns in \S\ref{sec:sound}; this is
somewhat technical, but turns out to be the key property needed for
some of our results in the next section.

\subsection{Algebraic patterns and Segal objects}
\label{sec:algebraic-patterns}
In this subsection we review the definitions of algebraic patterns and
Segal objects,
and some related basic notions introduced in \cite{patterns1}. We also
introduce a \emph{relative} version of Segal objects, which will show
up later.

\begin{defn}
  An \emph{algebraic pattern} is an \icat{} $\mathcal{O}$ equipped
  with a factorization system, whereby every morphism factors
  (uniquely up to equivalence) as an
  \emph{inert} morphism followed by an \emph{active} morphism, as well
  as a collection of \emph{elementary objects}. We write
  $\mathcal{O}^{\xint}$ and $\mathcal{O}^{\act}$ for the subcategories
  of $\mathcal{O}$ containing only the inert and active morphisms,
  respectively, and $\mathcal{O}^{\el}$ for the full subcategory of
  $\mathcal{O}^{\xint}$ containing elementary objects and inert
  morphisms among them. For $X \in \mathcal{O}$, we also write
  \[ \mathcal{O}^{\el}_{X/} := \mathcal{O}^{\el}
    \times_{\mathcal{O}^{\xint}} \mathcal{O}^{\xint}_{X/} \]
  for the \icat{} of inert maps $X \to E$ with $E \in \mathcal{O}^{\el}$.
\end{defn}

\begin{notation}
  We often indicate inert maps as $X\intto Y$ and active maps as $X
  \actto Y$. These arrows are not meant to indicate any particular
  intuition about inert or active morphisms. 
\end{notation}

\begin{ex}\label{ex:xF*}
  We write $\xF_{*}$ for a skeleton
  of the category of pointed finite sets, with objects $\angled{n} :=
  (\{0,1,\ldots,n\}, 0)$, and say a morphism $\phi \colon \angled{n}
  \to \angled{m}$ is \emph{inert} if $\phi$ restricts to an
  isomorphism $\angled{n}\setminus \phi^{-1}(0) \to \angled{m}
  \setminus \{0\}$, and \emph{active} if $\phi^{-1}(0) = \{0\}$. 
  Then the inert and active morphisms form a factorization system on $\xF_{*}$,
  and we make this an algebraic pattern%
  \footnote{
    In \cite{patterns1} this pattern was denoted $\xF_{*}^{\flat}$ to
    distinguish it from the pattern $\xF_{*}^{\natural}$, where the
    elementary objects are $\angled{0}$ and $\angled{1}$.
    However, in this paper $\xF_* = \xF_*^\flat$ is the key example, so we use a simplified notation for it.
    }
    by taking $\angled{1}$ to be the single elementary object. 
\end{ex}

\begin{ex}\label{ex:Dop}
  Another basic example is $\simp^{\op}$, where $\simp$
  is the simplex category. Recall that $\Dop$ admits an inert-active
  factorization system where inert maps are opposite to interval
  inclusions and active maps are opposite to maps preserving the
  maximal and minimal elements.
  To make $\Dop$ an algebraic pattern, we can take the elementary
  objects to be $[0]$ and $[1]$, in which case we denote the pattern
  by $\simp^{\op,\natural}$, or alternatively just $[1]$, in which
  case the pattern is denoted $\simp^{\op,\flat}$.
\end{ex}

The main reason for introducing algebraic patterns is that they
describe algebraic structures via Segal conditions:
\begin{defn}
  A functor $F \colon \mathcal{O} \to \mathcal{C}$ is a
  \emph{Segal $\mathcal{O}$-object} in the \icat{} $\mathcal{C}$ if for every
  object $X \in \mathcal{O}$ the induced functor
  \[ (\mathcal{O}^{\el}_{X/})^{\triangleleft} \to \mathcal{O} \xto{F}
    \mathcal{C} \]
  is a limit diagram. If $\mathcal{C}$ has limits for diagrams
  indexed by $\mathcal{O}^{\el}_{X/}$ for all $X \in \mathcal{O}$, in
  which case we say that $\mathcal{C}$ is \emph{$\mathcal{O}$-complete},
  then this condition is equivalent to the canonical maps
  \[ F(X) \to \lim_{E \in \mathcal{O}^{\el}_{X/}} F(E)\]
  being equivalences.
  We refer to Segal $\mathcal{O}$-objects in the
  \icat{} $\mathcal{S}$ of spaces as \emph{Segal $\mathcal{O}$-spaces}
  and Segal $\mathcal{O}$-objects in the
  \icat{} $\CatI$ of \icats{} as \emph{Segal $\mathcal{O}$-\icats{}}.
\end{defn}

\begin{ex}\label{ex:Segal-Finstar}
  We can
  identify $(\xF_{*})^{\el}_{\angled{n}/}$
  with the set $\{\rho_{i} : i = 1,\ldots,n\}$, where $\rho_{i}$ is
  the inert morphism $\angled{n} \to \angled{1}$ given by
  \[\rho_{i}(j) =
    \begin{cases}
      0, & j \neq i,\\
      1, & j = i.
    \end{cases}
    \]
  A functor $F \colon \xF_{*} \to \mathcal{C}$ is then a Segal $\xF_{*}$-object if
  for every $n$ the map
  \[ F(\angled{n}) \to \prod_{i=1}^{n} F(\angled{1}),\] induced by the
  maps $\rho_{i}$, is an equivalence.  Thus Segal $\xF_{*}$-objects
  are precisely commutative monoids in the sense of \cite{HA}*{\S
    2.4.2}. For $\calC = \calS$, this gives the $\infty$-categorical
  analogue of \emph{special $\Gamma$-spaces} in the sense of
  Segal~\cite{SegalCatCohlgy}.
\end{ex}

\begin{ex}
    Segal $\simp^{\op,\natural}$-spaces are precisely
    \emph{Segal spaces} in the sense of \cite{RezkCSS}, while
    Segal $\simp^{\op,\flat}$-objects in $\calC$ are associative
    monoids (or $E_1$-algebras). 
\end{ex}

Later on, we will also need to consider a \emph{relative} version of
Segal objects:
\begin{defn}\label{defn:relative-Segal}
  Let $\calO$ be an algebraic pattern and $\calC$ an $\mathcal{O}$-complete
  $\infty$-category. A \emph{relative Segal $\calO$-object} of $\calC$
  is a morphism $\pi \colon Y \to X$ in $\Fun(\calO,\calC)$
  such that for every $O \in \calO$ the natural commutative square
  \[\begin{tikzcd}
      Y(O) & \lim_{E \in \calO^{\el}_{O/}} Y(E) \\
      X(O)  & \lim_{E \in \calO^{\el}_{O/}} X(E)
      \arrow[from=1-1, to=1-2]
      \ar[from=1-1, to=2-1, "\pi(O)"{swap}]
      \arrow[from=2-1, to=2-2]
      \arrow[from=1-2, to=2-2, "\lim_{E \in \calO^{\el}_{O/}} \pi(E)"]
    \end{tikzcd}\] is cartesian. We denote by
  $\Seg^{/X}_{\calO}(\calC) \subseteq \Fun(\calO,\calC)_{/X}$ the full
  subcategory whose objects are the $X$-relative Segal
  $\calO$-objects.
\end{defn}

\begin{observation}\label{rmk:cancellation-for-relative-Segal-objects}
  If $Y \to X$ is a relative Segal $\mathcal{O}$-object of
  $\mathcal{C}$, then the pasting lemma for cartesian squares implies
  that a morphism $Z \to Y$ is a relative Segal
  $\mathcal{O}$-object \IFF{} the composite $Z \to X$ is one.
  Moreover, a morphism $X \to *$ to the terminal object is a relative
  Segal $\mathcal{O}$-object \IFF{} $X$ is a Segal
  $\mathcal{O}$-object in $\mathcal{C}$. Combining these two
  observations, we see that if $X$ is a Segal $\mathcal{O}$-object of
  $\calC$ then an
  $X$-relative Segal $\mathcal{O}$-object is just a Segal
  $\mathcal{O}$-object with a map to $X$, \ie{} we have
  \[ \Seg^{/X}_{\mathcal{O}}(\calC) =
    \Seg_{\mathcal{O}}(\calC)_{/X}\]
  as full subcategories of $\Fun(\mathcal{C}, \mathcal{O})_{/X}$.
\end{observation}

\begin{lemma}\label{rel Segal pullback}
  Suppose $X \to Y$ is a relative Segal $\mathcal{O}$-object in $\mathcal{C}$. Then for
  any map $\eta \colon Y' \to Y$, the pullback $X' := X \times_{Y}Y' \to Y'$ is also a
  relative Segal $\mathcal{O}$-object. In other words, pullback along
  $\eta$ gives a functor $\eta^{*} \colon
  \Seg^{/Y}_{\mathcal{O}}(\mathcal{C}) \to \Seg^{/Y'}_{\mathcal{O}}(\mathcal{C})$.
\end{lemma}
\begin{proof}
  For $O \in \mathcal{O}$, consider the commutative cube
  \[
    \begin{tikzcd}[row sep=tiny,column sep=tiny]
      X'(O) \arrow{rr} \arrow{dd} \arrow{dr} & & \lim_{E \in
        \mathcal{O}^{\el}_{O/}} X'(E) \arrow{dr}  \arrow{dd} \\
       & X(O) \arrow[crossing over]{rr} & & \lim_{E \in
         \mathcal{O}^{\el}_{O/}} X(E)  \arrow{dd} \\
       Y'(O) \arrow{rr} \arrow{dr} & & \lim_{E \in
         \mathcal{O}^{\el}_{O/}} Y'(E) \arrow{dr} \\
        & Y(O) \arrow{rr} \arrow[crossing over, leftarrow]{uu} & & \lim_{E \in
        \mathcal{O}^{\el}_{O/}} Y(E).
    \end{tikzcd}
  \]
  Here the left, right, and front faces are all cartesian, hence so is
  the back face.
\end{proof}

\begin{lemma}\label{lem:relative-Segal-strongly-reflective}
    For every presentable \icat{} $\calC$ the full subcategory 
    \[\Seg^{/X}_{\calO}(\calC) \subseteq \Fun(\calO,\calC)_{/X}\]
    is an accessible localization.
    In particular, it is a presentable \icat{}.
\end{lemma}
\begin{proof}
    Consider the following collection of morphisms in $\Fun(\calO, \calC)$:
    \[
        \left\{ \colim_{E \in \calO^\el_{X/}} y(E) \otimes C \to y(X) \otimes C \right\}_{X \in \calO, C \in K}
    \]
    where $K$ is a set of compact generators for $\calC$, $y$ is the
    Yoneda embedding for $\mathcal{O}^{\op}$, and $T \otimes C$ for $T
    \in \mathcal{S}$, $C \in \mathcal{C}$, is the canonical tensoring
    of $\mathcal{C}$ with $\mathcal{S}$, given by the colimit over $T$
    of the constant diagram with value $C$.
    A morphism $X \to Y$ in $\Fun(\calO, \calC)$ is a relative Segal $\calO$-object
    if and only if it is right orthogonal to this set of morphisms, hence the claim 
    follows from \cref{obs:rightortholoc}.
\end{proof}

Next, we take a brief look at morphisms between patterns:
\begin{defn}
  If $\mathcal{O}$ and $\mathcal{P}$ are algebraic patterns, a
  \emph{morphism of algebraic patterns} is a functor $f \colon
  \mathcal{O} \to \mathcal{P}$ that preserves inert and active
  morphisms as well as elementary objects. We say that such a morphism is a
  \emph{Segal morphism} if for every Segal $\mathcal{P}$-space $F$ and
  every $X \in \mathcal{O}$ the functor $f^{\el}_{X/}\colon \mathcal{O}^{\el}_{X/} \to
  \mathcal{P}^{\el}_{f(X)/}$ arising from $f$ induces an equivalence
  \[ \lim_{\mathcal{P}^{\el}_{f(X)/}} F \isoto
    \lim_{\mathcal{O}^{\el}_{X/}} F \circ f; \]
  by \cite{patterns1}*{Lemma 4.5} this is equivalent to composition
  with $f$ giving a functor \[f^{*} \colon \Seg_{\mathcal{P}}(\mathcal{C}) \to
  \Seg_{\mathcal{O}}(\mathcal{C})\] for any $\mathcal{O}$-complete \icat{} $\mathcal{C}$. The
  Segal morphisms that occur in practice are those where the functor
  $f^{\el}_{X/}$ is coinitial for all $X \in\mathcal{O}$; if this is
  the case we say that $f$ is a \emph{strong Segal morphism}. In the
  special case where $f^{\el}_{X/}$ is an \emph{equivalence} for every
  $X$, we say that $f$ is an \emph{iso-Segal morphism}.
\end{defn}

\begin{ex}\label{ex:Dop->F-isoSegal}
  There is a morphism of algebraic patterns $\mathfrak{c} \colon \simp^{\op,\flat}  \to \xF_{*}$, given on
  objects by $\mathfrak{c}([n]) = \angled{n}$, and with
  $\mathfrak{c}(\phi) \colon \angled{n} \to \angled{m}$ for a morphism
  $\phi \colon [m] \to [n]$ in $\simp$ given by
  \[ \mathfrak{c}(\phi)(i) =
    \begin{cases}
      j, & \text{ if } \phi(j-1) < i \leq \phi(j),\\
      0, & \text{otherwise}.
    \end{cases}
  \]
  It is straightforward to check that this is an iso-Segal morphism.
\end{ex}

\begin{notation}
  We write $\AlgPatt$ for the \icat{} of algebraic patterns together
  with all morphisms of algebraic patterns.
\end{notation}

\begin{observation}\label{rmk:strong Seg pres rel}
  Composition with a \emph{strong} Segal morphism $f \colon \mathcal{O} \to \mathcal{P}$ also preserves \emph{relative} Segal
  objects: If $X \to Y$ is a relative Segal $\mathcal{P}$-object in
  $\mathcal{C}$, then for $O \in \mathcal{O}$ we have a commutative
  diagram
  \[
    \begin{tikzcd}
      X(f(O)) \arrow{r} \arrow{d} & 
      \lim_{E \in \mathcal{P}^{\el}_{f(O)/}} X(E) \arrow{d} \arrow{r}{\sim}
      & \lim_{E' \in \mathcal{O}^{\el}_{O/}} X(f(E')) \arrow{d} \\
      Y(f(O))\arrow{r} & 
      \lim_{E \in \mathcal{P}^{\el}_{f(O)/}} Y(E) \arrow{r}{\sim}
      & \lim_{E' \in \mathcal{O}^{\el}_{O/}} Y(f(E'));
    \end{tikzcd}
  \]
  here both the left and right squares are cartesian, and hence so is the
  outer composite square.
  Composition with $f$ thus gives a functor
  $f^{*} \colon \Seg^{/Y}_{\mathcal{P}}(\mathcal{C}) \to \Seg^{/f^{*}Y}_{\mathcal{O}}(\mathcal{C})$.
\end{observation}

We now recall a simple criterion for a Segal morphism to give an equivalence
on Segal objects: 
\begin{propn}[\cite{ShaulThesis}*{Corollary 2.64}]\label{propn:Segmndcomp}
  Suppose $\mathcal{O}$ and $\mathcal{P}$ are algebraic patterns,
  and $f \colon \mathcal{O} \to \mathcal{P}$ is a strong Segal morphism
  such that
  \begin{enumerate}[(1)]
  \item\label{item:eleq} $f^{\el} \colon \mathcal{O}^{\el} \to \mathcal{P}^{\el}$ is an
    equivalence of \icats{},
  \item\label{item:acteq1} for every $O \in \mathcal{O}$, the functor
    $(\calO_{/O}^\act)^\simeq \to (\calP_{/f(O)}^\act)^\simeq$ 
    is an equivalence of \igpds{}.
  \end{enumerate}
  Then for any complete \icat{} $\mathcal{C}$ the
  functor
  $f^{*}\colon \Seg_{\mathcal{P}}(\mathcal{C}) \to
  \Seg_{\mathcal{O}}(\mathcal{C})$ is an equivalence, with inverse
  given by right Kan extension along $f$.
\end{propn}
\begin{proof}
  We refer to \cite[\S 2]{ShaulThesis} for a detailed proof, but since this result will play an important role in this paper we recall the key steps for the reader's convenience.

  By \cite{patterns1}*{Proposition 6.3}, condition \ref{item:acteq1} implies that right Kan extension along $f$ restricts to Segal objects, giving an adjunction
    \[ f^{*} : \Seg_{\mathcal{P}}(\mathcal{C}) \rightleftarrows \Seg_{\mathcal{O}}(\mathcal{C}) : f_{*}.\]
    Moreover, the proof of \cite{patterns1}*{Proposition 6.3} shows that for $F \in \Seg_{\mathcal{O}}(\mathcal{C})$ we have $(f_{*}F)|_{\mathcal{P}^{\el}} \simeq f^{\el}_{*}(F|_{\mathcal{O}^{\el}})$.

    Condition \ref{item:eleq} therefore implies that the counit $f^{*}f_{*}F \to F$ is an equivalence for $F \in \Seg_{\mathcal{P}}(\mathcal{C})$, since it is an equivalence when evaluated on $\mathcal{P}^{\el}$.    Moreover, \ref{item:eleq} implies that $f^{*}$ is conservative on Segal objects, again since equivalences are detected on elementary objects. To see that the unit map $G \to f_{*}f^{*}G$ is an equivalence, it then suffices to check this after applying $f^{*}$, but then the adjunction $f^{*} \dashv f_{*}$ implies that the composite
    \[ f^{*}G \to f^{*}f_{*}f^{*}G \isoto f^{*}G\]
    is an equivalence, and here we already saw that the counit is an equivalence. Since both the unit and counit of the adjunction are natural equivalences, it must be an equivalence of \icats{}.
\end{proof}

\begin{remark}
  If $(\calO_{/O}^\act)^\simeq$ is a Segal $\mathcal{O}$-space and
  $(\calP_{/f(O)}^\act)^\simeq$ is a
  Segal $\mathcal{P}$-space in \cref{propn:Segmndcomp}, then it
  suffices to check condition \ref{item:acteq1} for elementary objects
  in $\mathcal{O}$. This holds, for instance, if $\mathcal{O}$ and
  $\mathcal{P}$ are extendable (see \cref{defn:soundly-extendable}).
\end{remark}

\subsection{Examples of algebraic patterns}
\label{sec:examples}

We now look at some examples of algebraic patterns. Our focus here
will be on examples that will be relevant in the next sections; we
refer the reader to \cite{patterns1}*{\S 3} for many other
examples.

\begin{ex}
  We have patterns $\simp^{n,\op,\natural}$ and $\simp^{n,\op,\flat}$
  with underlying category
  $\simp^{n,\op} := (\simp^{\op})^{\times n}$, equipped with the
  factorization system where the inert and active maps are those that
  are inert or active in $\Dop$ in each component. Here
  $(\simp^{n,\op,\flat})^{\el} = \{([1],\ldots,[1])\}$ while
  $(\simp^{n,\op,\natural})^{\el}$ consists of all objects whose
  components are all either $[0]$ or $[1]$.
  Then Segal $\simp^{n,\op,\natural}$-spaces are $n$-uple Segal
  spaces, which model $n$-fold \icats{}, while Segal
  $\simp^{n,\op,\flat}$-objects are $\mathbb{E}_{n}$-algebras (by the
  Dunn--Lurie additivity theorem).
\end{ex}

\begin{ex}
  Let $\bbTheta_{n}$ be the inductively defined wreath product $\simp
  \wr \bbTheta_{n-1}$, starting with $\bbTheta_{0} = [0]$; see for
  example \cites{BergerWreath,thetan} for more details. This has a
  factorization system where the active/inert maps are those whose
  components in $\simp$ and $\bbTheta_{n-1}$ are both active or
  inert. There are two interesting pattern structures on
  $\bbTheta_{n}^{\op}$: if we define the objects $C_{i}$ in
  $\bbTheta_{n}$ by $C_{0} := [0]()$ and $C_{i} := [1](C_{i-1})$ for
  $i = 1,\ldots,n$, then for $\bbTheta_{n}^{\op,\flat}$ we take
  $C_{n}$ to be the only elementary object, while for
  $\bbTheta_{n}^{\op,\natural}$ we take all of $C_{0},\ldots,C_{n}$.
  Then Segal $\bbTheta_{n}^{\op,\natural}$-spaces are Rezk's model for
  $(\infty,n)$-categories \cite{RezkThetaN}, while  Segal
  $\bbTheta_{n}^{\op,\flat}$-object are again
  $\mathbb{E}_{n}$-algebras (see \cite{BarwickOpCat}).
\end{ex}

\begin{ex}\label{arity-example}
  Let $\xF_{\ast}^{\le k} \subseteq \xF_{\ast}$ denote the full
  subcategory containing pointed finite sets of cardinality $\le k$
  (excluding the basepoint). Consider $\xF_{\ast}^{\le k}$ as an
  algebraic pattern by restricting the inert-active factorization
  system on $\xF_{\ast}$ and choosing $\langle 1 \rangle $ to be the only
  elementary object.  Segal objects for $\xF_{\ast}^{\le k}$ are
  \emph{arity $k$-restricted commutative monoids} --- a variant of
  commutative monoids in which the homotopy coherence data is only
  supplied up to arity $k$. More generally, if $\calO$ is an
  $\infty$-operad then
  $\calO^{\le k} :=\xF_{\ast}^{\le k} \times_{\xF_{\ast}}\calO$ has a
  natural structure of an algebraic pattern whose Segal objects are
  arity $k$-restricted $\calO$-monoids. For more details see
  \cite{ShaulThesis}.
\end{ex}

The remaining examples we want to discuss are all instances of a
general class of algebraic patterns on \icats{} of spans. For this
purpose we briefly recall the construction of such \icats{} --- this
is originally due to Barwick~\cite{BarwickMackey}; see also
\cite{paradj2} for a more ``model-independent'' version.
\begin{construction}\label{constr:spancats}
  Let $\frX$ be an \icat{} equipped with a pair of
  wide subcategories $\frXb$ and $\frXf$ (where
  ``$b$'' stands for \emph{backwards} and ``$f$'' stands for
  \textit{forwards}. Following Barwick, we say that
  the triple $(\frX,\frXb,\frXf)$ is \emph{adequate} if for every
  pair of morphisms $\beta \colon x \to y$ in $\frXb$ and $\phi \colon
  y' \to y$ in
  $\frXf$,  we have:
  \begin{enumerate}
  \item the pullback $x' := x \times_{y} y'$ exists in $\frX$,
  \item the projection $x' \to y'$ lies in $\frXb$.
  \item the projection $x' \to x$ lies in $\frXf$.
  \end{enumerate}
  Given an adequate triple $(\frX,\frXb,\frXf)$ 
  Barwick defines an $\infty$-category $\Spanbf(\frX)$
  (denoted $\mathrm{A}^{\rm eff}(\frX,\frXb,\frXf)$ in \cite{BarwickMackey})
  such that the objects of $\Spanbf(\frX)$ are the objects of $\frX$
  and the morphisms from $x$ to $y$ are spans (or correspondences)
  \[
    \begin{tikzcd}
      {} &  w \arrow[dl, "\beta"{swap}] \arrow{dr}{\phi} \\
      x & & y
    \end{tikzcd}
  \]
  where the arrow $\beta$ lies in $\frXb$ and the arrow $\phi$ lies
  in $\frXf$.  The assumption that the triple is adequate allows for
  a composition law defined by taking pullbacks. If $\frX$ is an
  \icat{} with pullbacks, then we can take $\frXb = \frXf = \frX$,
  in which case we just write $\Span(\frX)$ for the corresponding
  \icat{} of spans. 
\end{construction}

\begin{observation}
  By the first part of \cite{paradj2}*{Proposition 4.9} the \icat{}
  $\Spanbf(\frX)$ always has a factorization system given by the
  classes of maps as above with $\phi$ or $\beta$ required to be an
  equivalence (which we might call the ``backwards'' and
  ``forwards'' maps) and the subcategories of these maps are
  equivalent to $\frX^{b,\op}$ and $\frXf$, respectively.
\end{observation}

\begin{defn}\label{defn:span}
  Given an adequate triple $(\frX,\frXb,\frXf)$ and a full
  subcategory $\frX_0 \subseteq \frX$, we denote by
  $\Spanbf(\frX; \frX_0)$ the algebraic pattern given by
  $\Spanbf(\frX)$ with the factorization system whose 
  inert
  and active maps are the backwards and forward maps, respectively,
  and with the objects of $\frX_{0}$ as the elementary objects.
\end{defn}

\begin{remark}\label{rem:Segal-condition-span-patterns}
  The Segal condition for $\Spanbf(\frX;\frX_0)$ takes the
  following form for a functor $F$:
  \[
    F(x) \simeq \lim_{e \rightarrow x \in (\frXb_{0/x})^{\op}} F(e),
  \]
  where $\Spanbf(\frX)^{\el}_{x/} \simeq (\frXb_{0/x})^{\op}$ with
  $\frXb_{0/x} := \frXb_{0}  \times_{\frXb} \frXb_{/x}$ and
  $\frXb_{0}$ is the full subcategory of $\frXb$ containing the
  objects of $\frX_{0}$.
\end{remark}

\begin{ex}
  Let $\xF$ denote the category of finite sets. Since this has
  pullbacks, \cref{constr:spancats} produces an \icat{} (in fact a
  (2,1)-category) $\SpF$ whose objects are finite sets, and whose
    morphisms are spans of the form
        \[
    \begin{tikzcd}
      {} & \fset{m} \arrow["\alpha"']{dl} \arrow["\beta"]{dr} \\
      \fset{n} & & \fset{n}'
    \end{tikzcd}
  \]
  for finite sets $\fset{n}$, $\fset{m}$, and $\fset{n}'$, with
  composition given by taking pullbacks.
  We consider $\SpF = \Span(\xF; \{\bfone\})$ as an algebraic pattern by
  taking the backward maps as inerts,  
  forward maps as actives and 
  $\bfone \in \SpF$ as the only elementary object.
\end{ex}

\begin{observation}\label{rmk:F*asspan}
  The category $\xF_{*}$ may be thought of as the wide subcategory
  $\Span_{\text{inj},\txt{all}}(\xF)$ of $\SpF$ containing only those
  morphisms where the backwards map is injective. The inert-active
  factorization system on $\xF_{*}$ then coincides with the one
  obtained by restriction from $\SpF$, and the inclusion
  $\xF_{\ast} \to \SpF$ is an iso-Segal morphism.
\end{observation}

\begin{ex}\label{ex:SpanFG}
  Let $G$ be a finite group and $\xF_G$ the category of finite
  $G$-sets. Denote by $\Orb_G \subseteq \xF_G$ the collection of
  $G$-orbits (\ie{} transitive $G$-sets). Since $\xF_{G}$ has pullbacks
  we have an \icat{} (really a (2,1)-category)
  $\Span(\xF_{G})$. Abusing notation slightly, we also denote the span
  pattern with the orbits as elementary objects by 
  $\Span(\xF_G):=\Span(\xF_G;\Orb_G)$.  
  Segal objects for this pattern are precisely $G$-commutative monoids in the sense of \cite{NardinThesis}; 
  they also appear in \cite{redshift}
  where they are called semiMackey functors.
    More generally, for any full subcategory
    $\calF \subseteq \Orb_G$ we have a span pattern
    $\Span(\xF_G;\calF)$ whose Segal objects may be thought of as
    $G$-commutative monoids that are Borel-$\calF$-complete.  Segal
    objects for $\Span(\xF_G;\{G/e\})$ appear implicitly in
    \cite{redshift}, where they are called Borel-equivariant.
\end{ex}

\begin{ex}\label{ex:Fin-G}
  As a variant of the previous example, we can consider subcategories
  $\xF_{G}^{f}$ of $\xF_{G}$ that are closed under base change; if
  $\xF_G^f$ is moreover closed under finite coproducts, this data is
  equivalent to an \textit{indexing system} in the sense of
  \cite{blumberg-hill}. We can then define the span pattern
  $\Span_{\text{all},f}(\xF_{G}) := \Span_{\text{all},f}(\xF_{G};
  \Orb_{G})$, whose Segal objects we can think of as $G$-commutative
  monoids where only transfers that lie in $\xF_{G}^{f}$ are allowed.
  As an illustrative example we may consider the extreme case where
  all forward maps are isomorphisms, \ie{} $\xF_G^f:=\xF_G^{\simeq}$.
  The corresponding span pattern
  $\Span_{\text{all},\simeq}(\xF_G;\Orb_G)$ has an underlying \icat{}
  equivalent to $\xF_G^{\op}$ with all the maps inert and with
  $\Orb_G^{\op}$ as the subcategory of elementary objects.  Segal
  objects for this pattern are thus equivalent to presheaves on
  $\Orb_G$, and by Elemendorf's theorem this \icat{} is equivalent to
  that of $G$-spaces.
\end{ex}

\begin{ex}\label{ex:m-commutative}
    A space $X \in \calS$ is called \emph{$m$-finite} if it is $m$-truncated and all of its homotopy groups are finite; we let $\calS_m \subseteq \calS$ denote the full subcategory of $m$-finite spaces. 
    Since $m$-finite spaces are closed under finite limits we may
    consider the span pattern $\Span(\calS_m):=\Span(\calS_m;\pt)$. 
    If we write
    $\mathcal{S}_{m}^{n\txt{-tr}}$ for the wide subcategory of
    $\mathcal{S}_{m}$ whose maps are $n$-truncated, then
    $(\mathcal{S}_{m},\mathcal{S}^{n\txt{-tr}}_{m},\mathcal{S}_{m})$ is also
      an adequate triple, and we can likewise consider the pattern
      \[ \Span_{n\text{-tr},\text{all}}(\mathcal{S}_{m}) :=
        \Span_{n\text{-tr},\text{all}}(\mathcal{S}_{m}; *) \]
      for any $n$.
   For $n=m-1$, the Segal objects for $\Span_{(m-1)\text{-tr},\text{all}}(\mathcal{S}_{m})$ are precisely the
    \emph{$m$-commutative monoids} of Harpaz~\cite{HarpazAmbi}. It also follows from \cite{HarpazAmbi}*{Proposition 5.14} that these are equivalent to Segal
      objects for $\Span(\calS_m)$.
\end{ex}


\subsection{Sound patterns}\label{sec:sound}

In this subsection we define the notion of a \emph{sound} pattern
--- a technical condition satisfied in almost all the usual examples.
This requires first introducing some notation:

\begin{notation}\label{omega_alpha!}
    Fix a morphism $\omega\colon X \to Y$ in an algebraic pattern $\calO$.
    For every elementary object $(\alpha\colon Y \xintto{} E) \in \calO_{Y/}^{\el}$ 
    we denote the inert-active factorization of $\alpha \circ \omega$ as follows:
  \[
    \begin{tikzcd}
        X \ar[d, "\omega"']  \ar[r, tail, "\omega_\alpha"] & 
        \omega_{\alpha!} X \ar[d, squiggly, "(\alpha \circ \omega)^\act"] \\
        Y \ar[r, "\alpha", tail] & E
    \end{tikzcd}
  \]
    Factorization defines a functor $\omega_{(-)}\colon \calO_{Y/}^\el \to \calO_{X/}^\xint$
    by sending $\alpha$ to $\omega_\alpha$.
\end{notation}

\begin{defn}
  For $\omega\colon X \actto Y$ we define $\calO^\el(\omega)$ as the pullback
  \[
    \begin{tikzcd}[column sep = large]
      {\calO^\el(\omega)} \ar[r] \ar[d] & 
      {\Ar(\calO^\xint_{X/})} \ar[d, "{(s,t)}"] \\
      \calO_{Y/}^\el \times \calO_{X/}^\el  \ar[r, "{(\omega_{(-)},\id)}"] &
      \calO_{X/}^\xint \times \calO_{X/}^\xint.
    \end{tikzcd}
  \]
  An object in $\mathcal{O}^{\el}(\omega)$ can thus be represented by 
  a diagram in $\calO$ of the following shape:
  \[
    \begin{tikzcd}
      X \ar[d, squiggly, "\omega"']  \ar[r, tail, "\omega_\alpha"] & 
      \omega_{\alpha!}X \ar[d, squiggly] \ar[r, tail] &
      E'
      \\
      Y \ar[r, tail, "\alpha"] & E,
    \end{tikzcd}
  \]
  where the arrows labeled by $\xintto{}$ and $\actto{}$ are required to be 
  inert and active, respectively, $E$ and $E'$ are elementary, and $\omega$ is fixed.
  Morphisms in $\calO^\el(\omega)$ are natural transformations of such diagrams
  that are constant at $\omega\colon X \actto{} Y$ and inert at all other objects.
\end{defn}

\begin{remark}\label{rmk:Oel(phi)-and-limits}
    By construction $\calO^\el(\omega) \to \calO_{Y/}^\el \times \calO_{X/}^\el$ is the bifibration (see \cite[Definition 2.4.7.2]{HTT}) 
    corresponding to the functor
    \[
        (\calO_{Y/}^\el)^{\op} \times \calO_{X/}^\el  \to \calS,  \qquad
        (\alpha\colon Y \intto E, \beta \colon X\intto E') \mapsto \Map_{\calO_{X/}^\xint}(\omega_\alpha, \beta).
      \]
\end{remark}

\begin{defn}\label{defn:sound}
    We say that a pattern $\calO$ is \emph{sound} if 
    for every active morphism $\omega \colon X \actto Y$ the functor $\calO^\el(\omega) \to \calO^\el_{X/}$ is coinitial.
\end{defn}

The point of introducing the condition of soundness is that it allows us to rewrite certain double limits, as described below in \cref{lem:sound-stronger-Segal}. Before we state this property we look at a first example, namely $\xF_{*}$, where soundness is particularly easy to check; further examples will be given below.

\begin{ex}\label{ex:xF-is-sound}
  In the pattern $\xF_{*}$ an active morphism 
  $\omega \colon X_+ \actto{} Y_+$ is simply a map $\omega \colon X \to Y$ in $\xF$.
  The inert undercategory $(\xF_*)^\xint_{Y_+/}$ may be identified 
  with the poset $(\mathrm{Sub}(Y), \supseteq)$ of subsets of $Y$, by assigning to each $\gamma \colon Y_+ \intto Z_+$ the subset $\gamma^{-1}(Z) \subset Y$.
  The category of elementary objects under $Y_+$ is given by the one-element subsets, and we may hence identify it with $Y$ itself.
  For an elementary $\alpha \colon Y_+ \intto{} E$ corresponding to $e \in Y$,
  the pushforward $\omega_{\alpha!}X_+$ 
  can be identified with $\omega^{-1}(e)_+ \subset X_+$.
  Hence we have a cartesian square:
    \[
        \begin{tikzcd}[column sep = large]
            {\xF_*^\el(\omega)} \ar[r] \ar[d] & 
            {\Ar(\Sub(X))} \ar[d, "{(t,s)}"] \\
            X \times Y  \ar[r, "{(\id, \omega^{-1})}"] &
            \Sub(X) \times \Sub(X).
        \end{tikzcd}
    \]
    and so $\xF_*^\el(\omega)$ is the poset of pairs $(x,y) \in X \times Y$ such that $\{x\} \subset \omega^{-1}(y)$.
    In other words, $y = \omega(x)$ and hence the map $\xF_*^\el(\omega) \to (\xF_*)^\el_{X_+/} \simeq X$ is an equivalence.
    In particular it is coinitial and thus $\xF_*$ is sound.
\end{ex}

\begin{observation}
    The composite $\calO^\el(\omega) \to \calO_{X/}^\el \times \calO_{Y/}^\el \to  \calO_{Y/}^\el$ is by construction a cartesian fibration.
    Its straightening is the functor 
    \[
        (\calO_{Y/}^\el)^\op \xrightarrow{\omega_{(-)}} (\calO_{X/}^\xint)^\op
        \xrightarrow{\calO_{-/}^\el} \Cat
    \]
    that sends $\alpha \colon Y \intto{} E$ to the \icat{} $\calO_{\omega_{\alpha!}X/}^\el$
    of elementaries under $\omega_{\alpha!}X$.
    Our definition of $\calO^\el(\omega)$ therefore matches that given in \cite[Remark 7.6]{patterns1}. Moreover, a limit over $\mathcal{O}^{\el}(\omega)$ can be rewritten as a double limit, that is for $F \colon \mathcal{O}^{\el}(\omega) \to \mathcal{C}$ we have
    \[ \lim_{\mathcal{O}^{\el}(\omega)} F \simeq \lim_{\alpha \in
        \mathcal{O}^{\el}_{Y/}}\lim_{\calO_{\omega_{\alpha!}X/}^\el}F. \]
    If $\calO$ is sound, then we can use this to rewrite a limit over
    $\mathcal{O}^{\el}_{X/}$ as a double limit.
\end{observation}

We now show that soundness is inherited along iso-Segal morphisms.
Together with \cref{ex:xF-is-sound} this implies that all cartesian patterns of \cite[Definition 2.6]{patterns2} are sound, and in particular any $\infty$-operad in the sense of Lurie is sound.
\begin{lemma}\label{lem:iso-Segal-sound}
    Let $f\colon \calO \to \calP$ be an iso-Segal morphism of algebraic patterns.
    Then $\calO$ is sound, if $\calP$ is sound.
    The converse implication holds if we further assume that \[\Ar^{\act}(f)\colon \Ar^\act(\calO) \to \Ar^\act(\calP)\] is essentially surjective.
\end{lemma}
\begin{proof}
    Being a morphism of algebraic patterns, $f$ induces for each active $\omega \colon X \actto Y$ a morphism of cartesian fibrations 
    \[\begin{tikzcd}
        \calO^\el(\omega) \ar[r, "f"] \ar[d] & \calP^\el(f(\omega)) \ar[d] \\
        \calO^\el_{Y/} \ar[r, "f", "\simeq"'] & \calP^\el_{f(Y)/}
    \end{tikzcd}\]
    where the bottom functor is an equivalence because we assumed that $f$ is iso-Segal.
    On the fibers over some $(\alpha \colon Y \intto E) \in \calO^\el_{Y/}$ and $f(\alpha) \in \calP^\el_{f(Y)/}$ we get an induced functor
    \[
         f\colon \calO^\el_{\alpha_!X/} \longrightarrow \calP^\el_{f(\alpha_! X)/}
    \]
    which again is an equivalence because $f$ is iso-Segal.
    Therefore it follows that $f\colon \calO^\el(\omega) \to \calP^\el(f(\omega))$ is an equivalence.
    This is also the top map in the square
    \[\begin{tikzcd}
        \calO^\el(\omega) \ar[r, "f", "\simeq"'] \ar[d] & \calP^\el(f(\omega)) \ar[d] \\
        \calO^\el_{X/} \ar[r, "f", "\simeq"'] & \calP^\el_{f(X)/}.
    \end{tikzcd}\]
    Here the right functor is coinitial because $\calP$ is sound, and hence the left functor is coinitial, which proves that $\calO$ is sound. 

    For the converse implication, let $\rho\colon X' \actto Y'$ be some active morphism in $\calP$.
    Because we assume that $\Ar^{\act}(f)$ is essentially surjective, we can write $\rho = f(\omega)$ for $\omega\colon X \actto Y$ active in $\calO$ as before. 
    Then the above argument shows that $\calP^\el(\rho) \to \calP^\el(Y')$ must be coinitial, as it is equivalent to $\calO^\el(\omega) \to \calO^\el(Y)$.
\end{proof}

The crucial application of soundness for us will be through the following lemma: this will be used in the proof of \cref{lem:relative-verySegal-iff-unstraighting-is-relative-fibrous}, which is how the assumption of soundness enters our main theorem.
\begin{lemma}\label{lem:sound-stronger-Segal}
    Let $\calO$ be a sound pattern and $\calC$ a sufficiently complete \icat{}.
    Consider a natural transformation 
    $(\eta\colon  F \Rightarrow G)\colon \calO_{X/}^\xint \to \calC$ such that
    for all $X \intto{} X' \in \calO_{X/}^\xint$ the square 
    \[ \begin{tikzcd}
        F(X') \ar[r] \ar[d, "\eta_{X'}"'] &
        \lim_{X' \intto{} E \in \calO_{X'/}^\el} F(E) \ar[d, "\lim \eta_E"] \\
        G(X') \ar[r] &
        \lim_{X' \intto{} E \in \calO_{X'/}^\el} G(E)
    \end{tikzcd} \]
    is cartesian.
    Then for every active morphism $\omega\colon X \actto{} Y$ the square
    \[ \begin{tikzcd}
        F(X) \ar[r] \ar[d, "\eta_X"'] & 
        \lim_{\alpha\colon  Y \intto E \in \calO^\el_{Y/}} F(\omega_{\alpha!} X) \ar[d, "\lim \eta_E"] \\
        G(X) \ar[r] & 
        \lim_{\alpha\colon  Y \intto E \in \calO^\el_{Y/}} G(\omega_{\alpha!} X)
    \end{tikzcd} \]
    is cartesian.
\end{lemma}
\begin{proof}
    Consider the commutative cube
    \[
    \begin{tikzcd}[column sep=-2ex, row sep=tiny]
      F(X) \arrow{dr} \arrow{rr}
      \arrow{dd} & &
      {\lim_{\alpha\colon Y \intto{} E' \in \calO^\el_{Y/}} F(\omega_{\alpha!}X)} \arrow{dd} 
      \arrow{dr} \\
       & \lim_{\beta\colon X \intto{} E \in \calO^\el_{X/}} F(E) \arrow[crossing
       over]{rr}[near start]{\sim}  & &
       \lim_{(\alpha\colon Y \intto{} E', \gamma\colon  \omega_{\alpha!}X \intto{} E) \in \calO^\el(\omega)} F(E) \arrow{dd} \\
       G(X) \arrow{dr} \arrow{rr}{\sim} & &
      \lim_{\alpha\colon Y \intto{} E' \in \calO^\el_{Y/}} G(\omega_{\alpha!}X) \arrow{dr} \\
      & \lim_{\beta\colon X \intto{} E \in \calO^\el_{X/}} G(E) \arrow{rr}[near start]{\sim}
       \arrow[leftarrow,crossing over]{uu}  & &
       \lim_{(\alpha\colon Y \intto{} E', \gamma\colon  \omega_{\alpha!}X \intto{} E) \in \calO^\el(\omega)} G(E).
    \end{tikzcd}
    \]
    The front horizontal maps are equivalences because $\calO$ 
    is assumed to be sound and hence $\calO^\el(\omega) \to \calO^\el_{X/}$
    is coinitial.
    The left square is cartesian by applying the assumption.
    We would like to show that the back square is cartesian
    and by pullback pasting it will suffice to show that the right
    square is cartesian.
    We may write the limit over $\calO^\el(\omega)$ as a double limit,
    by first right Kan extending
    along the cartesian fibration $\calO^\el(\omega) \to \calO^\el_{Y/}$,
    which is computed by taking limits over the fibers $\calO^\el_{\omega_{\alpha!}Y/}$,
    and then taking the limit over $\calO^\el_{Y/}$.
    Using this reformulation the right square can be written as a $\calO^\el_{Y/}$-limit of diagrams of the form
    \[ \begin{tikzcd}
        F(\omega_{\alpha!}X) \ar[r] \ar[d, "\eta_{\omega_{\alpha!}X}"'] &
        \lim_{\omega_{\alpha!}X \intto{} E \in \calO_{X'/}^\el} F(E) \ar[d, "\lim \eta_{E}"] \\
        G(\omega_{\alpha!}X) \ar[r] &
        \lim_{\omega_{\alpha!}X \intto{} E \in \calO_{X'/}^\el} G(E).
    \end{tikzcd} \]
    Each of these diagrams is cartesian by assumption, and hence so is their limit. 
\end{proof}

We will now check explicitly that the examples of patterns we discussed
above are indeed sound. To do so, the following observation will be useful:

\begin{lemma}\label{lem:O-el-beta}
    For an algebraic pattern $\calO$ the following conditions are equivalent:
    \begin{enumerate}
        \item
        $\calO$ is sound.
        \item 
        For every active morphism $\omega \colon X \actto{} Y$ and $\beta\colon X \xintto{} E' \in \calO^{\el}_{X/}$, the \icat{}
        \[\calO^\el_\beta(\omega) := \calO^\el_{Y/} \times_{\calO^\xint_{X/}} (\calO^\xint_{X/})_{/\beta}\] 
        is weakly contractible.
        \item 
        For every $\omega$ and $\beta$ as in $(2)$ we have $\colim_{\alpha \in (\calO^{\el}_{Y/})^{\op}} \Map_{\calO^{\xint}_{X/}}(\omega_{\alpha},\beta) \simeq \ast$.
    \end{enumerate}
\end{lemma}
\begin{proof}
    $(1 \Leftrightarrow 2)$ The functor $\calO^\el(\omega) \to \calO^\el_{X/}$ is a cocartesian fibration.
    By the dual of \cite[Theorem 4.1.3.2.]{HTT}
    it is coinitial if and only if its fibers are weakly contractible.
    Unwinding definitions yields the following description of the straightening: 
    \[
        \calO^\el_{X/} \to \Cat, \qquad
        (\beta\colon X \to E') \longmapsto \calO^\el_{Y/} \times_{\calO^\xint_{X/}} (\calO^\xint_{X/})_{/\beta}.
    \]
    $(2 \Leftrightarrow 3)$ 
    Since $\mathcal{O}^{\el}(\omega) \to \calO^{\el}_{Y/} \times \calO^{\el}_{X/}$ 
    is a bifibration, passing to the fiber over
    $\beta \in \calO^{\el}_{X/}$
    and taking opposites yields a left fibration $q:\mathcal{O}^{\el}_{\beta}(\omega)^\op \to (\mathcal{O}^{\el}_{Y/})^\op$.
    By \cite[Corollary 3.3.4.6]{HTT}, the $\infty$-groupoid
    $|\calO^\el_\beta(\omega)| \simeq |\calO^\el_\beta(\omega)^{\op}|$ 
    can be computed as the colimit of the straightening $\St{}{}(q)$, which is given by
    \[
        \St{}{}(q)\colon (\calO^\el_{Y/})^{\op} \to \calS, \qquad
        (\alpha\colon Y \to E) \longmapsto \Map_{\calO^{\xint}_{X/}}(\omega_{\alpha},\beta). \qedhere
    \]    
\end{proof}

\begin{observation}\label{obs:Oel-poset}
    Suppose that $\calO$ is a pattern such that for all $X \in \calO$ 
    the inert under-category $\calO^\xint_{X/}$ is a poset. 
    In this case, spelling out the definition as in \cref{ex:xF-is-sound} we may identify $\calO^\el_\beta(\omega)$ with the following sub-poset of $\calO^\el_{Y/}$:
    \[
        \calO^\el_\beta(\omega) \simeq
        \{ (\alpha\colon  Y \intto E) \in \calO^\el_{Y/} \;|\;
        \beta = \gamma \circ \omega_\alpha \}.
    \]
\end{observation}

\begin{ex}\label{ex:Dop-is-sound}
    For the pattern $\Dop$ the inert under-category $(\Dop)^\xint_{[n]/}$ 
    is equivalent to the poset of pairs $(a_0 \le a_1) \in [n]$. 
    This is elementary in $\simp^{\op,\flat}$ iff $a_1-a_0 = 1$ 
    and it is elementary in $\simp^{\op,\natural}$ iff $a_1-a_0 \le 1$.
    To check soundness we consider, for an active morphism $\omega \colon [m] \actto [n]$ in $\simp$ 
    and elementary $(b_0 \le b_1) \in [n]$, the poset
    \[
        (\Dop)^\el_\beta(\omega) \simeq
        \{  (a_0 \le a_1) \in (\Dop)^\el_{[m]/} \;|\; \omega(a_0) \le b_0 \le b_1 \le \omega(a_1) \}.
    \]
    In the case of $\simp^{\op,\flat}$ this poset has a single element,
    namely that given by $a_0 = \max\{a \in [m]\;|\; \omega(a) \le b_0\}$
    and $a_1 = a_0 + 1$, which satisfies $\omega(a_1) > b_0$
    and hence $\omega(a_1) \ge b_1 = b_0+1$.
    For the pattern $\simp^{\op,\natural}$ the poset still has 
     a single element if $b_1 = b_0 +1$ 
    or if $b_1 = b_0$ with $b_i \not\in \omega([m])$.
    But if $b_1 = b_0 = \omega(a)$ for some $a \in [m]$,
    then the poset is the category
    \[
        (a-1 \le a) \to (a \le a) \from (a \le a+1),
    \]
    which is not trivial, but still weakly contractible.
    This shows that $\simp^{\op,\flat}$ and $\simp^{\op,\natural}$ 
    are both sound.
\end{ex}

\begin{ex}\label{ex:Fptnatural}
    The pattern $\xF_*^\natural$ is sound.
    The inert under-category $(\xF_*^\natural)_{A_+/}$ is the poset of subsets $U \subset A$.
    Given an active morphism $\omega\colon A_+ \to B_+$ and an elementary $E \subset A$ (i.e.\ $|E| \le 1$),
    we need to check that the poset of $E' \subset Y$ with $|E'| \le 1$ and $E \subset \omega^{-1}(E')$ is contractible.
    If $E = \{a\} \neq \emptyset$, then this poset has exactly one element $E' = \{\omega(a)\}$,
    and if $E = \emptyset$, then this poset has an initial element $E' = \emptyset$.
    So the poset is contractible in both cases, which proves that $\xF_*^\natural$ is sound.
\end{ex}

\begin{lemma}\label{lem:prodsound}
        Products of sound patterns are sound: if $\calO_1$ and $\calO_2$ are sound patterns, then $\calO_1 \times \calO_2$ is also a sound pattern.
\end{lemma}
\begin{proof}
  Let $\omega = (\omega_1,\omega_2) \colon  (X_1,X_2) \actto{} (Y_1,Y_2)$ be
  an active morphism in $\calO_1 \times \calO_2$. The projection
  $(\calO_1 \times \calO_2)^{\el}(\omega) \longrightarrow (\calO_1
  \times \calO_2)^{\el}_{(X_1,X_2)/}$ can be identified with the
  product of the projections
  $\calO_1^{\el}(\omega_1) \times \calO_2^{\el}(\omega_2)
  \longrightarrow (\calO_1)^{\el}_{X_1/} \times
  (\calO_2)^{\el}_{X_2/}$ which, by assumption, is a product of
  coinitial functors and hence again coinitial.
\end{proof}

\begin{ex}
  Applying \cref{lem:prodsound} to \cref{ex:Dop-is-sound}, we see that the patterns $\simp^{n,\op,\flat}$ and $\simp^{n,\op,\natural}$ are both sound.
\end{ex}


Next, we introduce a further condition for sound patterns; for this we first need some notation:
\begin{notation}
  By \cref{propn:free-fibration-is-cocartesian}, evaluation at the target
  $\ev_{1} \colon \Ar_{\act}(\calO) \to \calO$ is a cocartesian fibration.
  Its straightening, denoted by
  $\calA_{\calO} \colon \mathcal{O} \to \CatI$, takes $X \in \mathcal{O}$ to the \icat{}
  $\calA_{\calO}(X) \simeq \mathcal{O}^{\act}_{/X}$ of active morphisms to $X$.
  (Compare with \cite{patterns1}*{Corollary 7.4 and Remark 7.5}.)
\end{notation}

\begin{defn}\label{defn:soundly-extendable}
  We say an algebraic pattern $\mathcal{O}$ is \emph{soundly extendable} 
  if it is sound and in addition
  the functor  $\calA_{\calO}$
  is a Segal $\calO$-$\infty$-category,
  \ie{} for every $X \in \mathcal{O}$, the functor
  \[ \calO_{/X}^\act \to \lim_{E \in \mathcal{O}^{\el}_{X/}} \calO_{/E}^\act \]
  is an equivalence.
\end{defn}

\begin{remark}\label{rmk:extOactseg}
  The notion of a soundly extendable pattern is a mild strengthening
  of the notion of extendable pattern from
  \cite{patterns1}*{Definition 8.5} (which uses a slightly weaker, but more complicated
  condition than what we are here calling ``soundness'').
  It was shown in \cite{patterns1}*{Lemma 9.14} that every extendable
  pattern $\calO$ satisfies the condition in
  \cref{defn:soundly-extendable}, so in particular a sound pattern is
  extendable if and only if it is soundly extendable.  In principle,
  there could exist extendable patterns that are not sound, but we are
  not aware of any examples.
\end{remark}

\begin{ex}\label{ex:soundext}
    The patterns $\xF_{*}$, $\simp^{\op,\natural}$, and $\simp^{\op,\flat}$ 
    are soundly extendable.
    Their soundness was verified in \cref{ex:xF-is-sound} and \cref{ex:Dop-is-sound}.
    For extendability see \cite[Example 8.13 and 8.14]{patterns1}. The pattern
    $\bbTheta_{n}^{\op,\natural}$ is soundly extendable for all $n$ 
    (by \cite[Proposition 2.7]{thetan} and \cite[Lemma 3.5]{thetan}),
    but note that 
    $\bbTheta_{n}^{\op,\flat}$ fails to be extendable for $n > 1$.
    (See \cite[Example 8.15]{patterns1}.)
\end{ex}

\begin{ex}\label{ex:operads-soundly-extendable}
    Let $\calO \to \xF_\ast$ be an \iopd{}. 
    Then $\calO$ is a soundly extendable pattern. 
    This will follow by \cref{ex:xF-fibrous} and \cref{lem:fbrssound} in the next section.
\end{ex}

\begin{ex}\label{ex:xFk-sound-not-extendable}
    The patterns $\xF_{*}^{\le k}$ are sound by \cref{lem:iso-Segal-sound},
    but not soundly extendable. 
    Indeed, $\calA_{\xF_{\ast}^{\le k}}\colon \xF_{\ast}^{\le k} \xrightarrow{} \CatI$ does not satisfy the Segal condition:
    for any $n \le k$ the Segal map may be identified with the inclusion
    \[  (\xF^{\times n})^{\le k} \simeq \calA_{\xF_{\ast}^{\le k}}(n)  \longrightarrow \calA_{\xF_{\ast}^{\le k}}(1)^{\times n} \simeq (\xF^{\le k})^{\times n} \]
    where $(\xF^{\le k})^{\times n}$ is the category of $n$-tuples of sets such that each set has size $\le k$,
    and $(\xF^{\times n})^{\le k}$ denotes the full subcategory on those $n$-tuples of total size $\le k$.
\end{ex}

\begin{lemma}\label{lem:product-of-soundly-extendable}
    Let $\calO$ and $\calP$ be soundly extendable patterns such that $\calO^{\el}_{O/}$ and $\calP^{\el}_{P/}$ are weakly contractible for all $O \in \calO$ and $P\in \calP$.
    Then $\calO \times \calP$ is a soundly extendable pattern.
\end{lemma}
\begin{proof}
    Soundness follows from \cref{lem:prodsound}.
    For extendability we have:
    \begin{align*}
        \lim_{(\alpha,\beta) : (O,P) \intto{} (E,E') \in (\calO \times \calP)^{\el}_{(O,P)/}} (\calO\times \calP)^{\act}_{/(E,E')}
        & \simeq
        \lim_{(\alpha:O\intto{} E,\beta:P\intto{} E') \in \calO^\el_{O/} \times \calP^\el_{P/}} \calO^{\act}_{/E} \times \calP^{\act}_{/E'}\\
        & \simeq 
        \lim_{\alpha:O\intto{} E\in \calO^\el_{O/}} \calO^{\act}_{/E} \times \lim_{\beta:P\intto{} E' \in \calP^\el_{P/}} \calP^{\act}_{/E'}\\
        & \simeq \calO^\act_{/O} \times \calP^{\act}_{/P}
    \end{align*}
    where in the second line we used that in any \icat{}, products distribute over weakly contractible limits.
\end{proof}

\begin{ex}
    The pattern $\simp^{n,\op,\natural}$ is soundly extendable.
    Indeed the case $n=1$ appears in \cref{ex:soundext}, and for $n >1$ this follows from \cref{lem:product-of-soundly-extendable} by observing that $(\simp^{\op,\natural})^\el_{[k]/}$ is weakly contractible for all $k$. (Note that this argument fails for $\simp^{n,\op,\flat}$ since $(\simp^{\op,\flat})^{\el}_{[0]/} = \emptyset$, and indeed this pattern is \emph{not} extendable for $n > 1$.)
\end{ex}

\begin{propn}\label{prop:Span-sound}
    The pattern $\Spanbf(\frX;\frX_0)$, as defined in \cref{defn:span}, is
   \begin{enumerate}[(1)]
        \item\label{it:spansound-alternative}  
        sound if 
        $\frXb_{/y} \to \frX_{/y}$ is fully faithful and
        the inclusion
            $\frXb_{0/y} \hookrightarrow \frX_{0/y}$ 
        is cofinal for every $y \in \frX$.
        \item\label{item:sound-ext-span} soundly extendable if and only if it is sound and the functor $\frXf_{/-} \colon \frX^{b,\op} \to \CatI$ (defined on morphisms by pullback) is right Kan extended from $\frX_{0}^{b,\op} \subseteq \frX^{b,\op}$.
    \end{enumerate}
\end{propn}
\begin{proof}
    \ref{it:spansound-alternative}
    By \cref{lem:O-el-beta}.(3) the pattern $\Span^{b,f}(\frX;\frX_0)$ is sound if and only if 
    for every $\beta \colon e' \to x$ in $\frXb$ and $\omega\colon x \to y$ in $\frXf$ the following 
    colimit indexed by $\alpha\colon e \rightarrow y \in \frX^b_{0/y}$
    is contractible:
    \begin{align*}
        & \colim_{\alpha \in \frX^b_{0/y}} \Map_{\frX^b_{/x}}(\beta\colon e' \rightarrow x,\; \omega^*\alpha\colon  x \times_y e \rightarrow x)  \\
        & \simeq \colim_{\alpha \in \frX^b_{0/y}} \Map_{\frX_{/x}}(\beta\colon e' \rightarrow x,\; \omega^*\alpha\colon  x \times_y e \rightarrow x)  
        & (\frXb_{/x} \subset \frX_{/x} \text{ full}) \\
        & \simeq \colim_{\alpha \in \frX^b_{0/y}} \Map_{\frX_{/y}}(\omega \circ \beta\colon e' \rightarrow y,\; \alpha\colon  e \rightarrow x)  
        & (\omega_! \dashv \omega^*) \\
        & \simeq \colim_{\alpha \in \frX^b_{0/y}} \Map_{\frX_{0/y}}(\omega \circ \beta\colon e' \rightarrow y,\; \alpha\colon  e \rightarrow x)  
        & (\frX_{0/y} \subset \frX_{/y} \text{ full}) \\
        & \simeq |\frX^b_{0/y} \times_{\frX_{0/y}}(\frX_{0/y})_{\omega \circ \beta/}|
    \end{align*}
    By \cite{HTT}*{Theorem 4.1.3.1} this category is weakly contractible if $\frXb_{0/y} \to \frX_{0/y}$ is cofinal, so the claim follows.
    
    \ref{item:sound-ext-span} 
    Since $\Spanbf(\frX;\frX_0)^{\act}_{/-} \simeq \frX_{/-}^f$,
    this is a consequence of the fact that a functor is Segal if and only if its restriction to the inert category is right Kan extended from the elementaries by \cite[Lemma 2.9]{patterns1}.
\end{proof}

As an important special case, we have:
\begin{cor}\label{cor:b=all-sound}
  If $\frXb=\frX$ then $\Span_{\txt{all},f}(\frX;\frX_0)$ is sound. 
\end{cor}

\begin{ex}\label{ex:spansoundext}
  The pattern $\Span(\xF)$ is soundly extendable.
\end{ex}

\begin{ex}\label{ex:spanFGrestrsound}
    Let $\xF_G^f \subset \xF_G$ be closed under base-change and coproduct as in \cref{ex:Fin-G}.
    The patterns $\Span_{\txt{all},f}(\xF_G)$ and $\Span_{\txt{inj},f}(\xF_G)$ are soundly extendable.
    The slice $(\xF_G)_{/A}^f$ decomposes as a product $\prod_{U \in A/G} (\xF_G)_{/U}^f$ since the morphisms of $\xF_G^f$ are closed under base-change.
    This implies that $(\xF_G)_{/-}^f$ is a $\Span_{\txt{inj},f}(\xF_G)$-Segal category 
    since the elementary slice category $\Span_{\txt{inj},f}(\xF_G)_{A/}^\el \simeq (\Orb_G)_{/A}^{\txt{inj}}$
    is equivalent to the discrete set $A/G$ over which we are taking the product.
    It also follows that $(\xF_G)_{/-}^f$ is a Segal $\Span_{\txt{all},f}(\xF_G)$-category since 
    $(\Orb_G)_{/A}^{\txt{inj}} \simeq \Span_{\txt{inj},f}(\xF_G)_{A/}^\el \hookrightarrow \Span_{\txt{all},f}(\xF_G)_{A/}^\el $ 
    is coinitial.
\end{ex}

\begin{ex}\label{ex:SpanSm-soundly-extendable}
    The pattern $\Span(\mathcal{S}_{m})$
    is soundly extendable. 
    Soundness follows from \cref{cor:b=all-sound}. 
   For extendability 
    we need to show that the functor
    \[
        (\calS_m)_{/-} \colon  \calS_m^\op \to \CatI
    \]
    is right Kan extended from its value at $* \in \calS_m^\op$.
    Since being $m$-truncated can be checked fiberwise over $Y \in \calS_m$, this functor is equivalent to $\Fun(-, \calS_m)$ by straightening.
    This is now right Kan extended because $\Fun(X, \calS_m) \simeq \lim_X \calS_m$.
    One can show that $\Span_{(m-1)\txt{-tr},\txt{all}}(\mathcal{S}_{m})$ is also soundly extendable; we will not need this, however.
\end{ex}

Finally, we give an example of a pattern that is \emph{not} sound:
\begin{ex}\label{ex:not-sound}
    We expect that the pattern $\mathbf{U}^\op$ of undirected graphs of Hackney, Robertson, and Yau \cite{HRYmodular} is sound.
    However, this pattern does not include the nodeless loop $S^1$.
    In \cite{HackneyGraphs}, Hackney gives a simpler description of $\mathbf{U}^\op$ and also defines a variant $\widetilde{\mathbf{U}}^\op$ that does include the nodeless loop. 
    We will now show that this is an example of a non-sound pattern $\calO = \widetilde{\mathbf{U}}^\op$.
    For the sake of brevity we shall not recall the definition, but rather the following facts:
    \begin{itemize}
        \item The category of elementaries under $S^1$ is trivial $\calO_{S^1/}^\el \simeq *$.
        \item There is an active morphism $\omega \colon S^1 \actto S^1_n$ to the $n$-vertex loop $S^{1}_{n}$ ($n\ge 2$), for which $\calO_{S^1_n/}^\el$ is the poset of simplices of $S^1_n$,
        which is weakly equivalent to $S^1$.
    \end{itemize}
    We can now use the characterisation of soundness from \cref{lem:O-el-beta}.(3)
    in the case of the active morphism $\omega\colon  S^1 \actto S^1_n$ described above.
    Since $\calO_{S^1/}^\el$ is trivial (and in this case $\omega_{\alpha!}S^1$ is always elementary), the colimit runs over the constant diagram on the point and hence evaluates to the classifying space of $\calO^\el_{S^1_n/}$, which is not contractible. 
    Note that this could be resolved by introducing a variant of $\widetilde{\mathbf{U}}^\op$ where 
    $\Map_{\calO}(S^1, S^1) \simeq \Map_{\calO}(S^1, e) \simeq O(2)$,
    in which case $\calO_{S^1/}^\el$ is equivalent to the $\infty$-groupoid $S^1$.
\end{ex}

\section{Fibrous patterns and Segal envelopes}\label{sec:fbrs}

We begin this section by introducing the notion of \emph{fibrous
  $\mathcal{O}$-patterns} as a generalization of \iopds{} over an
arbitrary base pattern $\mathcal{O}$ in
\S\ref{sec:weak-segal-fibr-1}. We then apply the results of
\S\ref{sec:factenv} to fibrous patterns in \S\ref{sec:segenv}, where
we prove \cref{thm:introenv}. 
Finally, in \S\ref{subsec:ex-of-envelopes} we give some examples of Segal envelopes.

\subsection{Fibrous patterns}
\label{sec:weak-segal-fibr-1}

In this subsection we introduce the notion of a \emph{fibrous
  $\mathcal{O}$-pattern} over a base algebraic pattern $\mathcal{O}$.
(We borrow the adjective ``fibrous'' from \cite{HA}*{\S 2.3.3}, where
it is used for a somewhat related concept.) Fibrous patterns
specialize to give, for example, Lurie's \iopds{} and generalized
\iopds{} if we take the base pattern to be $\xF_{*}$ or
$\xF_{*}^{\natural}$. The concept is also a variant of the definition
of \emph{weak Segal fibrations} given in \cite{patterns1}; as we will
see in \cref{propn:fibrous=WSF} the two notions coincide if the base
pattern is sound, \ie{} for almost all interesting examples of
patterns, but the definition of fibrous patterns seems to be 
simpler and better behaved if we do not assume soundness.

\begin{observation}\label{prefbrsdef}
  Let $\mathcal{O}$ be an algebraic pattern. If
  $\pi \colon \calP \to \calO$ has cocartesian lifts of inert
  morphisms, then applying \cref{propn:free-fibration-is-cocartesian}
  to the inert--active factorization system on $\calO$ furnishes a
  cocartesian fibration
  $\mathcal{P} \times_{\mathcal{O}} \Ar_{\act}(\mathcal{O}) \to
  \mathcal{O}$ (where this functor is given as $(P, \pi(P) \actto O) \mapsto O$). For a morphism $\omega \colon O_{1} \to O_{2}$ in $\calO$
  the cocartesian transport functor $\omega_{!} \colon \calP \times_{\calO} \mathcal{O}^{\act}_{/O_{1}} \to \calP \times_{\calO} \mathcal{O}^{\act}_{/O_{2}}$ is given by
  \[ (P,\, \phi \colon \pi(P) \actto O_{1}) \mapsto (\alpha_{!}P, \,\beta \colon O' \actto O_{2}), \]
  where \[\pi(P) \xintto{\alpha} O' \xactto{\beta} O_{2}\] is the inert--active factorization of the composite \[\pi(P) \xactto{\phi} O_{1} \xto{\omega} O_{2}\] and $P \to \alpha_{!}P$ is a cocartesian lift of $\alpha$.
\end{observation}

\begin{defn}\label{defn:fibrous}
  Let $\mathcal{O}$ be an algebraic pattern.
  Then a \emph{fibrous $\calO$-pattern} is a functor $\pi \colon
  \mathcal{P} \to \mathcal{O}$ such that:
  \begin{enumerate}[(1)]
  \item\label{fbrs1} 
    $\mathcal{P}$ has all $\pi$-cocartesian lifts of inert morphisms in $\mathcal{O}$.
  \item\label{fbrs2} For all $O \in \mathcal{O}$, the commutative square of \icats{}
    \[
      \begin{tikzcd}
        \mathcal{P} \times_{\mathcal{O}} \mathcal{O}^{\act}_{/O}
        \arrow{r} \arrow{d} & \lim_{E \in \mathcal{O}^{\el}_{O/}}
        \mathcal{P} \times_{\mathcal{O}} \mathcal{O}^{\act}_{/E}
        \arrow{d} \\
        \mathcal{O}^{\act}_{/O}
        \arrow{r}  & \lim_{E \in \mathcal{O}^{\el}_{O/}}
        \mathcal{O}^{\act}_{/E}
      \end{tikzcd}
    \]
    is cartesian.
    Here the horizontal functors are induced by cocartesian transport along the maps $O \intto E$ in $\mathcal{O}^{\el}_{O/}$ for the cocartesian fibrations from \cref{prefbrsdef}, applied to $\pi$ and $\id_{\calO}$.
  \end{enumerate}
\end{defn}

\begin{observation}\label{obs:fbrsrelseg}
  Condition \ref{fbrs2} in \cref{defn:fibrous} says precisely that the straightening of the
  projection $\mathcal{P} \times_{\mathcal{O}} \Ar_{\act}(\mathcal{O}) \to
  \Ar_{\act}(\mathcal{O})$ over $\mathcal{O}$, \ie{} the natural
  transformation $\StOint(\mathcal{P})\colon \St{\mathcal{O}}{}(\mathcal{P} \times_{\mathcal{O}} \Ar_{\act}(\mathcal{O})) \to
  \Afun{\mathcal{O}}$,
  is a relative Segal
  $\mathcal{O}$-\icat{}.
\end{observation}

\begin{remark}\label{rmk:fbrs-over-extendable}
  For many patterns $\mathcal{O}$, the functor
  $\mathcal{O}^{\act}_{/\blank}$ is a Segal $\mathcal{O}$-\icat{};
  this is the case, for instance, if $\mathcal{O}$ is
  \emph{extendable} in the sense of \cite{patterns1} by
  \cite{patterns1}*{Lemma 9.14}. In this case, 
  \cref{rmk:cancellation-for-relative-Segal-objects} implies that condition \ref{fbrs2}
  is satisfied if and only if the functor $\StOint(\mathcal{P})$
  is a Segal
  $\mathcal{O}$-\icat{}, \ie{} the functor \[\mathcal{P}
  \times_{\mathcal{O}} \mathcal{O}^{\act}_{/O} \to \lim_{E \in \mathcal{O}^{\el}_{O/}}
  \mathcal{P} \times_{\mathcal{O}} \mathcal{O}^{\act}_{/E}\] is
  an equivalence for all $O \in \mathcal{O}$.
\end{remark}

\begin{ex}\label{ex:xF-fibrous}
  Since $\xF_{*}$ is extendable, a fibrous $\xF_*$-pattern is a
  functor $\pi \colon \calP \to \xF_*$ such that
    $\mathcal{P}$ has $\pi$-cocartesian lifts for inerts, and for all $n$ the functor
    \[
      \calP^\act \times_{\xF} \xF_{/\angled{n}} \simeq \calP \times_{\xF_*} (\xF_*)_{/\angled{n}}^\act 
      \longrightarrow \prod_{\langle n \rangle \intto{} \angled{1}} \calP \times_{\xF_*} \xF
      \simeq (\calP^{\rm act})^n,
    \]
    is an equivalence. This functor takes an object
    $P \in \mathcal{P}^{\act}$ over $\angled{m}$ in $\xF_{*}$ together with
    an active map $\omega\colon  \angled{m} \actto{} \angled{n}$ to the list of objects $(P_1,\dots,P_n)$
    where $P \intto{} P_j$ is the cocartesian lift of the inert map $\omega_j:=(\rho_j \circ \omega )^{\xint}\colon  \angled{m} \intto{} \angled{m}_j$
    where $\rho_j\colon  \angled{n} \intto{} \angled{1}$ is as in \cref{ex:Segal-Finstar}.
    We will see later (\cref{propn:fibrous=WSF}) that this condition is equivalent to $\calP \to \xF_*$ being an $\infty$-operad
    in the sense of Lurie. 
\end{ex}

We can rewrite the second condition in \cref{defn:fibrous} to obtain
the following equivalent characterization of fibrous patterns:
\begin{propn}\label{propn:fibrous-axioms}
  For any algebraic pattern $\mathcal{O}$, a functor
  $\pi \colon \mathcal{P} \to \mathcal{O}$ is a fibrous
  $\calO$-pattern \IFF{}:
  \begin{enumerate}[(1)]
  \item\label{wsf1} $\mathcal{P}$ has $\pi$-cocartesian morphisms over
    inert morphisms in $\mathcal{O}$.
  \item\label{wsf2} For every active morphism $\omega \colon O_1 \actto O_2$ in $\mathcal{O}$, 
    and all objects $X_0 \in \mathcal{P}_{O_0}, X_1 \in \mathcal{P}_{O_1}$,
    the commutative square
      \begin{equation}
        \label{eq:wsegfibmapsq}
        \begin{tikzcd}
          \Map_{\mathcal{P}}(X_0,X_1) \arrow{r} \arrow{d} & \lim_{\alpha
            \colon O_2 \intto E \in \mathcal{O}^{\el}_{O_2/}}
          \Map_{\mathcal{P}}(X_0, \omega_{\alpha, !} X_1) \arrow{d} \\
          \Map_{\mathcal{O}}(O_0,O_1) \arrow{r} & \lim_{\alpha
            \colon O_2 \intto E \in \mathcal{O}^{\el}_{O_2/}}
          \Map_{\mathcal{O}}(O_0, \omega_{\alpha, !} O_1)
        \end{tikzcd}
    \end{equation}
    is cartesian. 
    Here the horizontal maps are defined using the functor 
    $\omega_{(-)} \colon \calO_{O_2/}^\el \to \calO_{O_1/}^\xint$
    from \cref{omega_alpha!}.
  \item\label{wsf3} For every active morphism $\omega \colon  O_1 \actto O_2$ in $\calO$, the functor
    \[ 
        \mathcal{P}_{O_1}^\simeq \to 
        \lim_{\alpha \colon  O_2 \to E \in \mathcal{O}^{\el}_{O_2/}} \mathcal{P}_{\omega_{\alpha !}O_1}^\simeq,
    \]
    induced by cocartesian transport along the inert morphisms 
    $\omega_{\alpha} \colon  O_1 \intto \omega_{\alpha !} O_1$ in $\calO^{\xint}_{O_1/}$, 
    is an equivalence.
  \end{enumerate}
\end{propn}

\begin{proof}
    A square of \icats{} is cartesian \IFF{} the underlying square of
    \igpds{} as well as all induced squares of mapping spaces are cartesian. 
    For the square in the definition of a fibrous pattern the underlying square of \igpds{} is 
    \[
    \begin{tikzcd}
      \mathcal{P}^{\simeq} \times_{\mathcal{O}^{\simeq}}(\mathcal{O}_{/O})^\simeq
      \arrow{r} \arrow{d} & \lim_{E \in \mathcal{O}^{\el}_{O/}}
      \mathcal{P}^{\simeq} \times_{\mathcal{O}^{\simeq}} (\mathcal{O}_{/E})^\simeq
      \arrow{d} \\
      (\mathcal{O}_{/O})^\simeq
      \arrow{r}  & 
      \lim_{E \in \mathcal{O}^{\el}_{O/}} (\mathcal{O}_{/E})^\simeq;
    \end{tikzcd}
    \]
    this is cartesian \IFF{} the map on fibers
    over each $\omega \colon O' \actto O$ is an equivalence. This map takes the form
    \begin{equation}
      \label{eq:fbrsgpd}
     \mathcal{P}_{O'}^{\simeq} \to \lim_{\alpha \colon O \intto E \in
        \mathcal{O}^{\el}_{O/}} \mathcal{P}_{\omega_{\alpha,!}O'}^{\simeq}.      
    \end{equation}
    and is induced by the cocartesian transport along the inert morphisms 
    $\omega_{\alpha} 
    \colon  O' \xintto{} \omega_{\alpha !} O' $ as in \cref{omega_alpha!}.
    This is exactly the map from condition \ref{wsf3}, so the square of 
    $\infty$-groupoids is cartesian \IFF{} condition \ref{wsf3} holds.

    Now consider the square of mapping spaces for two 
    objects $(P, \phi \colon \pi(P) \actto O)$ and 
    $(P', \phi' \colon \pi(P') \actto O) \in \calP \times_\calO \calO_{/O}^\act$:
      \begin{equation}
        \label{eq:fibractmapsq}
        \begin{tikzcd}
          \Map_{\mathcal{P} \times_{\mathcal{O}}\mathcal{O}^{\act}_{/O}}((P,\phi), (P',\phi')) 
          \arrow{r} \arrow{d} &
          \lim_{\alpha\colon  O \intto E \in \calO_{O/}^\el}
          \Map_{\mathcal{P} \times_{\mathcal{O}}\mathcal{O}^{\act}_{/E}}((P,\phi), (\phi'_{\alpha,!}P', (\alpha \circ \phi')^\act)) 
          \arrow{d} \\
          \Map_{\mathcal{O}_{/O}^\act}(\phi, \phi') \arrow{r} &
          \lim_{\alpha\colon  O \intto E \in \calO_{O/}^\el} 
          \Map_{\mathcal{O}_{/E}^\act}(\phi, (\alpha \circ \phi')^\act).
        \end{tikzcd}
    \end{equation}
    A point in $\Map_{\calO_{/O}^\act}(\phi, \phi')$ is 
    a (necessarily active) morphism $f \colon \pi(P) \actto \pi(P')$ 
    together with a homotopy $\phi \simeq \phi' \circ f$.
    To compute the fiber of the vertical maps at this point, 
    note that the mapping space in $\calP \times_\calO \calO^\act_{/O}$ can be computed as: 
    \[\Map_{\mathcal{P}
      \times_{\mathcal{O}}\mathcal{O}^{\act}_{/O}}((P,\phi),
    (P',\phi')) \simeq \Map_{\mathcal{P}}(P,P')
    \times_{\Map_{\mathcal{O}}(\pi(P), \pi(P'))}
    \Map_{\mathcal{O}^{\act}_{/O}}(\phi, \phi'),\]
    Hence the map on the vertical fibers of the square is given by
    \begin{equation}
      \label{eq:fbrsmap}
        \Map_\calP^f(P, P') \longrightarrow 
        \lim_{\alpha\colon O \intto E \in \calO_{O/}^\el} \Map_\calP^{\phi'_\alpha \circ f}(P, \phi'_{\alpha,!} P'),
      \end{equation}
      where the superscripts indicate fibers over maps in $\mathcal{O}$.
    This agrees with the map on fibers over $f$ of the square in condition \ref{wsf2}.
    Therefore condition \ref{wsf2} implies that the square of mapping spaces is a pullback.
    
    However, we have not shown the converse yet, 
    because we have only considered the fibers in \cref{eq:wsegfibmapsq} over morphisms
    $f \in \Map_\calO(O_0, O_1)$ that are active. 
    Let us now assume that the square of mapping spaces \cref{eq:fibractmapsq} is cartesian.
    For a general morphism $O_0 \to O_1$ we can find an inert-active factorization
    $O_0 \xintto{j} Q \xactto{g} O_1$. 
    Since $j$ is inert we can find a cocartesian lift $\widetilde{j} \colon P_0 \to j_!P_0$ 
    and by virtue of this being cocartesian, pre-composition with $\widetilde{j}$ 
    induces the vertical equivalences in the following diagram:
    \[
    \begin{tikzcd}
        \Map_\calP^g(j_! P, P') \ar[r] \ar[d, "\simeq", "(-)\circ \widetilde{j}"'] & 
        \lim_{\alpha\colon  O \intto E \in \calO_{O/}^\el} \Map_\calP^{\phi'_\alpha \circ g}(j_! P, \phi'_{\alpha,!} P') 
        \ar[d, "\simeq"', "(-) \circ \widetilde{j}"] \\
        \Map_\calP^{g \circ j}(P, P') \ar[r] &
        \lim_{\alpha\colon  O \intto E \in \calO_{O/}^\el} \Map_\calP^{\phi'_\alpha \circ g \circ j}(P, \phi'_{\alpha,!} P'). 
    \end{tikzcd}
    \]
    Since $g$ is active, the previous argument shows that the top map is an equivalence.
    Hence the bottom map is an equivalence and as $f = g \circ j$ was arbitrary
    this shows that condition \ref{wsf2} is implied.
\end{proof}

The conditions in \cref{propn:fibrous-axioms} are reminiscent of
Lurie's definition of an \iopd{} \cite{HA}. Note, however, that in
conditions \ref{wsf2} and \ref{wsf3} we need to consider \emph{all}
active maps in $\mathcal{O}$, while Lurie's definition of \iopds{}, or
the definition of \emph{weak Segal fibrations} in \cite{patterns1},
only involve the conditions corresponding to identity maps. If the
base pattern is \emph{sound}, however, the conditions for all active
maps are implied by this special case:
\begin{propn}\label{propn:fibrous=WSF}
    Suppose $\mathcal{O}$ is a sound pattern.
    Then a functor $\pi \colon \mathcal{P} \to \mathcal{O}$ is a fibrous $\mathcal{O}$-pattern 
    \IFF{} it is a weak Segal $\calO$-fibration
    in the sense of \cite[Definition 9.6]{patterns1},
    \ie{} the conditions of \cref{propn:fibrous-axioms} hold whenever $\omega$ is an identity morphism.
    Concretely:
  \begin{enumerate}[(1)]
  \item\label{wsf1-sound} $\mathcal{P}$ has all $\pi$-cocartesian
    lifts of
    inert morphisms in $\mathcal{O}$.
  \item\label{wsf2-sound} For every $O_1 \in \calO$, the functor
    \[ 
        \mathcal{P}_{O_1}^\simeq \to 
        \lim_{\alpha \colon O_1 \to E \in \mathcal{O}^{\el}_{O_1/}} \mathcal{P}_{E}^\simeq,
    \]
    induced by cocartesian transport along 
    $\alpha \colon O_1 \intto E$ is an equivalence.
  \item\label{wsf3-sound} For all $O_{0},O_1 \in \mathcal{O}$,
    and all objects $X_0 \in \mathcal{P}_{O_0}, X_1 \in \mathcal{P}_{O_1}$,
    the commutative square
    \[
      \begin{tikzcd}
        \Map_{\mathcal{P}}(X_0,X_1) \arrow{r} \arrow{d} & \lim_{\alpha
          \colon O_1 \intto E \in \mathcal{O}^{\el}_{O_1/}}
        \Map_{\mathcal{P}}(X_0, \alpha_! X_1) \arrow{d} \\
        \Map_{\mathcal{O}}(O_0,O_1) \arrow{r} & \lim_{\alpha
          \colon O_1 \intto E \in \mathcal{O}^{\el}_{O_1/}}
        \Map_{\mathcal{O}}(O_0, E)
      \end{tikzcd}
    \]
    is cartesian. 
  \end{enumerate}
\end{propn}

\begin{remark}
  In \cite{patterns1} (and \cite{HA}), the analogue of condition \ref{wsf2-sound} says
  that the functor
    \[ 
        \mathcal{P}_{O_1} \to 
        \lim_{\alpha \colon O_1 \to E \in \mathcal{O}^{\el}_{O_1/}} \mathcal{P}_{\alpha_! O_2'}
    \]
    is an equivalence, rather than that the underlying map of \igpds{}
    is one. However, it follows from  \ref{wsf3-sound} that this
    functor gives an equivalence on mapping spaces, \ie{} it is already
    fully faithful, and so it is an equivalence \IFF{} it is an
    equivalence on underlying \igpds{}. In fact, it would suffice in
    \ref{wsf2-sound} to assume that the map is merely surjective on $\pi_{0}$.
\end{remark}

\begin{proof}[Proof of \cref{propn:fibrous=WSF}]
    Suppose $\pi\colon \calP \to \calO$ is a weak Segal fibration.
    Consider the functor $F \colon \calO_{O_1/}^\xint \to \calO^\xint \to \calS$ 
    defined by $F(O_1 \intto O_2) := \calP_{O_2}^\simeq$ and cocartesian transport along inerts. 
    The natural transformation $\eta\colon F \Rightarrow *$ to the terminal functor
    satisfies the conditions of \cref{lem:sound-stronger-Segal}.
    The conclusion of the lemma tells us that \ref{wsf2-sound}
    holds for all $\omega \colon O_1 \actto{} O_2$.
    
    For property \ref{wsf3-sound}, fix $X_0, X_1 \in \calP$ 
    with $\pi(X_0) = O_0$ and $\pi(X_1) = O_1$.
    Then cocartesian transport along inerts defines a functor
    \[
        F \colon  \calO_{O_1/}^\xint \to \calS, \quad 
        (\phi \colon O_1 \intto{} O_2) \mapsto \Map(X_0, \phi_!X_1)
    \]
    and this admits a canonical natural transformation 
    to the functor $G(\phi \colon O_1 \intto{} O_2) := \Map(O_0, O_2)$.
    Applying lemma \ref{lem:sound-stronger-Segal}
    to $\eta \colon F \Rightarrow G$ shows that \ref{wsf3-sound} 
    holds for all $\omega \colon O_1 \actto O_2$.
\end{proof}

\begin{ex}\label{exs:fibrous-patterns-generalize-operads}
  Fibrous $\xF_{*}$-patterns are precisely (symmetric) \iopds{} as
  defined in \cite{HA}, 
  while fibrous $\xF_{\ast}^{\natural}$- patterns are generalized (symmetric) \iopds{}.
  Similarly, fibrous $\simp^{\op,\flat}$- and
  $\simp^{\op,\natural}$-patterns are non-symmetric (or planar) \iopds{} and
  generalized non-symmetric \iopds{}, respectively.
\end{ex}

\begin{observation}\label{cor:cocartesian-fibrous}
  For a sound pattern $\mathcal{O}$ we can also describe the fibrous
$\mathcal{O}$-patterns that are cocartesian fibrations as the
unstraightenings of Segal $\mathcal{O}$-\icats{}, \ie{} as the
\emph{Segal $\mathcal{O}$-fibrations} of  \cite[Definition
9.1]{patterns1}. This is easy to check directly, but it is also a special case of 
\cref{lem:relative-verySegal-iff-unstraighting-is-relative-fibrous} (taking $Y = \ast$), which we will prove below.
\end{observation}

Fibrous $\calO$-patterns admit a canonical pattern structure, which we now introduce:
\begin{defn}\label{defn:pattern-structure-on-fibrous}
  Suppose $\pi \colon \mathcal{P} \to \mathcal{O}$ is a fibrous
  $\calO$-pattern. We say a morphism in $\mathcal{P}$ is \emph{inert}
  if it is $\pi$-cocartesian and lies over an inert morphism in
  $\mathcal{O}$, and \emph{active} if it just lies over an active
  morphism in $\mathcal{O}$. The inert and active morphisms then form
  a factorization system on $\mathcal{P}$ by \cite{HA}*{Proposition
    2.1.2.5}, and we give $\mathcal{P}$ an algebraic pattern structure
  with this factorization system by taking the elementary objects to
  be all those that lie over elementary objects in $\mathcal{O}$.
\end{defn}

\begin{defn}
  A \emph{morphism of fibrous $\calO$-patterns} is a
  commutative triangle
  \[
    \begin{tikzcd}
      \mathcal{P} \arrow{rr}{f} \arrow{dr}[swap]{\pi} & &  \mathcal{P}'
      \arrow{dl}{\pi'} \\
       & \mathcal{O},
    \end{tikzcd}
  \]
  where $\pi$ and $\pi'$ are fibrous $\calO$-patterns and $f$ is a
  morphism of algebraic patterns. It is immediate from the definition
  of the pattern structures that for this it suffices to require that
  $f$ preserves inert morphisms. We write $\Fbrs(\mathcal{O})$ for the
  full subcategory of $\AlgPatt_{/\mathcal{O}}$ whose objects are the
  fibrous $\calO$-patterns; this is equivalently a full subcategory of
  $\CatIOintcoc$.
\end{defn}

\begin{lemma}\label{lem:fbrslim}
  The inclusion $\Fbrs(\mathcal{O}) \hookrightarrow \CatIOintcoc$
  preserves limits and $\kappa$-filtered colimits where $\kappa$ is a
  regular cardinal such that $\mathcal{O}^{\el}_{O/}$ is
  $\kappa$-small for all $O \in \mathcal{O}$. Limits and
  $\kappa$-filtered colimits of $\mathcal{O}$-fibrous patterns can
  therefore be computed in $\CatIsl{\calO}$.
\end{lemma}
\begin{proof}
  By \cref{obs:Lcoclim} the forgetful functor $\CatIOintcoc \to \CatIsl{\mathcal{O}}$ preserves limits and $\kappa$-filtered colimits, and is also conservative.
  It therefore suffices to observe that the commutative square that is required to be cartesian for an
  object of $\CatIOintcoc$ to be a fibrous $\mathcal{O}$-pattern
  commutes with limits and $\kappa$-filtered colimits of \icats{}. 
  Since a limit or filtered colimit of cartesian squares in $\CatI$ is again
  cartesian, this implies the result.
\end{proof}

\begin{observation}\label{rmk:pattern-iso-Segal}
  If $\pi \colon \mathcal{P} \to \mathcal{O}$ is a fibrous $\calO$-pattern, 
  then for every object $\overline{X} \in \mathcal{P}$ over $X$ in $\mathcal{O}$, the functor
  \[ \mathcal{P}^{\el}_{\overline{X}/} \to \mathcal{O}^{\el}_{X/} \]
  is an equivalence.
  Indeed, since $\calP^\xint \to \calO^\xint$ is a cocartesian fibration
  the functor $\calP_{\overline{X}/}^\xint \to \calO_{X/}^\xint$ is 
  an equivalence, 
  and the above functor is obtained by restricting to the full 
  subcategories of elementary objects.
  In particular, $\pi$ is an iso-Segal morphism. More generally, if $f
  \colon \mathcal{P} \to \mathcal{Q}$ is a morphism of fibrous
  $\mathcal{O}$-patterns, then $f$ induces an equivalence
  \[ \mathcal{P}^{\el}_{\overline{X}/} \isoto
    \mathcal{Q}^{\el}_{f(\overline{X})/} \]
  for the same reason, so that $f$ is also an iso-Segal morphism.
\end{observation}

\begin{lemma}\label{lem:fbrssound}
  Suppose $\mathcal{O}$ is a sound pattern and $\pi \colon \mathcal{P}
  \to \mathcal{O}$ is $\calO$-fibrous. Then $\mathcal{P}$ is also a
  sound pattern. Moreover, if $\mathcal{O}$ is soundly extendable,
  then so is $\mathcal{P}$.
\end{lemma}
\begin{proof}
    As we just observed that $\pi$ is iso-Segal, soundness follows from \cref{lem:iso-Segal-sound}.
    
    Now assume $\calO$ is soundly extendable.
    Then, by \cref{rmk:fbrs-over-extendable}, the functor
    \[
    \calP \times_\calO \calO^{\act}_{/Y} \to 
    \lim_{E' \in \calO^{\el}_{Y/}} \calP \times_\calO \calO^{\act}_{/E'}
    \] 
    is an equivalence.
    Since any morphism in $\calP$ that is mapped to an active morphism in $\calO$ is active by definition and active morphisms satisfy cancellation, we have that
    $\calP \times_\calO \calO^\act_{/Y} = \calP^\act \times_{\calO^\act} \calO^\act_{/Y}$.
    Consider the case where $Y = \pi(X)$ for $X \in \calP$.
    Since $\calP \to \calO$ is an equivalence on elementary slices, we can rewrite the limit on the right-hand side as a limit over $E \in \calP^\el_{X/}$ and set $E' := \pi(E)$:
    \[
    \calP^\act \times_{\calO^\act} \calO^{\act}_{/\pi(X)} \xrightarrow{\simeq} 
    \lim_{E \in \calP^{\el}_{X/}} \calP^\act \times_{\calO^\act} \calO^{\act}_{/\pi(E)}.
    \]
    Now, passing to the over-category of $(X, \id_{\pi(X)})$ we obtain an equivalence:
    \[
    \calP_{/X}^\act \simeq 
    (\calP^\act \times_{\calO^\act} \calO^{\act}_{/\pi(X)})_{/(X,\id_{\pi(X)})} \xrightarrow{\simeq}  
    \lim_{E \in \calP^{\el}_{X/}} (\calP^\act \times_{\calO^\act} \calO^{\act}_{/\pi(E)})_{/(E, \id_{\pi(E)})}
    \simeq \lim_{E \in \calP^{\el}_{X/}} \calP_{/E}^\act,
    \] 
    which shows that $\calP$ is soundly extendable.
\end{proof}

\begin{propn}\label{propn:fibrous-over}
  Suppose we have a commutative triangle of algebraic patterns
  \[
    \begin{tikzcd}
      \mathcal{Q} \arrow{rr}{F} \arrow{dr}[swap]{q} & & \mathcal{P}
      \arrow{dl}{p} \\
       & \mathcal{O},
    \end{tikzcd}
  \]
  where $\mathcal{P}$ is $\mathcal{O}$-fibrous. 
  Assume further that $\calO$ is sound. 
  Then $\mathcal{Q}$ is $\mathcal{O}$-fibrous \IFF{} it is $\mathcal{P}$-fibrous.
\end{propn}
\begin{proof}
    By \cref{lem:fbrssound} $\calP$ is also sound, so we may use the characterisation from \cref{propn:fibrous=WSF}.
    
  Any inert morphism
  $ \pi \colon P \to P'$ in $\mathcal{P}$ is cocartesian over an inert
  morphism $\omega \colon O
  \to O'$ in $\mathcal{O}$; if $\phi \colon Q \to Q'$ is an inert morphism over
  $\omega$ in $\mathcal{Q}$ such that $F(Q) \simeq P$, then we have $F(\phi) \simeq \pi$ since
  $F$ preserves inert morphisms and $\pi$ is the unique inert morphism
  over $\omega$ with source $P$. 
  It now follows from \cite[Proposition 2.4.1.3]{HTT} that 
  $\phi \colon Q \to Q'$ is $F$-cocartesian \IFF{} it is $q$-cocartesian.
  Thus condition \ref{wsf1-sound} in \cref{propn:fibrous=WSF} holds for $F$ \IFF{} it holds for $q$.
  
  Assuming this holds, then for 
  $Q, Q' \in \calQ$, $P = F(Q)$, $P' = F(Q')$ and $O = q(Q)$, $O' = q(Q')$, 
  we have a commutative diagram
  \[
    \begin{tikzcd}
      \Map_{\mathcal{Q}}(Q, Q') \arrow{r} \arrow{d} &
      \lim_{\alpha \in \mathcal{O}^{\el}_{O'/}} \Map_{\mathcal{Q}}(Q,
      \alpha_! Q')  \arrow{d} \\
      \Map_{\mathcal{P}}(P, P') \arrow{r} \arrow{d} & \lim_{\alpha \in \mathcal{O}^{\el}_{O'/}} \Map_{\mathcal{P}}(P,
      \alpha_! P')
      \arrow{d}  \\
      \Map_{\mathcal{O}}(O, O') \arrow{r} &  \lim_{(\alpha\colon O' \intto E) \in
        \mathcal{O}^{\el}_{O'/}} \Map_{\mathcal{O}}(O, E).
    \end{tikzcd}
  \]
  Here the bottom square is cartesian since $\mathcal{P}$ is
  $\mathcal{O}$-fibrous, so the top square is cartesian \IFF{} the
  outer square is cartesian. But since $p$ is an iso-Segal morphism (by \cref{rmk:pattern-iso-Segal}) we
  can rewrite the top square as
  \[
    \begin{tikzcd}
      \Map_{\mathcal{Q}}(Q, Q') \arrow{r} \arrow{d} &
      \lim_{\beta \in \mathcal{P}^{\el}_{P'/}} \Map_{\mathcal{Q}}(Q,
      \beta_! Q') \arrow{d} \\
      \Map_{\mathcal{P}}(P, P') \arrow{r}  & \lim_{(\beta\colon P' \to E') \in \mathcal{P}^{\el}_{P'/}} \Map_{\mathcal{P}}(P,
      E'),
    \end{tikzcd}
  \]
  and so we have that condition \ref{wsf3-sound} in \cref{propn:fibrous=WSF} holds for $F$ \IFF{} it holds for
  $q$.
  The proof for \ref{wsf2-sound} is similar.
\end{proof}

\begin{cor}\label{cor:compeqWSF}
  If $\calO$ is sound and $\pi \colon \mathcal{P} \to \mathcal{O}$ exhibits $\mathcal{P}$
  as an $\mathcal{O}$-fibrous pattern, then composition with $\pi$ gives a functor
  \[ \pi_{!} \colon \Fbrs(\mathcal{P}) \to \Fbrs(\mathcal{O}),\]
  and this induces an equivalence
  \[ \Fbrs(\mathcal{P}) \isoto \Fbrs(\mathcal{O})_{/\mathcal{P}}.\]
\end{cor}

\begin{ex}\label{ex:fibrous-for-operads}
    Let $(\pi\colon  \calO \to \xF_*) \in \Opd_\infty$ be an $\infty$-operad
    in the sense of Lurie, \ie{} a fibrous $\xF_*$-pattern.
    Applying \cref{cor:compeqWSF} we obtain an equivalence:
    \[
        \Fbrs(\calO) \isoto \Fbrs(\xF_*)_{/\calO} = \Opd_{\infty/\calO}
    \]
    so fibrous $\calO$-patterns are simply $\infty$-operads over $\calO$.
\end{ex}

\begin{lemma}\label{lem:pullback-of-fibrous}
  Suppose $f \colon \mathcal{O} \to \mathcal{P}$ is a strong Segal morphism.
  Then pullback along $f$ restricts to a functor 
  \[ 
  f^{*} \colon \Fbrs(\mathcal{P}) \to \Fbrs(\mathcal{O}) , \qquad
  (\pi \colon \calF \to \calP) \mapsto (f^*\pi \colon \calF \times_\calP \calO \to \calO).
  \]
\end{lemma}
\begin{proof}
    Suppose $\pi \colon \mathcal{F} \to \mathcal{P}$ is a
    $\mathcal{P}$-fibrous pattern. Condition \ref{fbrs1} in
    \cref{defn:fibrous} for $f^{*}\mathcal{F}$ follows from the usual description of cocartesian
    morphisms in a pullback, since $f$ preserves inert morphisms. 
    To prove \ref{fbrs2}, we observe that 
    $f^{*}\mathcal{F} \times_{\mathcal{O}} \Aract(\calO) 
    \simeq \mathcal{F} \times_{\mathcal{P}} \Aract(\calO)$,
    so that we have a cartesian square
    \[
      \begin{tikzcd}
        f^{*}\mathcal{F} \times_{\mathcal{O}} \Aract(\calO) \arrow{r} \arrow{d} &
        \mathcal{F} \times_{\mathcal{P}} \Aract(\calP) \times_\calP \calO
    \arrow{d} \\
    \Aract(\calO) \arrow{r} & \Aract(\calP) \times_{\calP} \calO
      \end{tikzcd}
    \]
    of cocartesian fibrations over $\calO$.
    Straightening yields the cartesian square:
    \[
      \begin{tikzcd}
        \StOint(f^*\calF) \arrow{r} \arrow{d} &
        \St{\calP}{\xint}(\calF) \circ f \arrow{d} \\
        \Afun{\mathcal{O}} \arrow{r} & 
        \Afun{\mathcal{P}} \circ f
      \end{tikzcd}
    \]
    of functors $\calO \to \CatI$.
    By \cref{obs:fbrsrelseg} the natural transformation $\St{\calP}{\xint}(\calF) \to \calA_\calP$ is a relative $\calP$-Segal \icat.
    This remains true after precomposing with $f$ 
    (by \cref{rmk:strong Seg pres rel}, since $f$ is strong Segal) .
    Hence the right vertical map in the square is a relative $\calO$-Segal \icat and 
    by \cref{rel Segal pullback} so is the left vertical arrow.
    Using \cref{obs:fbrsrelseg} again we see that $f^*\calF$ is fibrous.
\end{proof}

\begin{ex}
    The morphism
    $\mathfrak{c} \colon \simp^{\op,\flat} \to \xF_{*}$
    from \cref{ex:Dop->F-isoSegal}
    is iso-Segal and hence \cref{lem:pullback-of-fibrous} 
    shows that pulling back along it defines a functor:
    \[
        \mathfrak{c}^*\colon  \Fbrs(\xF_*) \longrightarrow \Fbrs(\simp^{\op,\flat}).
    \]
    Under the identifications of
    \cref{exs:fibrous-patterns-generalize-operads}
    this is exactly the forgetful functor from (symmetric) $\infty$-operads
    to non-symmetric $\infty$-operads.
  \end{ex}

Finally, let us note that we can lift the comparison of \cref{propn:Segmndcomp} to fibrous
patterns:
\begin{propn}\label{propn:pullback-of-fibrous-along-morita-equivalence}
    Suppose $f \colon \calO \to \calP$ is a strong Segal morphism that satisfies the conditions of \cref{propn:Segmndcomp} and let $\pi\colon  \calQ \to \calP$ be a fibrous pattern.
    Then $\overline{f} \colon f^{\ast}\calQ \to \calQ$ is also a strong Segal morphism that satisfies the conditions of \cref{propn:Segmndcomp} and thus induces an equivalence
      \[ \overline{f}^{*} \colon \Seg_{\calQ}(\mathcal{S}) \isoto \Seg_{f^{*}\calQ}(\mathcal{S}). \]
\end{propn}
\begin{proof}
    Denote by $\pi' \colon \calQ':= f^{\ast} \calQ \to \calO$ the projection map.
    Since $\calQ$ is fibrous and $f$ is strong Segal, it follows from \cref{lem:pullback-of-fibrous} that $\calQ'$ is also fibrous.
    By \cref{rmk:pattern-iso-Segal} we have $\calQ^{\el}_{Q/} \simeq \calP^{\el}_{\pi(Q)/}$ and similarly for $\calQ'$ and $\calO$.
    The map $(\calQ')^{\el}_{Q/} \to \calQ^{\el}_{\overline{f}(Q)/}$ thus identifies with 
    $\calO^{\el}_{\pi'(Q)/} \to \calP^{\el}_{f(\pi'(Q))/}$ which is
    coinitial by the assumption that $f$ is strong Segal.
    We conclude that $\overline{f}$ is strong Segal.
    We proceed by verifying the conditions.
    Condition $(1)$ of \cref{propn:Segmndcomp} is visibly stable under basechange so it remains to check $(2)$.
    Observe that for every object $Q \in \calQ'$ that lies over $O \in \calO$
    we have by \cite[Lemma 5.4.5.4]{HTT} a pullback square of slice \icats{} 
    \[\begin{tikzcd}
	{\calQ'_{/Q}} & \calQ_{/\overline{f}(Q)} \\
	{\calO_{/O}} & \calP_{/f(O)}.
	\arrow[from=1-1, to=1-2]
	\arrow[from=1-2, to=2-2]
	\arrow[from=1-1, to=2-1]
	\arrow[from=2-1, to=2-2]
	\arrow["\lrcorner"{anchor=center, pos=0.125}, draw=none, from=1-1, to=2-2]
    \end{tikzcd}\]
    By assumption the bottom map induces an equivalence on the underlying spaces of active maps and since the square is cartesian the same holds for the top map.
\end{proof}

\subsection{Segal envelopes}\label{sec:segenv}

In this section we will specialize our results from \cref{sec:factenv}
to fibrous $\mathcal{O}$-patterns over an algebraic pattern
$\mathcal{O}$. Recall that we have shown that from the inert--active factorization system on $\mathcal{O}$ we obtain an adjunction
\[ (\blank) \times_{\mathcal{O}} \Aract(\mathcal{O}) \colon
  \CatIOintcoc \rightleftarrows
(\Cat_{\infty/\mathcal{O}}^{\txt{cocart}})_{/\Aract(\mathcal{O})}, \]
where the right adjoint is given by pulling back along the map $\mathcal{O} \to \Aract(\mathcal{O})$ given by the degeneracy $[1] \to [0]$. This can equivalently be interpreted as a 
``straightening--unstraightening'' adjunction
\[
    \StOint \colon
    \CatIOintcoc \rightleftarrows
    \Fun(\calO, \CatI)_{/\calA_\calO}
    :\! \UnOint
\]
in which the left adjoint is fully faithful 
with image the $\calA_\calO$-equifibered functors.

We can immediately identify the image of the full subcategory $\Fbrs(\mathcal{O})$
under this fully faithful functor:
\begin{propn}\label{propn:general-fully-faithful-envelope}
  For any  algebraic pattern $\calO$, the fully faithful functor
  $\StOint$ identifies $\Fbrs(\mathcal{O})$ with the full subcategory
  of $\Fun(\calO,\CatI)_{/\calA_{\calO}}$ spanned by the equifibered
  maps that are also relative Segal objects. In other words, the
  functor $\StOint$ restricts to a fully faithful functor
    \[ \sliceEnv{\calO} := \StOint|_{\Fbrs(\mathcal{O})} \colon \Fbrs(\mathcal{O}) \hookrightarrow
    \Seg^{/\mathcal{A}_{\mathcal{O}}}_{\mathcal{O}}(\CatI) \]
  with image the equifibered objects. Moreover, for any strong Segal morphism $f \colon \mathcal{O} \to \mathcal{P}$, we have a commutative square
  \begin{equation}
    \label{eq:segenvsq}
    \begin{tikzcd}
      \Fbrs(\mathcal{P}) \arrow[hook]{d}{\sliceEnv{\mathcal{P}}} \arrow{r}{f^{*}} & 
      \Fbrs(\mathcal{O}) \arrow[hook]{d}{\sliceEnv{\mathcal{O}}} \\
      \Seg^{/\mathcal{A}_{\mathcal{P}}}_{\mathcal{P}}(\CatI) \arrow{r}{f^{\ostar}} & 
      \Seg^{/\mathcal{A}_{\mathcal{O}}}_{\mathcal{O}}(\CatI)
    \end{tikzcd}
  \end{equation}
  where the functor $f^{\ostar}$ is given by the composite
  \[ \Seg^{/\mathcal{A}_{\mathcal{P}}}_{\mathcal{P}}(\CatI) \xto{f^{*}} \Seg^{/f^{*}\mathcal{A}_{\mathcal{P}}}_{\mathcal{O}}(\CatI) \to \Seg^{/\mathcal{A}_{\mathcal{O}}}_{\mathcal{O}}(\CatI)\]
  of restriction along $f$ and pullback along the natural map $\Afun{\mathcal{O}} \to f^{*}\Afun{\mathcal{P}} $ (\cf{} \cref{rmk:strong Seg pres rel} and \cref{rel Segal pullback}).
\end{propn}
\begin{proof}
  From \cref{obs:fbrsrelseg} we know that an object $\mathcal{P}$ of
  $\CatIOintcoc$ is a fibrous $\mathcal{O}$-pattern \IFF{}
  $\StOint(\mathcal{P})$ is a relative Segal $\mathcal{O}$-object in
  $\CatI$. The commutative square \cref{eq:segenvsq} likewise follows
  by restricting the square \cref{eq:fostarEsq} in
  \cref{obs:basechange-envelope} to full subcategories.
\end{proof}

From this observation we can deduce some pleasant properties of the \icats{} of fibrous patterns:

\begin{cor}\label{cor:fibrouspresentable}
  For any algebraic pattern $\mathcal{O}$, the \icat{}
  $\Fbrs(\mathcal{O})$ is presentable, and fits in a cartesian square
  of fully faithful right adjoints
  \[
    \begin{tikzcd}
      \Fbrs(\mathcal{O}) \ar[r, hook, "{\sliceEnv{\calO}}"] \ar[d, hook] &
      \Seg^{/\Afun{\mathcal{O}}}_{\mathcal{O}}(\CatI)  \ar[d, hook] \\
      \CatIOintcoc \ar[r, hook, "\StOint"] &
      \Fun(\calO^\op, \CatI)_{/\calA_\calO}.
    \end{tikzcd}
  \]
\end{cor}
\begin{proof}
  We know from \cref{propn:general-fully-faithful-envelope} that we
  have the given cartesian square of fully faithful functors; it remains to show that this is a square in $\Pr^R$.
  For the bottom horizontal and right vertical functor we have shown
  this in  \cref{lem:Env-admits-left-adjoint} and \cref{lem:relative-Segal-strongly-reflective},
  respectively.
  It now follows that the rest of the diagram also lies in $\Pr^R$, since
  the diagram is cartesian and by \cite[Theorem 5.5.3.18]{HTT} 
  $\Pr^R$ admits pullbacks and the inclusion $\Pr^R \subset \CatI$ preserves them.
\end{proof}

\begin{cor}\ 
  \begin{enumerate}[(1)]
  \item For any algebraic pattern $\mathcal{O}$, 
  the following functors admit left adjoints:
  \[
  \Fbrs(\mathcal{O}) \hookrightarrow \CatIOintcoc \to \CatIsl{\calO}.
  \]
  \item For any strong Segal morphism $f \colon \mathcal{O} \to
    \mathcal{P}$, the functor $f^{*} \colon \Fbrs(\mathcal{P}) \to
    \Fbrs(\mathcal{O})$ admits a left adjoint.
  \end{enumerate}
\end{cor}
\begin{proof}
    The first claim was shown in \cref{cor:fibrouspresentable} and \cref{obs:Lcoclim}.
    In particular limits and $\kappa$-filtered colimits in $\Fbrs(\calO)$ for appropriate $\kappa$ are computed in $\CatIsl{\calO}$. 
    This implies that $f^* \colon \Fbrs(\calP) \to \Fbrs(\calP)$ preserves limits and $\kappa$-filtered colimits,
    since we know pullback along $f$ preserves limits and filtered colimits as a functor $\CatIsl{\calP} \to \CatIsl{\calO}$.
    Hence the claim follows from the adjoint functor theorem.
\end{proof}

Note that in \cref{propn:general-fully-faithful-envelope} we only
showed that the left adjoint $\StOint$ restricts to a functor from
fibrous patterns to relative Segal objects --- in general the right
adjoint $\UnOint$ does not necessarily
take relative Segal $\mathcal{O}$-\icats{} over $\Afun{\mathcal{O}}$
to fibrous $\mathcal{O}$-patterns. 
However, this is the case if $\calO$ is sound; to see this, we first need a technical lemma:
\begin{lemma}\label{lem:relative-verySegal-iff-unstraighting-is-relative-fibrous}
    Let $\calO$ be a sound algebraic pattern and let $\gamma \colon X \to Y$
    be a morphism in $\Fun(\calO,\CatI)$, with $\Gamma \colon
    \mathcal{X} \to \mathcal{Y}$ denoting its unstraightening. Then the following are equivalent:
    \begin{enumerate}
        \item $\gamma \colon X \to Y$ is a relative Segal object.
        \item $\StOint(\Gamma) \colon \StOint(\mathcal{X}) \to \StOint(\mathcal{Y})$ is a
          relative Segal object, \ie{} the commutative square
    \[\begin{tikzcd}
    		\mathcal{X} \times_\calO \calO^\act_{/O} \ar[d] \ar[r] &
    		\lim_{E \in \calO^{\el}_{O/}} \mathcal{X} \times_\calO \calO^\act_{/E} \ar[d] \\
    		\mathcal{Y} \times_\calO \calO^\act_{/O} \ar[r] & 
    		\lim_{E \in \calO^{\el}_{O/}} \mathcal{Y} \times_\calO \calO^\act_{/E} 
    \end{tikzcd}\]
  is cartesian for all $O \in \mathcal{O}$.
    \end{enumerate}
\end{lemma}
\begin{proof}
    For $O \in \mathcal{O}$, we consider the following commutative diagram:
    \[\begin{tikzcd}[column sep=-1em]
    		\mathcal{X} \times_\calO \calO^\act_{/O} \ar[rrr] \ar[ddr]
                \ar[drr]  & & & 
    		\lim_{E \in \calO^{\el}_{O/}} \mathcal{X} \times_\calO
                \calO^\act_{/E} \ar[drr] \ar[ddr] \\
    		& &  \mathcal{Y} \times_\calO \calO^\act_{/O} \ar[rrr,
                crossing over]
                \ar[dl] &
                & &  
    		\lim_{E \in \calO^{\el}_{O/}} \mathcal{Y} \times_\calO
                \calO^\act_{/E}  \ar[dl] \\
    		& \calO^\act_{/O} \ar[rrr] & & &
    		\lim_{E \in \calO^{\el}_{O/}} \calO^\act_{/E}. 
    \end{tikzcd}\]
  Here all four functors to the bottom row are cocartesian fibrations,
  and the morphisms in the top square preserve cocartesian morphisms.
  We therefore see that condition (2), which asks for the top square
  to be cartesian, is equivalent to all squares of fibers over
  $\omega \colon O' \actto O$ in $\mathcal{O}^{\act}_{/O}$ being
  cartesian. 
  The relevant square of fibers is
    \[\begin{tikzcd}
	X(O') & \lim_{(\alpha\colon O \intto{} E) \in \calO^{\el}_{O/}} X(\omega_{\alpha!}O') \\
	Y(O') & \lim_{(\alpha\colon O \intto{} E) \in \calO^{\el}_{O/}} Y(\omega_{\alpha!}O').
	\arrow[from=1-2, to=2-2]
	\arrow[from=1-1, to=1-2]
	\arrow[from=1-1, to=2-1]
	\arrow[from=2-1, to=2-2]
      \end{tikzcd}\] Considering the special case $\omega = \id_O$ we
    see that (2) implies (1), while to see that the converse holds
    when $\mathcal{O}$ is sound we apply
    \cref{lem:sound-stronger-Segal} with $F=X$ and $G=Y$.
\end{proof}

\begin{propn}\label{prop:right-adjoint-to-envelope}
    If the pattern $\calO$ is sound, then the adjunction of \cref{defn:partial-s/u}
    restricts to an adjunction
    \[ \sliceEnv{\calO} \colon \Fbrs(\mathcal{O}) \adj
      \Seg^{/\mathcal{A}_{\mathcal{O}}}_{\mathcal{O}}(\CatI) :\!  \UnOint. \]
    Moreover, if $f \colon \mathcal{O} \to \mathcal{P}$ is a strong Segal morphism between sound patterns, then in addition to the square \cref{eq:segenvsq} we also have a commutative square
    \begin{equation}
      \label{eq:segradjsq}
    \begin{tikzcd}
      \Seg^{/\mathcal{A}_{\mathcal{P}}}_{\mathcal{P}}(\CatI) \arrow{r}{f^{\ostar}} \arrow{d}[swap]{\Un{\calP}{\xint}}& \Seg^{/\mathcal{A}_{\mathcal{O}}}_{\mathcal{O}}(\CatI) \arrow{d}{\Un{\calO}{\xint}} \\
      \Fbrs(\mathcal{P})  \arrow{r}{f^{*}} & \Fbrs(\mathcal{O}).
    \end{tikzcd}      
    \end{equation}
\end{propn}
\begin{proof}
  We need to show that
  $\Un{\calO}{\xint} \colon \Fun(\calO,\CatI)_{/\calA_\calO} \to
  \CatIOintcoc$ sends $\calA_\calO$-relative Segal objects to fibrous
  $\calO$-patterns. Since we know an object of $\CatIOintcoc$ is
  fibrous \IFF{} its image under $\StOint$ is a relative Segal object,
  it suffices to show that $\StOint \circ \Un{\calO}{\xint}$ preserves
  relative Segal objects. 

  Let $X \to \calA_{\calO}$ be a relative Segal object; then
  $\StOint ( \Un{\calO}{\xint}(X))$ fits into a cartesian square
    \[
        \begin{tikzcd}
            \StOint ( \Un{\calO}{\xint}(X)) \ar[r] \ar[d] & 
            \St{\calO}{\xint} ( \mathrm{Un}_{\calO}(X)) \ar[d] \\
            \calA_{\calO} \ar[r] &
            \St{\calO}{\xint} (\mathrm{Un}_{\calO}(\calA_{\calO}))
        \end{tikzcd}
    \]
    obtained from applying $\St{\calO}{\xint}$ to the cartesian square
    defining $\Un{\calO}{\xint}(X)$. Since relative Segal objects are
    stable under base change by \ref{rel Segal pullback}, it suffices
    to show the right vertical map is a relative Segal object, which
    follows from
    \cref{lem:relative-verySegal-iff-unstraighting-is-relative-fibrous}. 
    The commutative square \cref{eq:segradjsq} follows by restricting the
    square \cref{eq:ostarradjsq} in \cref{obs:basechange-envelope} to
    full subcategories.
\end{proof}

For soundly extendable patterns $\calO$ we can furthermore think of
this adjunction as being induced by one between fibrous patterns and
Segal $\mathcal{O}$-\icats{}:
\begin{thm}\label{thm:envelope-for-soundly-extendable}
  Let $\mathcal{O}$ be a soundly extendable pattern. 
  Then there is an adjunction
  \[ \Env_{\mathcal{O}} \colon \Fbrs(\mathcal{O}) \rightleftarrows
    \Seg_{\mathcal{O}}(\CatI), \]
  where 
  $\Env_{\mathcal{O}}(\mathcal{P})(X) := \calP \times_{\calO}
  \calO^\act_{/X}$
  and the right adjoint is given by unstraightening.
  This induces an adjunction
  \[ 
    \sliceEnv{\calO} \colon
    \Fbrs(\mathcal{O}) \rightleftarrows
    \Seg_{\mathcal{O}}(\CatI)_{/\calA_{\calO}} 
    \]
  where the left adjoint is fully faithful and the image consists of
  the Segal $\mathcal{O}$-\icats{} that are equifibered over
  $\Afun{\mathcal{O}}$.
\end{thm}
\begin{proof}
  It remains to show that the adjunction
  \[ (\blank) \times_{\mathcal{O}} \Aract(\mathcal{O}) \colon
    \CatIOintcoc \rightleftarrows \CatIOcoc \simeq \Fun(\mathcal{O}, \CatI) \]
  from \cref{cor:factadjn}  restricts to an adjunction between
  $\Fbrs(\mathcal{O})$ and $\Seg_{\mathcal{O}}(\CatI)$. Since
  $\Afun{\mathcal{O}}$ is a Segal $\mathcal{O}$-\icat{}, we have by
  \cref{rmk:cancellation-for-relative-Segal-objects} and
  \cref{propn:general-fully-faithful-envelope} that
  the left adjoint takes fibrous patterns to Segal
  $\mathcal{O}$-\icats{}. On the other hand, the right adjoint takes
  the latter to fibrous patterns by \cref{cor:cocartesian-fibrous}.
\end{proof}

\begin{remark}\label{rmk:replete}
    Note that in the context of \cref{thm:envelope-for-soundly-extendable} the right adjoint of $\Env_{\calO}$ is faithful and replete.
    It induces an equivalence between $\Seg_\calO(\CatI)$ and the subcategory of $\Fbrs(\calO)$
    whose objects are cocartesian fibrous patterns and whose morphisms preserve all cocartesian edges.
\end{remark}

\begin{remark}
  If $f \colon \mathcal{O} \to \mathcal{P}$ is a strong Segal morphism between soundly extendable patterns, then pullback/restriction along $f$ gives a commutative square
  \[
    \begin{tikzcd}
      \Seg_{\mathcal{P}}(\CatI) \arrow{r}{f^{*}} \arrow{d} & \Seg_{\mathcal{O}}(\CatI) \arrow{d} \\
    \Fbrs(\mathcal{P}) \arrow{r}{f^{*}} & \Fbrs(\mathcal{O}).      
    \end{tikzcd}
  \]
  Note, however, that the corresponding Beck--Chevalley transformation is usually not invertible, so we have to slice over $\Afun{\mathcal{P}}$ and $\Afun{\mathcal{O}}$ to get a commutative square of envelopes
    \begin{equation}
    \label{eq:extsegenvsq}
    \begin{tikzcd}
      \Fbrs(\mathcal{P}) \arrow{d}[swap]{\sliceEnv{\mathcal{P}}} \arrow{r}{f^{*}} & \Fbrs(\mathcal{O}) \arrow{d}{\sliceEnv{\mathcal{O}}} \\
      \Seg_{\mathcal{P}}(\CatI)_{/\mathcal{A}_{\mathcal{P}}} \arrow{r}{f^{\ostar}} & \Seg_{\mathcal{O}}(\CatI)_{/\mathcal{A}_{\mathcal{O}}}
    \end{tikzcd}
  \end{equation}
  as a special case of \cref{eq:segenvsq}.
\end{remark}

\subsection{Examples of Segal envelopes}\label{subsec:ex-of-envelopes}

\begin{ex}\label{ex:Env-for-xF}
  For the soundly extendable pattern $\xF_{*}$, we know that fibrous
  patterns are exactly $\infty$-operads, while
    Segal $\xF_*$-\icats{} are symmetric monoidal \icats{}; here
    $\calA_{\xF_{*}}$ is the symmetric monoidal category $\xF^\amalg$
    of finite sets under disjoint union.
    Hence \cref{thm:envelope-for-soundly-extendable} yields an adjunction
    \[ \Env_{\xF_*}^{/\xF^\amalg} \colon   
    \OpdI = \Fbrs(\xF_{*}) \rightleftarrows 
    \Seg_{\xF_{*}}(\CatI)_{/\calA_{\xF_{*}}} = \CMon(\CatI)_{/(\xF,\amalg)}.
    \]
    The left adjoint is fully faithful and a symmetric monoidal functor
    $\pi \colon (\calC,{\otimes}) \to (\xF,\amalg)$ lies in the essential image if and only if
    it is equifibered. 
    This means that the following square is cartesian for all maps
    $\omega \colon X \to Y$ in $\xF$:
    \[
    \begin{tikzcd}
      \calC^X \ar[r, "\omega_\otimes"] \ar[d, "\pi^X"'] & 
      \calC^Y \ar[d, "\pi^Y"] \\
      \xF^X \ar[r, "\omega_\amalg"] &
      \xF^Y.
    \end{tikzcd}
    \]
    Here the horizontal functors tensor over fibers of $\omega$.
    In fact, it follows by taking products and pasting pullback diagrams%
    \footnote{
        See \cref{lem:G-eqf} for an elaboration of this argument.
    }
    that it suffices to check the case of $\omega \colon  \{1,2\} \to \{1\}$. 
\end{ex}

\begin{observation}\label{obs:HKcomp}
  The essential image of the sliced envelope functor
  $\Env_{\xF_*}^{/\xF^\amalg} \colon \OpdI \hookrightarrow
  \CMon(\CatI)_{/(\xF,\amalg)} $ was first described in
  \cite{iopdprop}, but the characterization there looks at first
  glance quite different from ours. Let us therefore compare these
  two descriptions:

  For a symmetric monoidal functor
  $\pi \colon \calC \to \xF$, let us write $\calC_{(1)} \subset \calC$
  for the full subcategory of those $x \in \calC$ with $|\pi(x)| =
  1$. Then the characterization of \cite{iopdprop} is that the
  essential image consists of those $\pi$ that satisfy the following
  pair of conditions:
  \begin{enumerate}
  \item Every object $x \in \calC$ is equivalent to
    $x_1 \otimes \dots \otimes x_n$ for some $x_i \in \calC_{(1)}$.
  \item For every $n,m \ge 0$ and any two tuples $x_1,\dots,x_m \in \calC_{(1)}$ and $y_1, \dots, y_n \in \calC_{(1)}$, the canonical map
    \[
      \coprod_{\phi\colon m \rightarrow n } \prod_{i=1}^n\Map_{\calC}\left(\bigotimes_{j \in \phi^{-1}(i)} x_j, y_i\right) 
      \to
      \Map(\otimes_{j=1}^m x_j, \otimes_{i=1}^n y_i)
    \]
    is an equivalence.
  \end{enumerate}
  These conditions must be equivalent to our equifiberedness
  condition since they describe the same full subcategory. 
  To check this more explicitly, we consider the functor
    \[
      D_n\colon \calC^n \to \xF^n \times_{\xF} \calC,
    \]
    which is an equivalence for all $n$ if and only if
    $p\colon \calC \to \xF$ is equifibered.  The functor $D_n$ is
    essentially surjective if and only if for any $x \in \calC$ and a
    decomposition $\pi(x) = A_1 \amalg \dots \amalg A_n$ there is a
    decomposition $x = x_1 \otimes \dots \otimes x_n$ such that
    $\pi(x_i) \cong A_i$ compatibly with the decomposition.  By
    choosing the trivial decomposition with $|A_i| = 1$ this recovers
    condition (1).  Conversely, given condition (1) we can decompose
    $x$ as $\otimes_{a \in \pi(x)} y_a$ and then find the desired $x_i$
    as $x_i = \otimes_{a \in A_i} y_a$.
    
    To see that the full faithfulness of the $D_{n}$'s corresponds to
    condition (2), we first observe that in the presence of condition (1) we can replace
    condition (2) with the following:
    \begin{itemize}
        \item[(2$'$)] For every $n \ge 0$ and any two tuples $z_1,\dots,z_n \in \calC$ and $y_1, \dots, y_n \in \calC$, the canonical map
        \[
            \prod_{i=1}^n\Map_{\calC}(z_i, y_i) 
            \to
            \coprod_{(\phi_i\colon \pi(z_i) \rightarrow \pi(y_i))}
            \Map^{\coprod_{i=1}^n\phi_i}(\otimes_{i=1}^n z_i, \otimes_{i=1}^n y_i)
        \]
        is an equivalence.  Here we write $\Map_\calC^\phi(a,b)$ for
        the fiber of $\Map_\calC(a,b)$ over some
        $\phi\colon \pi(a) \to \pi(b)$.
    \end{itemize}
    To relate this to condition (2), first decompose $y_i$ using
    condition (1) and use $2$-out-of-$3$ to reduce to the case where
    $|\pi(y_i)|=1$. Then write
    $z_i = \otimes_{j \in \phi^{-1}(i)} x_j$ and argue as in
    \cite[Remark 2.4.8]{iopdprop}.
    
    Now we can observe that $D_n$ is fully faithful if and only if condition (2$'$) holds:
    indeed, the mapping space in $\xF^n \times_{\xF} \calC$ can be described as
    \begin{align*}
        \Map_{\xF^n \times_{\xF} \calC}&((x, \pi(x) = A_1 \amalg \dots \amalg A_n), (y, \pi(y) = B_1 \amalg \dots \amalg B_n))\\
        &\simeq
        \Map_{\xF^n}( (A_i), (B_i)) \times_{\Map_\xF(\pi(x), \pi(y))} \Map_\calC(x,y)\\
        &\simeq \coprod_{(\phi_i\colon A_i \to B_i)} \Map^{\coprod \phi_i}(x, y).
    \end{align*}
    Applying this to the images of $(x_1,\dots,x_n)$ and $(y_1,\dots,y_n)$ 
    under $\calC^n \to \xF^n \times_\xF \calC$ yields the desired form.
    
    It is interesting to note that while in condition (2) we need to quantify over all $n,m \ge 0$,
    in condition (2$'$) it suffices to consider only the case $n=2$ as all other cases can be obtained inductively.
    This works because the objects $z_i$ and $y_i$ in condition (2$'$) are themselves allowed to be composite.
\end{observation}

\begin{ex}\label{ex:Delta-flat-fibrous-envelope}
    For the soundly extendable pattern $\simp^{\op,\flat}$ fibrous patterns are non-symmetric $\infty$-operads,
    while Segal $\simp^{\op,\flat}$-\icats{} are monoidal \icats{}.
    We therefore denote
    $\OpdI^{\ns}:=\Fbrs(\simp^{\op,\flat})$ and $\Mon(\CatI):=\Seg_{\simp^{\op,\flat}}(\CatI)$.
    The Segal $\simp^{\op,\flat}$-category $\calA_{\simp^{\op,\flat}}$
    is equivalent to the category $\simp_+$ of finite (possibly empty) linearly ordered sets,
    with the monoidal structure given by concatenation. 
    The envelope functor $\Env^{/\simp_+}_{\simp^{\op,\flat}}$ can then be interpreted as a fully faithful embedding:
    \[\Env^{/\simp_+}_{\simp^{\op,\flat}} \colon \OpdI^{\ns} \hookrightarrow \Mon(\CatI)_{/\simp_+}\]
    Similarly to \cref{ex:Env-for-xF}
    we can describe the essential image as those monoidal functors 
    $\pi:\calV \to \simp_+$ 
    for which the following natural square is cartesian:
    \[\begin{tikzcd}
	{\calV \times \calV} & \calV \\
	{\simp_+ \times \simp_+} & {\simp_+}.
	\arrow["\otimes", from=2-1, to=2-2]
	\arrow["\pi", from=1-2, to=2-2]
	\arrow["\pi"', from=1-1, to=2-1]
	\arrow["\otimes", from=1-1, to=1-2]
    \end{tikzcd}\]
\end{ex}

\begin{ex}
  For the soundly extendable pattern $\simp^{\op,\natural}$, fibrous
  patterns are generalized non-symmetric \iopds{} as defined in
  \cite{enr}, while Segal $\simp^{\op,\natural}$-\icats{} are category
  objects in $\CatI$, i.e.\ double \icats{}.  We thus write
  $\Opd_{\infty}^{\gen,\ns}:=\Fbrs(\simp^{\op,\natural})$ and
  $\Dbl := \Seg_{\simp^{\op,\natural}}(\CatI)$. We may regard \itcats{} (in the form of complete 2-fold Segal spaces) as those double \icats{} $\mathcal{X}_{\bullet}$ such that $\mathcal{X}_{0}$ is an \igpd{}
  and which satisfy a completeness condition. In particular, the Segal
  $\simp^{\op,\natural}$-\icat{} $\calA_{\simp^{\op,\natural}} \simeq \calA_{\simp^{\op,\flat}}$
  may be thought of as the one-object \itcat{} $\mathfrak{B} \simp_+$ where the endomorphisms of the single object are $\simp_{+}$, with the monoidal structure corresponding to composition.
    The envelope functor $\Env^{/\mathfrak{B} \simp_+}_{\simp^{\op,\natural}}$
    can then be interpreted as giving fully faithful embedding:
    \[\Env^{/\mathfrak{B} \simp_+}_{\simp^{\op,\natural}} \colon \OpdI^{\gen,\ns} \hookrightarrow \catname{DblCat}_{\infty/\mathfrak{B}\simp_+}\]
    The essential image is characterized by a pullback square analogous to the one from \cref{ex:Delta-flat-fibrous-envelope}. Note that the morphisms in $\OpdI^{\gen,\ns}$ among the cocartesian fibrations that correspond to \itcats{} are precisely \emph{lax functors} as defined for instance in \cite{GaitsgoryRozenblyum1}, so we obtain a description of these in terms of $\catname{DblCat}_{\infty/\mathfrak{B}\simp_+}$. (More generally, we can also consider the envelope for $\simp^{n,\op,\natural}$, which was briefly discussed in \cite{nmorita}.)
\end{ex}

\begin{ex}
    Let $\calO \to \xF_\ast$ be an \iopd{}.
    Fibrous $\calO$-patterns are, by \cref{ex:fibrous-for-operads},
    exactly \iopds{} over $\calO$, while Segal $\calO$-\icats{} are precisely $\calO$-monoidal \icats{} which we denote by $\Mon_\calO(\CatI):=\Seg_{\calO}(\CatI)$. 
    By \cref{ex:operads-soundly-extendable}, $\calO$ is soundly extendable and our construction recovers the $\mathcal{O}$-monoidal envelope of \cite{HA}*{\S 2.2.4}. In particular, we see that this gives a fully faithful embedding
    \[ \Env^{/\calA_\calO}_{\calO} \colon \catname{Opd}_{\infty/\calO} \to \Mon_\calO(\CatI)_{/\calA_{\calO}}.\]
    In the case $\calO=\mathbb{E}_n$, the \icat{} $\calA_{\mathbb{E}_{n}}$ admits an alternative description as the $\mathbb{E}_n$-monoidal \icat{} of embedded $n$-disks in $\mathbb{R}^n$.
\end{ex}

\section{The comparison theorem}\label{sec:comp}
In \S\ref{subsec:comparison} we use the Segal envelopes to prove the comparison result, \cref{introthm:comp}. 
We then discuss the application of this to equivariant \iopds{},
\cref{introcor:Gopd}, in \S\ref{subsec:G-operads}.
Finally, we explain how to upgrade the envelope and comparison equivalences to
equivalences of \itcats{} in \S\ref{subsec:itcat}.

\subsection{Comparing fibrous patterns}
\label{subsec:comparison}

In this subsection we will use Segal envelopes to obtain a criterion
for a morphism of patterns $f \colon \mathcal{O}\to \mathcal{P}$
to induce via pullback an equivalence
\[ f^{*} \colon \Fbrs(\mathcal{P}) \isoto \Fbrs(\mathcal{O}) \]
between the corresponding \icats{} of fibrous patterns. We specialize
this to recover some comparison results from \cite{HA} without using
the technical results on approximations to \iopds{} from
\cite{HA}*{\S 2.3.3}. As new applications, we show that (symmetric) \iopds{} can also be described as fibrous patterns over $\Span(\xF)$, and that fibrous patterns over $\Span(\mathcal{S}_{m})$ and $\Span_{(m-1)\txt{-tr},\txt{all}}(\mathcal{S}_{m})$ are equivalent.

\begin{thm}\label{fbrseq}
  Suppose $\mathcal{O}$ is a pattern,
  $\mathcal{P}$ is a soundly extendable pattern,
  and $f \colon \mathcal{O} \to \mathcal{P}$ is a strong Segal
  morphism such that the following conditions hold:
  \begin{enumerate}[(i)]
  \item $f^{\el}\colon \mathcal{O}^{\el} \to \mathcal{P}^{\el}$ is an
    equivalence of \icats{},
  \item $(\calO^\act_{/X})^\simeq \to (\calP^\act_{/f(X)})^\simeq$ is an equivalence for all $X \in \calO$.
  \end{enumerate}
  Then pullback along $f$ gives an equivalence
  \[ f^{*} \colon \Fbrs(\mathcal{P}) \isoto \Fbrs(\mathcal{O}).\]
\end{thm}

\begin{remark}
    If we also assume that $\Afun{\calO}^\simeq = (\calO_{/-}^\act)^\simeq$ is an $\calO$-Segal space, for example if $\calO$ is soundly extendable,
    then it suffices to check condition (ii) when $X$ is elementary.
\end{remark}

\begin{ex}
    Let $\calP$ be a soundly extendable pattern, and define $\calO \subset \calP$ as the full subpattern on the ``necessary objects'' in the sense of \cite[Definition 14.7]{patterns1}.
    This means that $\calO$ contains those $X \in \calP$ for which there exists an active morphism $X \actto E$ with $E$ elementary.
    Then \cref{fbrseq} applies to the full inclusion $\calO \subset \calP$ and hence restriction yields an equivalence
    $\Fbrs(\calP) \simeq \Fbrs(\calO)$.
\end{ex}

First we show that condition (ii) can always be strengthened as follows.
\begin{lemma}\label{lem:actcatcomp}
  In the situation of \cref{fbrseq} the induced natural transformation
  \[ 
        \alpha \colon \mathcal{A}_{\mathcal{O}} \to f^{*} \mathcal{A}_{\mathcal{P}}
  \] 
  of functors $\mathcal{O} \to \CatI$ is an equivalence.
  In particular $\calA_\calO$ is $\calO$-Segal.
\end{lemma}
\begin{proof}
  By assumption, the functor $\Afun{\calO}(X) \to \Afun{\calP}(f(X))$ is an equivalence on underlying \igpds{},
  so it remains to show that it is fully faithful. 
  To see this, observe that given active maps 
  $\phi \colon X \actto Y$ and $\phi' \colon X' \actto Y$, 
  the mapping space
  $\Map_{\mathcal{O}^{\act}_{/Y}}(\phi', \phi)$ is the fiber at
  $\phi'$ of the map 
  $(\phi \circ -): \calA_{\calO}^{\simeq}(X) \to \calA_{\calO}^{\simeq}(Y)$.
  This map fits into a square
  \[
    \begin{tikzcd}
      \Agpd{\mathcal{O}}(X) \arrow{r}{\sim} \arrow{d} &
      \Agpd{\mathcal{P}}(fX) \arrow{d} \\
      \Agpd{\calO}(Y) \arrow{r}{\sim} & \Agpd{\mathcal{P}}(fY)
    \end{tikzcd}
  \]
  where the horizontal maps are equivalences.
  Then we also have equivalences on fibers, which gives the desired full faithfulness.
  Finally we note that $\calA_\calO \simeq f^* \calA_\calP$ implies that $\calA_\calO$ is Segal
  since $\calA_\calP$ was assumed to be Segal and $f^*$ preserves Segal objects.
\end{proof}

The following lemma tells us that for sound patterns it suffices to check $\Ar_\act(\calO)$-equifiberedness on active morphisms that end in elementary objects.
\begin{lemma}\label{lem:equifibered-checked-on-elementary}
    Let $\calO$ be a sound pattern and let $(\eta \colon F \Rightarrow G)$ be a relative Segal object over $\calO$ in a sufficiently complete $\infty$-category $\calC$.
    Suppose that the naturality squares
    \[
        \begin{tikzcd}
            F(X) \ar[r, "{F(\omega)}"] \ar[d, "\eta_X"] & F(Y) \ar[d, "\eta_Y"] \\
            G(X) \ar[r, "{G(\omega)}"] & G(Y)
        \end{tikzcd}
    \]
    are cartesian for active morphisms $\omega\colon X \actto Y$ where $Y$ is elementary.
    Then they are also cartesian for arbitrary $Y$, \ie{} $\eta$ is $\Ar_\act(\calO)$-equifibered.
\end{lemma}
\begin{proof}
    For an arbitrary active morphism $\omega \colon X \actto Y$ consider the commutative cube
    \[
        \begin{tikzcd}
            & \lim_{\alpha \colon  Y \intto E \in \calO^\el_{Y/}} F(\omega_{\alpha!} X) \ar[rr, "\lim F(\omega_\alpha)"] \ar[dd, "\lim \eta_{\omega_{\alpha!}X}", near end] &&
            \lim_{\alpha \colon  Y \intto E \in \calO^\el_{Y/}} F(E) \ar[dd, "\lim \eta_E"]  \\
            F(X) \ar[rr, "{F(\omega)}", near end, crossing over] \ar[dd, "\eta_X"] \ar[ru] && 
            F(Y) \ar[ru] & \\
            & \lim_{\alpha \colon  Y \intto E \in \calO^\el_{Y/}} G(\omega_{\alpha!} X) \ar[rr, "\lim G(\omega_\alpha)", near start] &&
            \lim_{\alpha \colon  Y \intto E \in \calO^\el_{Y/}} G(E) \\
            G(X) \ar[rr, "{G(\omega)}"] \ar[ru] && 
            G(Y). \ar[ru]
            \ar[from=2-3, to=4-3, "\eta_Y", near end, crossing over] 
        \end{tikzcd}
    \]
    The back square is cartesian as it is a limit over squares that we have assumed to be cartesian. 
    (Note that $\omega_\alpha \colon  \omega_{\alpha!}X \actto E$ is an active morphism with elementary target.) 
    The right face is cartesian because $\eta$ is a relative Segal object, and so is the left face from this and \cref{lem:sound-stronger-Segal}.
    Therefore the front face is cartesian by the pullback pasting lemma.
\end{proof}

\begin{proof}[Proof of \cref{fbrseq}]
  It follows from \cref{propn:Segmndcomp} that the functor
  \[f^{*} \colon \Seg_{\mathcal{P}}(\CatI) \to
  \Seg_{\mathcal{O}}(\CatI)\] is an equivalence. 
  From \cref{lem:actcatcomp} we have $\Afun{\calO} \simeq f^*\Afun{\calP}$ and that $\Afun{\calO}$ is Segal.
  Hence the induced functor
  \[f^{\ostar} \colon
    \Seg_{\mathcal{P}}(\CatI)_{/\mathcal{A}_{\mathcal{P}}} \to
    \Seg_{\mathcal{O}}(\CatI)_{/\mathcal{A}_{\mathcal{O}}}\]
  is also an equivalence. 
  This means in the commutative square
  \[
    \begin{tikzcd}
      \Fbrs(\mathcal{P}) \arrow{r}{f^{*}}
      \arrow[hook]{d}{\sliceEnv{\mathcal{P}}} & 
      \Fbrs(\mathcal{O}) \arrow[hook]{d}{\sliceEnv{\mathcal{O}}} \\
      \Seg_{\mathcal{P}}(\CatI)_{/\mathcal{A}_{\mathcal{P}}} \arrow{r}{f^{\ostar}}
      & \Seg_{\mathcal{O}}(\CatI)_{/\mathcal{A}_{\mathcal{O}}} 
    \end{tikzcd}
  \]
  from \cref{propn:general-fully-faithful-envelope}, the bottom
  horizontal functor $f^{\ostar}$ is an equivalence, while the
  vertical functors are fully faithful. It follows that the top
  horizontal functor $f^{*}$ is also fully faithful. To prove that it
  is also essentially surjective, it suffices to show that an object
  of $\Seg_{\mathcal{P}}(\CatI)_{/\Afun{\mathcal{P}}}$ is in
  the image of $\sliceEnv{\mathcal{P}}$ if its image under the
  equivalence $f^{\ostar}$ is in the image of
  $\sliceEnv{\mathcal{O}}$.

  Suppose we are given some $(\eta \colon F \Rightarrow \Afun{\calP}) \in \Seg_\calP(\CatI)_{/\Afun{\calP}}$ such that $f^\ostar F \Rightarrow \Afun{\calO}$ is equifibered.
  Equivalently, $\eta_{\circ f} \colon (F \circ f) \Rightarrow (\Afun{\calP} \circ f)$ is equifibered.
  By \cref{lem:equifibered-checked-on-elementary} it suffices to check that the naturality squares are cartesian for active morphisms $\omega \colon  X \actto E \in \calP$ ending in an elementary.
  Since $f \colon \calO^\el \to \calP^\el$ is an equivalence,
  we may write $E \simeq f(E')$ for $E' \in \calO$.
  Moreover, since $f \colon \calO_{/E'}^\act \to \calP_{/f(E')}^\act$ is an equivalence,
  we can find $\rho \colon  Y \actto E' \in \calO$ such that $f(\rho) \simeq \omega$ as objects of $\Ar_\act(\calP)$.
  Now it follows that the naturality square of $\eta$ at $\omega$ is cartesian since we assumed that the naturality square of $\eta_{\circ f}$ at $\rho$ is equifibered.
  This shows that $\eta$ is $\Ar_\act(\calP)$-equifibered,
  and hence that $f^* \colon \Fbrs(\calP) \to \Fbrs(\calO)$ is essentially surjective.
\end{proof}

As a variant of \cref{fbrseq}, we get a useful criterion for
identifying the effect of the pushforward functor
$f_{!} \colon \Fbrs(\mathcal{O}) \to \Fbrs(\mathcal{P})$ for a map of
patterns $f \colon \mathcal{O} \to \mathcal{P}$:
\begin{cor}
  Suppose we have a commutative diagram of patterns
  \[
    \begin{tikzcd}
      \mathcal{Q} \arrow{r}{g}  \arrow{d}{p} & \mathcal{R} \arrow{d}{q} \\
      \mathcal{O} \arrow{r}{f} & \mathcal{P}
    \end{tikzcd}
  \]
  such that
  \begin{enumerate}[(i)]
  \item $\mathcal{O}$ is sound and $\mathcal{Q}$ is a fibrous $\mathcal{O}$-pattern,
  \item $\mathcal{P}$ is soundly extendable and $\mathcal{R}$ is a fibrous $\mathcal{P}$-pattern,
  \item $f$ is a strong Segal morphism,
  \item $g$ satisfies the assumptions of \cref{fbrseq}.
  \end{enumerate}
  Then the induced map of fibrous $\mathcal{O}$-patterns
  $\mathcal{Q} \to f^{*}\mathcal{R}$ is adjoint to an equivalence
  $f_{!}\mathcal{Q} \isoto \mathcal{R}$.
\end{cor}
\begin{proof}
  For a fibrous $\mathcal{P}$-pattern $\mathcal{T}$, we have natural equivalences
  \[
    \begin{split}
      \Map_{\Fbrs(\mathcal{O})}(\mathcal{Q}, f^{*}\mathcal{T}) & \simeq \Map_{\Fbrs(\mathcal{O})_{/\mathcal{Q}}}(\mathcal{Q}, p^{*}f^{*}\mathcal{T}) \\
                                                               & \simeq \Map_{\Fbrs(\mathcal{Q})}(g^{*}\mathcal{R}, g^{*}q^{*}\mathcal{T}) \\
                                                               & \simeq \Map_{\Fbrs(\mathcal{R})}(\mathcal{R}, q^{*}\mathcal{T}) \\
       & \simeq \Map_{\Fbrs(\mathcal{P})}(\mathcal{R}, \mathcal{T}),
    \end{split}
  \]
  where we have used \cref{fbrseq} and \cref{cor:compeqWSF}.
\end{proof}

\begin{cor}\label{cor:symmetrization}
  Suppose $\mathcal{O}$ is a sound pattern, $q \colon \mathcal{P} \to \xF_{*}$ is a symmetric \iopd{}, and $f \colon \mathcal{O} \to \mathcal{P}$ is a strong Segal morphism that satisfies the assumptions of \cref{fbrseq}. Then $f$ exhibits $\mathcal{P}$ as the \emph{symmetrization} of $\mathcal{O}$, in the sense that the induced map $(qf)_{!} \mathcal{O} \to \mathcal{P}$ is an equivalence. \qed
\end{cor}

\begin{ex}\label{ex:nonsymmcomp}
  Let $\catname{Ass}$ be the (symmetric) associative \iopd{} as defined in \cite{HA}*{Definition 4.1.1.1.}, and let
  $\mathrm{Cut} \colon \Dop \to \catname{Ass}$ denote the functor defined in \cite{HA}*{Construction 4.1.2.9.}. Then
  pullback along $\mathrm{Cut}$ gives an equivalence
  \[ \Fbrs(\simp^{\op,\flat}) \isofrom \Fbrs(\catname{Ass}) \isoto
    \Fbrs(\xF_{*})_{/\catname{Ass}} \]
  between non-symmetric \iopds{} and symmetric \iopds{} over
  $\catname{Ass}$, where the second equivalence is that of
  \cref{cor:compeqWSF}. In other words, non-symmetric \iopds{} are
  equivalent to symmetric \iopds{} over the associative \iopd{}. Moreover, $\catname{Ass}$ is the symmetrization of $\simp^{\op,\flat}$.
\end{ex}

The equivalence of \cref{ex:nonsymmcomp} is also proved by Lurie as
\cite{HA}*{Theorem 4.1.3.14}, which is a special case of \cite{HA}*{Theorem 2.3.3.26}. 
This more general statement can also be proved by our
methods; to see this, we first need to recall some definitions:

\begin{defn}
  Let $\pi \colon \mathcal{O} \to \xF_{*}$ be an \iopd{}.  We say a
  functor $f \colon \mathcal{C} \to \mathcal{O}$ is an
  \emph{approximation} if the following conditions hold:
  \begin{enumerate}[(1)]
  \item For $C \in \mathcal{C}$ over $\angled{n}$ in $\xF_{*}$, there
    exists for $i = 1,\ldots,n$ a locally cocartesian morphism
    $\rho_{i}^{C} \colon C \to
    C_{i}$ in $\mathcal{C}$ over $\rho_{i} \colon \angled{n} \to
    \angled{1}$. Moreover, the image of $\rho_{i}^{C}$ in
    $\mathcal{O}$ is inert. 
  \item $\mathcal{C}$ has all $f$-cartesian lifts of active morphisms
    in $\mathcal{O}$.
  \end{enumerate}
  Following \cite{HinichYoneda}, we say that $f$ is a \emph{strong
    approximation} if we additionally have:
  \begin{itemize}
  \item[(3)] The functor $\mathcal{C}_{\angled{1}} \to
    \mathcal{O}_{\angled{1}}$ is an equivalence. 
  \end{itemize}
\end{defn}

\begin{remark}
  Suppose $\mathcal{O}$ is an \iopd{} and
  $f \colon \mathcal{C} \to \mathcal{O}$ is an approximation. We say a
  morphism in $\mathcal{C}$ is \emph{inert} if its image in
  $\mathcal{O}$ is inert, and \emph{active} if it is $f$-cartesian and
  its image in $\mathcal{O}$ is active. Then the inert and active
  morphisms in $\mathcal{C}$ give a factorization system. We think of
  $\mathcal{C}$ as an algebraic pattern using this factorization
  system, with the elementary objects being those that map to
  $\angled{1}$ in $\xF_{*}$; then $f$ is a morphism of algebraic
  patterns.
\end{remark}

\begin{propn}\label{propn:approx}
  Suppose $\mathcal{O}$ is an \iopd{} and
  $f \colon \mathcal{C} \to \mathcal{O}$ is a strong
  approximation. Then:
  \begin{enumerate}[(i)]
  \item\label{item:eleq1} $\mathcal{C}^{\el} \to \mathcal{O}^{\el}$ is an equivalence.    
  \item\label{item:segcoinit} $\mathcal{C}^{\el}_{C/}
    \to \mathcal{O}^{\el}_{f(C)/}$ is an equivalence for all $C \in
    \mathcal{C}$, i.e.\ $f$ is an iso-Segal morphism.
  \item\label{item:acteq} $\mathcal{C}^{\act}_{/C} \to \mathcal{O}^{\act}_{/f(C)}$ is an
    equivalence for all $C \in \mathcal{C}$.
  \end{enumerate}
\end{propn}
\begin{proof}
  For \ref{item:eleq1}, observe that from the equivalence
  $\mathcal{C}_{\angled{1}} \isoto \mathcal{O}_{\angled{1}}$ it
  follows that a morphism in $\mathcal{C}$ over $\angled{1}$ is inert
  \IFF{} it is an equivalence (since the equivalences are precisely
  the inert morphisms in $\mathcal{O}_{\angled{1}}$). Hence
  $\mathcal{C}^{\el} = \mathcal{C}_{\angled{1}}^{\simeq}$, so the
  functor $\mathcal{C}^{\el} \to \mathcal{O}^{\el}$ is just the
  underlying morphism of \igpds{} of the functor between fibers over
  $\angled{1}$ that is an equivalence by assumption.

  To show \ref{item:segcoinit}, we first observe that 
  $\mathcal{C}^{\el}_{C/}$ is an \igpd{}, since morphisms are given by inert maps over $\angled{1}$ and these are invertible. Moreover, if $C$ lies over $\angled{n}$ then the fiber
  of $\mathcal{C}^{\el}_{C/}$ over $\rho_{i}$ is contractible, since there by assumption exists a locally cocartesian morphism over $\rho_{i}$ --- this is then initial in the \icat{} $(\mathcal{C}_{C/})_{\rho_{i}}$ and so in particular has no automorphisms.

  We thus have a commutative triangle
  \[
    \begin{tikzcd}
      \mathcal{C}^{\el}_{C/} \arrow{rr} \arrow{dr}[swap]{\sim} & & \mathcal{O}^{\el}_{f(C)/} \arrow{dl}{\sim} \\
       & (\xF_{*}^{\el})_{\angled{n}}
    \end{tikzcd}
  \]
  where both maps to $(\xF_{*}^{\el})_{\angled{n}}$ are equivalences, hence so is the top horizontal map.

  To prove \ref{item:acteq}, observe that by assumption
  $\mathcal{C}^{\act} \to \mathcal{O}^{\act}$ is the underlying right
  fibration of the cartesian fibration $\mathcal{C}
  \times_{\mathcal{O}} \mathcal{O}^{\act} \to
  \mathcal{O}^{\act}$. This gives the required equivalence of slices
  by \cite{Kerodon}*{Tag \texttt{00TE}}.
\end{proof}

\begin{cor}\label{strong-approx-comp}
  Suppose $f \colon \mathcal{C} \to \mathcal{O}$ is a strong
  approximation to an \iopd{} $q \colon \mathcal{O} \to \xF_{*}$. 
  \begin{enumerate}
  \item If $\mathcal{X}$ is an \icat{} with finite products, then restriction along
    $f$ gives an equivalence
    \[ f^{*} \colon \Seg_{\mathcal{O}}(\mathcal{X}) \isoto
      \Seg_{\mathcal{C}}(\mathcal{X}).\]
  \item Pullback along $f$ gives an equivalence
    \[ f^{*} \colon \Fbrs(\mathcal{O}) \isoto \Fbrs(\mathcal{C}).\]
  \item The map $f$ exhibits $\mathcal{O}$ as the symmetrization of $\mathcal{C}$, \ie{} $(qf)_{!}\mathcal{C} \simeq \mathcal{O}$.
  \end{enumerate}
\end{cor}
\begin{proof}
  Combine \cref{propn:approx} with \cref{propn:Segmndcomp},
  \cref{fbrseq}, and \cref{cor:symmetrization}.
\end{proof}

\begin{remark}
  Lurie's proof of \cite{HA}*{Theorem 2.3.3.26} uses envelopes for
  approximations to \iopds{}, just as our proof of \cref{fbrseq}, and 
  we do not claim that our proof is different in any essential
  way.
\end{remark}

We end this section with a couple of examples that do not follow from
\cref{strong-approx-comp} or \cite{HA}*{Theorem 2.3.3.26}. These
involve patterns defined using spans, so we start with a general
observation about comparisons of these:
\begin{observation}\label{obs:comparison-for-span}
  Consider two adequate triples $(\frX,\frXb,\frXf)$ and
  $(\frY,\frY^b,\frY^f)$ and a functor $F \colon \frX \to \frY$ that
  preserves the two subcategories and also preserves pullbacks of backwards
  maps along forwards maps.  Suppose further that we have full
  subcategories $\frX_0 \subset \frX$ and $\frY_0 \subset \frY$
  such that $F(\frX_0) \subset \frY_0$. 
  Then $F$ induces a morphism
  of patterns:
    \[
        F \colon \Spanbf(\frX; \frX_0) \to \Spanbf(\frY; \frY_0).
    \]
    We may apply \cref{fbrseq} to this if the following conditions hold:
    \begin{enumerate}[(1)]
        \item $\Spanbf(\frY; \frY_0)$ is soundly extendable. (See \cref{prop:Span-sound}.)
        \item\label{it:fibrous-span-cond2} For all $x \in \frX$, the map 
        $\frXb_0 \times_{\frXb} \frXb_{/x} \to \frY^b_0 \times_{\frY^b} \frY^b_{/F(x)}$ is cofinal.
        \item $F \colon \frX_0^b \to \frY_0^b$ is an equivalence of $\infty$-categories.
        \item $F \colon \frXf_{/x} \to \frY^f_{/F(x)}$ induces an  equivalence on maximal subgroupoids for all $x \in \frX$.
    \end{enumerate}
    Note that point \ref{it:fibrous-span-cond2} ensures that $F$ is a strong Segal morphism
    since $\Spanbf(\frX; \frX_0)^\xint \simeq (\frX^b)^\op$ with the elemetaries being $(\frXb_0)^\op$.
\end{observation}

\begin{cor}\label{cor:SpanF}
  Pullback along the inclusion $\mathfrak{i} \colon \xF_{*} \simeq \Span_{\txt{inj},\txt{all}}(\xF) \to \Span(\xF)$ 
  gives an equivalence
  \[ \mathfrak{i}^{*} \colon \Fbrs(\Span(\xF)) \isoto \Fbrs(\xF_{*}) \simeq \OpdI.\]
\end{cor}
\begin{proof}
    We check the conditions of \cref{fbrseq} in the form stated in \cref{obs:comparison-for-span}:
    \begin{enumerate}
        \item The pattern is soundly extendable by \cref{ex:spansoundext}.
        \item For $A \in \xF$ the relevant functor is the restriction of $\xF_{/A}^{\txt{inj}} \to \xF_{/A}$ to elementaries.
        But every map out of a one-point set is injective, so this is an equivalence.
        \item Similarly, the functor on backwards morphisms $\xF^{\txt{inj}} \to \xF$ restricts to an equivalence on elementaries.
        \item Both categories have the same forward morphisms.
        \qedhere
    \end{enumerate}
\end{proof}

More generally, we have:
\begin{cor}\label{cor:SpanSm}
  Pullback along the inclusion 
  $\mathfrak{i}_{m} \colon \Span_{(m-1)\txt{-tr,all}}(\mathcal{S}_{m}) \rightarrow \Span(\mathcal{S}_{m})$
  induces an equivalence
  \[ \mathfrak{i}_{m}^{*} \colon \Fbrs(\Span(\mathcal{S}_{m})) \to
    \Fbrs(\Span_{(m-1)\txt{-tr},\txt{all}}(\mathcal{S}_{m})).\]
\end{cor}
\begin{proof}
    We can apply \cref{fbrseq}:
    The target pattern $\Span(\mathcal{S}_{m}))$ is soundly extendable by \cref{ex:SpanSm-soundly-extendable}
    and in this example we also note that $\mathfrak{i}_m$ is an iso-Segal morphism.
    Condition (i) of \cref{fbrseq} holds because in both cases the elementary \icat{} is the terminal \icat{}.
    Condition (ii) holds because both span \icats{} have the same forward morphisms.
\end{proof}


\subsection{\texorpdfstring{$G$-equivariant $\infty$-operads}{G-equivariant infinity-operads}}
\label{subsec:G-operads}

In this section we apply the theory of fibrous patterns and envelopes
in the setting of $G$-equivariant \iopds{} developed in \cite{NS}.
While their paper works in the generality of $T$-parametrized
\iopds{}, we will restrict to the special case of the orbit category
$T = \Orb_G$ for simplicity.  Our main result is that the
$G$-$\infty$-operads of \cite{NS} are equivalent to fibrous
$\Span(\xF_G)$-patterns; we will also show that the sliced envelope
for $G$-$\infty$-operads is fully faithful and characterize the image, giving a third description
of these objects.

First, we recall some constructions in equivariant higher algebra,
which were pioneered in \cite{BarwickMackey} and further developed in \cite{NardinThesis} and \cite{NS}.
Fix a finite group $G$ throughout.
\begin{defn}
    Let $\xF_G$ be the category of finite $G$-sets,
    $\xF_{G,*}$ the category of finite pointed $G$-sets,
    and  $\Orb_G \subset \xF_G$ the full subcategory of $G$-orbits.
\end{defn}

\begin{defn}
    A \emph{$G$-$\infty$-category} is a functor $\Orb_G^\op \to \CatI$ and
    a \emph{$G$-symmetric monoidal $\infty$-category} is a $\Span(\xF_G)$-Segal object in $\CatI$.
    We write 
    \[
        \CatG := \Fun(\Orb_G^\op, \CatI)
        \qquad \text{ and } \qquad
        \CatG^\otimes := \Seg_{\Span(\xF_G)}(\CatI)
    \]
    and define the forgetful functor 
    $\CatG^{\otimes} \to \CatG$
    by restricting to the elementaries $\Orb_G^\op \to \Span(\xF_G)$.
\end{defn}

\begin{notation}
    For a $G$-$\infty$-category $\calC \colon \Orb_G^\op \to \CatI$ 
    we denote its value at $G/H$ by $\calC^H$ an refer to it 
    as the $H$-fixed point category of $\calC$.
    There are restriction maps $\calC^H \to \calC^K$ for $K \subset H \subset G$.
    Given a $G$-symmetric monoidal $\infty$-category $\calD \colon \Span(\xF_G) \to \CatI$ we further have tensor products
    $\otimes \colon  \calD^H \times \calD^H \to \calD^H$
    and so-called norm maps 
    $\mathrm{Nm}_K^H \colon  \calD^K \to \calD^H$ 
    for all $K \subset H \subset G$ coming from the span $(G/K \xfrom{=} G/K \to G/H)$.
\end{notation}

\begin{ex}\label{ex:ulF_G}
    Since $\Span(\xF_G)$ is an extendable pattern (\cref{ex:spanFGrestrsound})
    $\calA_{\Span(\xF_G)}$ is a Segal object in $\CatI$.
    We denote this $G$-symmetric monoidal $\infty$-category by
    \[
        \ulF_G := \calA_{\Span(\xF_G)}(-) = \Span(\xF_G)_{/-}^{\act} \colon 
        \Span(\xF_G) \to \CatI.
    \]
    The $H$-fixed point category is the category of finite $H$-sets:
    \[
        (\ulF_G)^H = \Span(\xF_G)_{/(G/H)}^\act \simeq (\xF_G)_{/(G/H)} \simeq \xF_H.
    \]
    The restriction maps are given by restriction, the tensor product by disjoint union, and the norm maps are
    $(- \times H)_{/K} \colon  \xF_K \to \xF_H$.
    In summary, $\ulF_G$ is $\xF_G$ with its natural structure as a $G$-symmetric monoidal $\infty$-category.
\end{ex}

Below we will see that fibrous $\Span(\xF_G)$-patterns model $G$-$\infty$-operads.
We now explain how $\mathcal{N}_\infty$-operads fit into this framework:
\begin{ex}\label{ex:SpanfG-is-fibrous}
    Let $\xF_G^f \subset \xF_G$ be a wide subcategory closed under base-change and disjoint union.
    Then the inclusion functor 
    $\Span_{\txt{all},f}(\xF_G) \to \Span(\xF_G)$
    defines a fibrous $\Span(\xF_G)$-pattern.
    To see that it has cocartesian lifts for inerts, note that any functor of the form
    $\Span_{b,f}(\calC) \to \Span_{b,\txt{all}}(\calC)$
    has cocartesian lifts for backwards maps.
    For the second condition we need to show that 
    \[
        (\xF_G^f)_{/A} \to \lim_{U \in (\Orb_G)_{/A}} (\xF_G^f)_{/U}
    \]
    is an equivalence. 
    The limit may be rewritten as a product over the set of orbits of $A$ and then the equivalence follows because $\xF_G^f$ is closed under base-change and disjoint union.
    
    Categories $\xF_G^f$ that in addition to the above also contain all fold maps $\nabla: G/H \amalg G/H \to G/H$, are in bijection with the indexing systems of \cite{blumberg-hill}, see \cite[Remark 2.4.12]{NS}.
    Under the equivalence $\Fbrs(\Span(\xF_G)) \simeq \OpdG$ proved below
    the fibrous $\Span(\xF_G)$-patters described above
    are the ``commutative $G$-$\infty$-operads'' from \cite[Definition 2.4.10]{NS},
    which correspond to the $\mathcal{N}_\infty$-operads of \cite{blumberg-hill} by \cite[Remark 2.4.12]{NS}.
\end{ex}

We now quickly recall the necessary notation from \cite{NS} to state their definition of $G$-\iopds{}, but we refer the reader there for details.
\begin{defn}
    Define $\ulF_G^\vee \subset \Ar(\xF_G)$ as the full subcategory of those morphisms $(f \colon U \to V)$ where $V$ is an orbit:
    $\ulF_G^\vee := \Ar(\xF_G) \times_{\xF_G} \Orb_G$.
    We say that a morphism $f \to g$ given by
    \[
        \begin{tikzcd}
            U \ar[r, "h"] \ar[d, "f"] & 
            X \ar[d, "g"] \\
            V \ar[r, "k"] &
            Y
        \end{tikzcd}
    \]
    \begin{itemize}
        \item lies in $(\ulF_G^\vee)^{\txt{si}}$ if it is a summand inclusion, i.e.~$U \to X \times_Y V$ is injective,
        \item lies in $(\ulF_G^\vee)^{\txt{tdeg}}$ if it is target degenerate, i.e.~$k \colon V \to Y$ is an equivalence.
    \end{itemize}
\end{defn}

\begin{defn}
    Define $\ulF_{G,*}$ as the algebraic pattern
    \[
        \ulF_{G,*} := \Spansitdeg(\ulF_G^\vee; \Orb_G),
      \]
      where the elementary objects are those in the essential image of
      the identity inclusion
      $\Orb_G \to \Ar(\Orb_G) \subset \ulF_G^\vee$.  
\end{defn}

\begin{remark}
  The functor $\ev_1 \colon \ulF_G^\vee \to \Orb_{G}$ induces a cocartesian fibration 
    \[
        \ulF_{G,*} = \Spansitdeg(\ulF_G^\vee)
        \xrightarrow{\ev_1}
        \Span_{\txt{all},\txt{iso}}(\Orb_G) \simeq \Orb_G^\op.
    \]
  Straightening this yields a $G$-$\infty$-category whose $H$-fixed point category is $(\ulF_{G,*})^H \simeq \xF_{H,*}$, similarly to \cref{ex:ulF_G}.
\end{remark}

\begin{observation}\label{obs:elemenatry-slice-of-ulF}
    For $(U \to V) \in \ulF_{G,*}$ the category of elementaries under $(U \to V)$ is equivalent to the opposite of the category of orbits over $U$ (as in \cref{rem:Segal-condition-span-patterns}):
    \[
        (\ulF_{G,*})^\el_{(U \to V)/} 
        \simeq (\Orb_G \times_{(\ulF_G^\vee)} (\ulF_G^{v, si})_{/(U \to V)})^\op
        \simeq (\Orb_G \times_{\xF_G} (\xF_G)_{/U})^\op.
    \]
    Here we used that any morphism $(Q \xto{=} Q) \to (U \to V)$ (where $Q$ is an orbit) is automatically in $(\ulF_G^\vee)^{\txt{si}}$ since $Q \to Q \times_V U$ is injective.
    Now consider the full subcategory on those $(Q \to U)$ that are injective. 
    This subcategory is equivalent to the discrete set of orbits $U/G$ and moreover the inclusion of the subcategory is a left adjoint:
    \[
        U/G \hookrightarrow 
        (\Orb_G \times_{\xF_G} (\xF_G)_{/U})^\op \simeq
        (\ulF_{G,*})^\el_{(U \to V)/}, 
    \]
    with right adjoint given by sending $(f \colon Q \to U)$ to $(f(Q) \hookrightarrow U)$.
    In particular, the inclusion of $U/G$ is a coinitial functor.
    This means that for any kind of (weak) Segal condition over $\ulF_{G,*}$ the limit involved can be rewritten as a product indexed by the finite set $U/G$.
\end{observation}

\begin{cor}\label{cor:ulFG-sound}
    The pattern $\ulF_{G,*}$ is sound.%
    \footnote{
        In fact this pattern is soundly extendable. 
        This follows because the functor $\ulF_{G,*} \to \Span(\xF_G)$ discussed in \cref{prop:GOpd=Span}
        is iso-Segal and induces an equivalence on forward maps.
        However, the extendability of $\ulF_{G,*}$ will not be needed here.
    }
\end{cor}
\begin{proof}
    We check the conditions of \cref{prop:Span-sound}.
    First we show that the backwards maps satisfy cancellation.
    Consider two morphisms in $\ulF_G^\vee$:
    \[
        \begin{tikzcd}
            A \ar[r, "a"] \ar[d, "e"] & 
            U \ar[r, "h"] \ar[d, "f"] & 
            X \ar[d, "g"] \\
            B \ar[r, "b"] &
            V \ar[r, "k"] &
            Y
        \end{tikzcd}
    \]
    such that $A \to B \times_Y X$ is injective.
    We can write this map as a composite $A \to B \times_V U \to B \times_Y X$,
    the first map of which then has to be injective.
    In other words $(a,b):e \to f$ is in $\ulF_G^{v, \rm si}$ as claimed.
    
    We also need to show that the inclusion $\frX^b_{0/y} \hookrightarrow \frX_{0/y}$ is cofinal.
    In the case at hand this inclusion is
    $\Orb_G \times_{\ulF_G^\vee} (\ulF_G^{v,\txt{si}})_{/(U \to V)} \to 
    \Orb_G \times_{\ulF_G^\vee} (\ulF_G^{v})_{/(U \to V)} $,
    which is an equivalence by the argument from \cref{obs:elemenatry-slice-of-ulF}.
\end{proof}

\begin{defn}[\cite{NS}]
    A \emph{$G$-$\infty$-operad} is a weak Segal fibration over $\ulF_{G,*}$ in the sense of \cite[Definition 9.6]{patterns1}, see also \cref{propn:fibrous=WSF}.
    Let $\OpdG$ denote the full subcategory of $\Cat_{\infty/\ulF_{G,*}}^{\intcoc}$
    on the $G$-$\infty$-operads.
\end{defn}

\begin{observation}
    This agrees with the definition of \cite{NS}.
    First we note that given $p \colon \calP \to \ulF_{G,*}$ with cocartesian lifts for inerts, the composite $\mathrm{ev}_1 \circ p \colon \calP \to \Orb_G^\op$ exhibits $\calP$ as a cocartesian fibration over $\Orb_G^\op$, \ie{} an $\Orb_G$-\icat{}, and $p$ as an $\Orb_G$-functor.
    This holds because the inert morphisms in $\ulF_{G,*}$ contain all the cocartesian lifts of $\mathrm{ev}_1\colon \ulF_{G,*} \to \Orb_G^\op$.
    We hence have an identification:
    \[
        \Cat_{\infty/\ulF_{G,*}}^{\intcoc} =
        (\CatG)_{/\ulF_{G,*}}^{\intcoc}.
    \]
    It remains to see that their conditions (2) and (3) exactly amount to the weak Segal conditions (2) and (3) in \cite[Definition 9.6]{patterns1}.
    Indeed, this follows by inspection using 
    \cref{obs:elemenatry-slice-of-ulF} and \cite[Remark 9.7]{patterns1}.
\end{observation}

\begin{cor}\label{cor:OpdG=Fbrs}
    We have $\OpdG = \Fbrs(\ulF_{G,*})$.
\end{cor}
\begin{proof}
    The pattern $\ulF_{G,*}$ is sound by \cref{cor:ulFG-sound}
    and hence weak Segal fibrations and fibrous patterns are the same by \cref{propn:fibrous=WSF}.
\end{proof}

\begin{propn}\label{prop:GOpd=Span}
    Restriction along the morphism of patterns
    $\ulF_{G,*} \xrightarrow{s} \Span(\xF_G)$
    induced by the functor $\ulF_{G}^{v} \to \xF_{G}$ given by evaluation at $0$
    yields an equivalence
    \[
        s^{*} \colon \Fbrs(\Span(\xF_G)) \xrightarrow{\simeq}
        \Fbrs(\ulF_{G,*}) = \OpdG. 
    \]
\end{propn}
\begin{proof}
    We need to show that the morphism of patterns
\[
        s \colon \ulF_{G,*} = \Spansitdeg(\ulF_G^\vee; \Orb_G) 
        \to \Span(\xF_G; \Orb_G)
\]        
    satisfies the conditions of \cref{fbrseq}.
    Since this comes from a morphism of adequate triples, we can use the formulation in \cref{obs:comparison-for-span}.
    We check each of the conditions there in turn:
    \begin{enumerate}
        \item It was checked in \cref{ex:spansoundext} that $\Span(\xF_G)$ is soundly extendable.
        \item We need to show that 
        \[
            (\Orb_G \times_{(\ulF_G^\vee)} (\ulF_G^{v, \txt{si}})_{/(U \to V)})^\op
            \to
            (\Orb_G \times_{\xF_G} (\ulF_G)_{/U})^\op
        \]
        is cofinal. 
        But we have already noted in \cref{obs:elemenatry-slice-of-ulF} that it is an equivalence.
        \item This holds since the functor induces the identity on $\Orb_G$.
        \item For all $U \in \xF_G$ the functor
        \[
            (\ulF_G^{v,\txt{tdeg}})_{/(U\to V)} \to (\xF_G)_{/U}
        \]
        is an equivalence by inspection of the definition of $(\ulF_G^{v})^{\txt{tdeg}}$.
        \qedhere
    \end{enumerate}
\end{proof}

As a consequence we obtain a fully faithful envelope into the \icat{} of $G$-symmetric monoidal \icats{} over $\ulF_G$ and a characterization of the image.
\begin{cor}
    There is an adjunction
    \[
        \Env_G \colon \OpdG \adj \CatG^\otimes :\! \mathrm{forget}
    \]
    where the left adjoint may be lifted to a fully faithful functor
    \[
        \Env_G\colon \OpdG \hookrightarrow (\CatG^\otimes)_{/\ulF_G}.
    \]
    This functor has both adjoints and its essential image consists
    of those $G$-symmetric monoidal functors $p\colon\calC \to \ulF_G$
    that are $\Ar_\act(\Span(\xF_G))$-equifibered.
\end{cor}
\begin{proof}
    Using that $\OpdG \simeq \Fbrs(\Span(\xF_G))$ by \cref{prop:GOpd=Span}, this is an instance of \cref{thm:envelope-for-soundly-extendable}.
    Note that the envelope of the terminal $G$-$\infty$-operad is 
    $\Env_{\Span(\xF_G)}(*) = \calA_{\Span(\xF_G)} = \ulF_G$ 
    by \cref{ex:ulF_G}.
\end{proof}

We elaborate further on the characterization of the image:
\begin{lemma}\label{lem:G-eqf}
    A $G$-symmetric monoidal functor $F\colon\calC \to \calD$ is $\Ar_\act(\Span(\xF_G))$-equifibered
    if and only if 
    \[
    \begin{tikzcd}
        \calC^H \times \calC^H \ar[r, "\otimes"] \ar[d] &
        \calC^H \ar[d] \\
        \calD^H \times \calD^H \ar[r, "\otimes"] &
        \calD^H
    \end{tikzcd}
        \qquad \text{and} \qquad
    \begin{tikzcd}
        \calC^K \ar[r, "\mathrm{Nm}_K^H"] \ar[d] &
        \calC^H \ar[d] \\
        \calD^K \ar[r, "\mathrm{Nm}_K^H"] &
        \calD^H
    \end{tikzcd}
    \] 
    are pullback squares of $\infty$-categories for all subgroups $K \subset H \subset G$.
\end{lemma}
\begin{proof}
    $F$ induces a natural transformation of functors $\xF_G \to \Cat$, defined by restricting to forwards maps in $\Span(\xF_G)$.
    Let $\calK \subset \xF_G$ denote the maximal subcategory such that the restriction of $F$ to $\calK$ is a cartesian natural transformation.
    Then $F$ is $\Ar_\act(\Span(\xF_G))$-equifibered if and only if $\calK = \xF_G$.
    Note that $\calK$ is closed under composition and right-cancellation, since pullback squares are, and contains all equivalences.
    Moreover, $\calK$ is closed under disjoint union since both functors $\calC, \calD \colon \xF_G \to \Cat$ send disjoint unions to products.
    Using this one can see that to show $\calK = \xF_G$,
    it suffices to check that $\calK$ contains the morphisms 
    \[
        \nabla \colon  G/H \amalg G/H \to G/H, 
        \qquad \text{and} \qquad
        G/K \to G/H
    \]
    for all subgroups $K \subset H \subset G$.
    This is exactly the condition stated in the lemma.
\end{proof}

\begin{remark}
    One might hope that $G$-$\infty$-operads are also equivalent to fibrous $\xF_{G,*}$-patterns, in analogy with what we showed in \cref{cor:SpanF} for $G=\{e\}$, but this is false for non-trivial groups.
    Note that the orbit functor $(-)_G \colon  \xF_{G,*} \to \xF_*$ exhibits $\xF_{G,*}$ as a fibrous $\xF_*$-pattern, i.e.\ an $\infty$-operad in the sense of Lurie.
    Therefore there is an equivalence $\Fbrs(\xF_{G,*}) \simeq (\OpdI)_{/\xF_{G,*}}$.
    We refer to this as the \icat{} of \emph{naive $G$-$\infty$-operads}.
    There is an inclusion of patterns $\xF_{G,*} \to \Span(\xF_G)$ similar to the one used in \cref{cor:SpanF}, and this is a strong Segal morphism by an argument as in \cref{obs:elemenatry-slice-of-ulF}.
    Therefore there is a restriction functor:
    \[
        \OpdG \simeq \Fbrs(\Span(\xF_G)) \longrightarrow \Fbrs(\xF_{G,*}) \simeq (\OpdI)_{/\xF_{G,*}}
    \]
    which forgets from (genuine) $G$-$\infty$-operads to naive $G$-$\infty$-operads.
    However, we cannot apply the comparison theorem \ref{fbrseq} since 
    $(\Orb_G^\op)^\simeq \simeq \xF_{G,*}^\el \to \Span(\xF_G)^\el \simeq \Orb_G^\op$
    is not an equivalence.
\end{remark}

\subsection{Upgrading to \texorpdfstring{$(\infty,2)$}{(infinity,2)}-categories}\label{subsec:itcat}
In this subsection we will upgrade our main results from \icats{} to
\itcats{}: we will see that the comparison equivalence of
\cref{fbrseq} is an equivalence of \itcats{} and the fully faithful
envelope functor of \cref{propn:general-fully-faithful-envelope} is a
fully faithful functor of \itcats{}. More precisely, we will show that
these functors are compatible with natural $\CatI$-module structures
on the \icats{} involved. It then follows from results of
Hinich~\cite{HinichYoneda} and Heine~\cite{HeineEnr} that these
\icats{} can be upgraded to \itcats{} and the functors to functors of
\itcats{}. We will not comment further on this, however, as our primary
interest is in showing that our equivalences are compatible with the
natural \icats{} of maps, which is an immediate consequence of
compatibility with the $\CatI$-module structures.  We begin by
defining such module structures on the \icats{} and functors we
studied in \S\ref{sec:factenv}:

\begin{construction}\label{obs:ar0catmod}
  Let $\mathcal{B}$ be an \icat{} equipped with a wide subcategory
  $\mathcal{B}_{0}$. The forgetful functor $\CatIsl{\mathcal{B}} \to
  \CatI$ has a right adjoint, taking $\mathcal{C} \in \CatI$ to the
  projection $\mathcal{C} \times \mathcal{B} \to \mathcal{B}$; this
  factors through the subcategory $\CatIzcoc{\mathcal{B}}$ and thus
  gives symmetric monoidal functors
  \[ \CatI \to \CatIzcoc{\mathcal{B}} \to \CatIsl{\mathcal{B}} \]
  with respect to the cartesian products. It follows that 
  both $\CatIsl{\mathcal{B}}$ and
  $\CatIzcoc{\mathcal{B}}$ are $\CatI$-modules, with the tensoring in
  both cases simply given by cartesian product, \ie{}
  \[ (\mathcal{C}, \mathcal{E} \to \mathcal{B}) \quad \mapsto \quad
    \mathcal{E} \times \mathcal{C} \to \mathcal{B}, \] and that the
  forgetful functor $\CatIzcoc{\mathcal{B}} \to \CatIsl{\mathcal{B}}$
  is a $\CatI$-module functor.  Moreover, both $\CatI$-module
  structures are adjoint to an enrichment in $\CatI$, given
  respectively by $\Funzcoc{\mathcal{B}}(\blank,\blank)$ and
  $\Fun_{/\mathcal{B}}(\blank,\blank)$.
  Similarly, if $(\mathcal{B},\mathcal{B}_{L},\mathcal{B}_{R})$ is an
  \icat{} equipped with a factorization system, then the \icats{}
  $\CatIBL$ and $\CatIcoc{\mathcal{B}}$ are $\CatI$-modules, with the
  tensoring given by the cartesian product, and the forgetful functor
  $\CatIcoc{\mathcal{B}} \to \CatIBL$ is a $\CatI$-module functor; it
  is easy to see that this $\CatI$-module structure on
  $\CatIcoc{\mathcal{B}}$ corresponds under the equivalence with
  $\Fun(\mathcal{B},\CatI)$ to that given by taking products with
  constant functors.
\end{construction}

\begin{propn}\label{obs:B0coccotensor}
  \
  \begin{enumerate}[(i)]
  \item\label{it:Bcotens} For any \icat{} $\mathcal{B}$, the tensoring of
    $\CatIsl{\mathcal{B}}$ over $\CatI$ from \cref{obs:ar0catmod} is
    adjoint to a cotensoring, with the cotensor of
  $\mathcal{C} \in \CatI$ and $\mathcal{E} \to \mathcal{B}$ given by
  the pullback
  \[ \mathcal{E}^{\mathcal{C}}_{/\mathcal{B}} :=
    \Fun(\mathcal{C},\mathcal{E})
    \times_{\Fun(\mathcal{C},\mathcal{B})} \mathcal{B}\]
  along the constant diagram functor $\mathcal{B} \to
  \Fun(\mathcal{C},\mathcal{B})$.
\item\label{it:B0cotens} If $\mathcal{B}_{0}$ is a wide subcategory of $\mathcal{B}$,
  then $\CatIzcoc{\mathcal{B}}$ is also cotensored over $\CatI$, with
  the cotensor of $\mathcal{C} \in \CatI$ and $\mathcal{E} \to
  \mathcal{B}$ again given by
  $\mathcal{E}^{\mathcal{C}}_{/\mathcal{B}}$. In particular, the
  forgetful functor $\CatIzcoc{\mathcal{B}} \to \CatIsl{\mathcal{B}}$
  preserves the cotensoring.
  \end{enumerate}
\end{propn}
\begin{proof}
  Part \ref{it:Bcotens} follows from the natural equivalences
  \[ \Map_{\CatIsl{\mathcal{B}}}(\mathcal{C} \times \mathcal{F},
    \mathcal{E}) \simeq \left\{
      \begin{tikzcd}
        \mathcal{C} \times \mathcal{F} \arrow{r} \arrow{d} &
        \mathcal{E} \arrow{d} \\
        \mathcal{C} \times \mathcal{B} \arrow{r}{\txt{proj}} & \mathcal{B}
      \end{tikzcd}
    \right\} \simeq \left\{
      \begin{tikzcd}
        \mathcal{F} \arrow{r} \arrow{d} & \Fun(\mathcal{C},
        \mathcal{E}) \arrow{d} \\
        \mathcal{B} \arrow{r}{\txt{const}} & \Fun(\mathcal{C},
        \mathcal{B})
      \end{tikzcd}
    \right\}
    \simeq \Map_{\CatIsl{\mathcal{B}}}(\mathcal{F}, \mathcal{E}^{\mathcal{C}}_{/\mathcal{B}}).
  \]
  To prove \ref{it:B0cotens}, we observe that if
  $\mathcal{E} \to \mathcal{B}$ is in $\CatIzcoc{\mathcal{B}}$, then
  so is $\mathcal{E}^{\mathcal{C}}_{/\mathcal{B}}$ by
  \cite{HTT}*{Proposition 3.1.2.3}, and a morphism
  $[1]\to \mathcal{E}^{\mathcal{C}}_{/\mathcal{B}}$ is cocartesian
  \IFF{} the corresponding map $[1]\times \mathcal{C} \to \mathcal{E}$
  has cocartesian components at every $c \in \mathcal{C}$. Thus a
  morphism $\mathcal{F} \to \mathcal{E}^{\mathcal{C}}_{/\mathcal{B}}$
  over $\mathcal{B}$ preserves cocartesian morphisms over
  $\mathcal{B}_{0}$ \IFF{} the corresponding map
  $\mathcal{F} \times \mathcal{C} \to \mathcal{E}$ preserves
  cocartesian morphisms over $\mathcal{B}_{0}$, so that the previous
  equivalence of mapping spaces restricts on subspaces to an
  equivalence
  \[ \Map_{\CatIzcoc{\mathcal{B}}}(\mathcal{C} \times \mathcal{F},
    \mathcal{E})
    \simeq \Map_{\CatIzcoc{\mathcal{B}}}(\mathcal{F}, \mathcal{E}^{\mathcal{C}}_{/\mathcal{B}}),
  \]
  as required.
\end{proof}

\begin{observation}\label{obs:factcotensor}
  If $(\mathcal{B},\mathcal{B}_{L},\mathcal{B}_{R})$ is an \icat{}
  equipped with a factorization system, then the \icats{} $\CatIBL$
  and $\CatIcoc{\mathcal{B}}$ are similarly cotensored over $\CatI$,
  with the same cotensors as in \cref{obs:B0coccotensor}, and the
  forgetful functor $\CatIcoc{\mathcal{B}} \to \CatIBL$ preserves the
  cotensoring.
\end{observation}

\begin{propn}\label{propn:envcatmod1}
  \
  \begin{enumerate}[(i)]
  \item\label{it:envcat0} Let $\mathcal{B}$ be an \icat{} with a wide subcategory
    $\mathcal{B}_{0}$. Then the left adjoint
    \[ (\blank)\times_{\mathcal{B}}\Ar_{0}(\mathcal{B}) \colon
      \CatIsl{\mathcal{B}} \to \CatIzcoc{\mathcal{B}}\]
    of the forgetful functor from \cref{cor:freecocadj} is a
    $\CatI$-module functor, with the adjunction being an adjunction of $\CatI$-modules.
  \item\label{it:envcatL} If $(\mathcal{B},\mathcal{B}_{L},\mathcal{B}_{R})$ is an
    \icat{} equipped with a factorization system, then the left
    adjoint 
    \[ (\blank)\times_{\mathcal{B}}\ArRB \colon \CatILcoc{\mathcal{B}}
      \to \CatIcoc{\mathcal{B}}\] of the forgetful functor from
    \cref{cor:factadjn} is a $\CatI$-module functor, with the
    adjunction being an adjunction of $\CatI$-modules.
  \end{enumerate}
\end{propn}
\begin{proof}
  The forgetful functor $\CatIzcoc{\mathcal{B}} \to
  \CatIsl{\mathcal{B}}$ is a $\CatI$-module functor by
  \cref{obs:ar0catmod}. By
  \cite{paradj}*{Theorem 3.4.7}, the left adjoint
  then has a canonical oplax $\CatI$-module structure, given for
  $\mathcal{C} \in \CatI$ and $\mathcal{E} \to \mathcal{B}$ in
  $\CatIzcoc{\mathcal{B}}$ by the
  natural map
  \[ (\mathcal{C} \times \mathcal{B}) \times_{\mathcal{B}}
    \Ar_{0}(\mathcal{B}) \to \mathcal{C} \times (\mathcal{B}
    \times_{\mathcal{B}} \Ar_{0}(\mathcal{B}));\]
  this is clearly an equivalence, so the adjunction of
  \cref{cor:freecocadj} lifts to an adjunction of
  $\CatI$-modules. This proves \ref{it:envcat0}, and the proof of
  \ref{it:envcatL} is the same.
\end{proof}

\begin{remark}
  The $\CatI$-module structures on $\CatIzcoc{\mathcal{B}}$ and
  $\CatIsl{\mathcal{B}}$ are
  adjoint to enrichments in $\CatI$, given respectively by
  $\Funzcoc{\mathcal{B}}(\blank,\blank)$ and
  $\Fun_{/\mathcal{B}}(\blank,\blank)$; the equivalence of
  \cref{propn:freefibdesc} is then precisely that induced by the
  $\CatI$-module adjunction from \cref{propn:envcatmod1}. Similarly,
  if $(\mathcal{B},\mathcal{B}_{L},\mathcal{B}_{R})$ is an \icat{}
  equipped with a factorization system, then the equivalence of
  \cref{propn:envfuneq} is also induced by the $\CatI$-module
  adjunction above.
\end{remark}

\begin{lemma}\label{lem:f*cotens}
  \
  \begin{enumerate}[(i)]
  \item\label{it:pbcatmod} For any functor of \icats{} $f \colon \mathcal{A} \to
    \mathcal{B}$ the functor $f^{*} \colon \CatIsl{\mathcal{B}} \to
    \CatIsl{\mathcal{A}}$ given by pullback along $f$ is a
    $\CatI$-module functor and also preserves the cotensoring with $\CatI$.
  \item\label{it:pb0catmod} Suppose $\mathcal{A}$ and $\mathcal{B}$ are \icats{} equipped
    with wide subcategories $\mathcal{A}_{0}$ and $\mathcal{B}_{0}$,
    respectively, and that $f \colon \mathcal{A} \to \mathcal{B}$ is a
    functor that takes $\mathcal{A}_{0}$ into $\mathcal{B}_{0}$. Then
    the functor
    $f^{*} \colon \CatIzcoc{\mathcal{B}} \to \CatIzcoc{\mathcal{A}}$
    given by pullback along $f$ is a $\CatI$-module functor and also
    preserves the cotensoring with $\CatI$.
  \end{enumerate}
\end{lemma}
\begin{proof}
  We prove \ref{it:pbcatmod}; the proof of \ref{it:pb0catmod} is the
  same. The functor $f^{*}$ fits in a commutative triangle
  \[
    \begin{tikzcd}
      {} & \Cat \arrow{dl} \arrow{dr} \\
      \CatIsl{\mathcal{B}} \arrow{rr}{f^{*}} & & \CatIsl{\mathcal{A}}
    \end{tikzcd}
  \]
  where all three functors preserve finite products, and so are
  symmetric monoidal with respect to the cartesian products. Hence
  $f^{*} \colon \CatIsl{\mathcal{B}} \to \CatIsl{\mathcal{A}}$ is a
  $\CatI$-module functor. To see that $f^{*}$ also preserves the
  cotensoring, observe that for $\mathcal{E} \to \mathcal{B}$ in
  $\CatIzcoc{\mathcal{B}}$ or $\CatIsl{\mathcal{B}}$ and $\mathcal{C}
    \in \CatI$ we have a natural commutative cube
  \[
    \begin{tikzcd}[row sep=tiny, column sep=tiny]
      (f^{*}\mathcal{E})_{/\mathcal{A}}^{\mathcal{C}} \arrow{rr} \arrow{dd}
      \arrow{dr} & & \Fun(\mathcal{C}, f^{*}\mathcal{E}) \arrow{dr} \arrow{dd} \\
       & \mathcal{E}_{/\mathcal{B}}^{\mathcal{C}} \arrow[crossing over]{rr}& &
       \Fun(\mathcal{C}, \mathcal{E}) \arrow{dd} \\
       \mathcal{A} \arrow{rr} \arrow{dr} & & \Fun(\mathcal{C},
       \mathcal{A}) \arrow{dr} \\
        & \mathcal{B} \arrow[leftarrow,crossing over]{uu} \arrow{rr}& & \Fun(\mathcal{C}, \mathcal{B})
    \end{tikzcd}
  \]
  where the front, back and right faces are cartesian. The left
  vertical square is therefore also cartesian, giving an equivalence
  \[ (f^{*}\mathcal{E})^{\mathcal{C}}_{/\mathcal{A}} \isoto
    f^{*}(\mathcal{E}^{\mathcal{C}}_{/\mathcal{B}}),\]
  as required.
\end{proof}

\begin{observation}\label{obs:freeLitftr}
  For $f \colon \mathcal{A} \to \mathcal{B}$ a functor that preserves
  wide subcategories $\mathcal{A}_{0}$ and $\mathcal{B}_{0}$, we have
  a commutative diagram
  \[
    \begin{tikzcd}
      \CatI \arrow{dr} \arrow[bend left=10]{drr} \arrow[bend right=10]{ddr} \\
      & \CatIzcoc{\mathcal{B}} \arrow{d}\arrow{r}{f^{*}} &
      \CatIzcoc{\mathcal{A}} \arrow{d} \\
      & \CatIsl{\mathcal{B}} \arrow{r}{f^{*}} & \CatIsl{\mathcal{A}},
    \end{tikzcd}
  \]
  of symmetric monoidal functors (with the cartesian monoidal
  structures). It follows that the commutative square on the bottom
  right (as in \cref{obs:freesubcatftr}) is a square of
  $\CatI$-modules. Similarly, if $f$ is
  compatible with
  factorization systems $(\mathcal{A},
  \mathcal{A}_{L},\mathcal{A}_{R})$ and
  $(\mathcal{B},\mathcal{B}_{L},\mathcal{B}_{R})$, then the
  commutative square
    \[
    \begin{tikzcd}
       \CatIcoc{\mathcal{B}} \arrow{d}\arrow{r}{f^{*}} &
      \CatIcoc{\mathcal{A}} \arrow{d} \\
       \CatILcoc{\mathcal{B}} \arrow{r}{f^{*}} & \CatILcoc{\mathcal{A}},
    \end{tikzcd}
  \]
  is a square of $\CatI$-modules. It follows that for both squares the
  Beck--Chevalley map is a natural transformation of $\CatI$-modules.
\end{observation}

\begin{propn}\label{propn:E-Q-adjunction-Cat-compatible}
  Let $(\mathcal{B},\mathcal{B}_{L},\mathcal{B}_{R})$ be a
  factorization system. Then there is a natural $\CatI$-module structure on the \icat{}
  $(\CatIcoc{\mathcal{B}})_{/\ArRB}$, with the
  tensoring given by cartesian products, and the
  adjunction
  \[ \rmE \colon \CatILcoc{\mathcal{B}} \rightleftarrows
    (\CatIcoc{\mathcal{B}})_{/\ArRB} :\! \rmQ \]
  is compatible with the $\CatI$-module structures. Moreover, $(\CatIcoc{\mathcal{B}})_{/\ArRB}$ is also cotensored over $\CatI$, with the cotensor of $\mathcal{C} \in \CatI$ and $\mathcal{E} \to \ArRB$ in
  $(\CatIcoc{\mathcal{B}})_{/\ArRB}$ being $\mathcal{E}^{\mathcal{C}}_{/\ArRB}$.
\end{propn}
\begin{proof}
  The forgetful functor
  $(\CatIcoc{\mathcal{B}})_{/\ArRB} \to
  \CatIcoc{\mathcal{B}}$ has a right adjoint, which takes
  a cocartesian fibration $\mathcal{E} \to \mathcal{B}$ to the
  projection $\mathcal{E} \times_{\mathcal{B}} \ArRB
  \to \ArRB$. We thus have a commutative diagram
  \[
    \begin{tikzcd}
      {} & \CatIcoc{\mathcal{B}} \arrow{dd}{(\blank)
        \times_{\mathcal{B}} \ArRB} \arrow{dr}{\txt{forget}} 
       \\
      \CatI \arrow{ur}{(\blank) \times \mathcal{B}} 
      \arrow{dr}[swap]{(\blank) \times \ArRB} & & \CatIBL\\
       & (\CatIcoc{\mathcal{B}})_{/\ArRB} \arrow{ur}{\rmQ} 
    \end{tikzcd}
  \]
  of right adjoints, which are then symmetric monoidal functors with respect to cartesian
  products. This in particular shows that
  $(\CatIcoc{\mathcal{B}})_{/\ArRB}$ is a $\CatI$-module, with the
  tensoring given by taking cartesian products, and the functor $\rmQ$
  is compatible with the $\CatI$-module structures. As in
  \cref{obs:ar0catmod}, it follows that the left adjoint $\rmE$ is an
 oplax $\CatI$-module functor, and that the oplax structure maps are
 equivalences; thus we have a $\CatI$-module adjunction.

 To identify the cotensor, we first observe that
 $(\CatIcoc{\mathcal{B}})_{/\ArRB}$ can be described as a subcategory
 of $\Cat_{\infty/\ArRB}$; the $\CatI$-module structures on both are
 clearly compatible, and the latter has a cotensoring given by
 $(\mathcal{C},\mathcal{E} ) \mapsto
 \mathcal{E}^{\mathcal{C}}_{/\ArRB}$ by \cref{obs:B0coccotensor}. It
 thus suffices to show that $\mathcal{E}^{\mathcal{C}}_{/\ArRB}$ is an
 object of $(\CatIcoc{\mathcal{B}})_{/\ArRB}$, i.e. that the composite
 to $\mathcal{B}$ is a cocartesian fibration, that the morphism to
 $\ArRB$ preserves cocartesian morphisms over $\mathcal{B}$, and that a morphism
 $\mathcal{F} \to \mathcal{E}^{\mathcal{C}}_{/\ArRB}$ preserves cocartesian morphisms
 over $\mathcal{B}$ \IFF{} the adjoint map $\mathcal{F} \times
 \mathcal{C} \to \mathcal{E}$ does so. To see this, consider the
 commutative cube
 \begin{equation}
   \label{eq:slicecotenscube}
   \begin{tikzcd}[row sep=tiny,column sep=tiny]
     \mathcal{E}^{\mathcal{C}}_{/\ArRB} \arrow{rr} \arrow{dr} \arrow{dd} & &
     \mathcal{E}^{\mathcal{C}} \arrow{dr} \arrow{dd} \\
     & \ArRB \arrow[crossing over]{rr} & & \ArRB^{\mathcal{C}}
     \arrow{dd} \\
     \mathcal{B} \arrow{rr} \arrow[equals]{dr} & & \mathcal{B}^{\mathcal{C}}
     \arrow[equals]{dr} \\
     & \mathcal{B} \arrow[leftarrow,crossing over]{uu} \arrow{rr} & & \mathcal{B}^{\mathcal{C}}.
   \end{tikzcd}
 \end{equation}
 Here the top and bottom squares are cartesian, the vertical maps are
 cocartesian fibrations, and both maps to $\Ar_R(\calB)^\calC$
 preserve cocartesian morphisms. 
 It follows that
 $\mathcal{E}^{\mathcal{C}}_{/\ArRB} \to \mathcal{B}$ is a cocartesian
 fibration, and a morphism here is cocartesian \IFF{} its images in
 $\ArRB$ and $\mathcal{E}^{\mathcal{C}}$ are both
 cocartesian. Combining this with the description of cocartesian
 morphisms in $\mathcal{E}^{\mathcal{C}}$ from \cite{HTT}*{Proposition
 3.1.2.1} gives the required
 description of cocartesian morphisms in $\mathcal{E}^{\mathcal{C}}_{/\ArRB}$.
\end{proof}

\begin{observation}\label{obs:envffit}
   Let us write $\Fun_{/\ArRB}^{\mathcal{B}\dcoc}(\blank,\blank)$ for
 the enrichment adjoint to the $\CatI$-module structure on
 $(\CatIcoc{\mathcal{B}})_{/\ArRB}$; this satisfies
 \[ \Map_{\CatI}(\mathcal{C},
   \Fun_{/\ArRB}^{\mathcal{B}\dcoc}(\blank,\blank)) \simeq
   \Map_{(\CatIcoc{\mathcal{B}})_{/\ArRB}}(\mathcal{C} \times \blank,
   \blank);\]
identifying the right-hand side as a fiber product we see that for
$\alpha \colon \mathcal{E} \to \ArRB, \beta \colon \mathcal{F} \to
\ArRB$ we have a natural cartesian square
\[
  \begin{tikzcd}
    \Fun_{/\ArRB}^{\mathcal{B}\dcoc}((\mathcal{E},\alpha),(\mathcal{F},\beta))
    \arrow{r} \arrow{d} & \Fun^{\coc}_{/\mathcal{B}}(\mathcal{E},
    \mathcal{F}) \arrow{d} \\
    \{\alpha\} \arrow{r} & \Fun^{\coc}_{/\mathcal{B}}(\mathcal{E}, \ArRB).
  \end{tikzcd}
\]
Since the functor $\rmE$ is fully faithful and compatible with the
$\CatI$-module structures we conclude that it gives a natural
equivalence
\[ \Fun_{/\mathcal{B}}^{L\dcoc}(\blank,\blank) \isoto
  \Fun_{/\ArRB}^{\mathcal{B}\dcoc}(\rmE(\blank),\rmE(\blank)).
\]
\end{observation}

\begin{observation}\label{obs:f*factsyst}
  Suppose $f \colon \mathcal{A} \to \mathcal{B}$ is a functor
  compatible with specified factorization systems.
  Passing to vertical left adjoints in the commutative square \cref{obs:basechange-envelope}
  yields a Beck--Chevalley transformation
  \[ \rmE_{\mathcal{A}} f^{*} \to f^{\ostar}\rmE_{\mathcal{B}};\]
  Unwinding the definitions, this is given at $\mathcal{E} \to
  \mathcal{B}$ in $\CatIBL$  by the natural map
  \[ \left(\mathcal{E} \times_{\mathcal{B}} \mathcal{A}
    \right)\times_{\mathcal{A}} \Ar_{R}(\mathcal{A}) \to
    \left(\mathcal{E} \times_{\mathcal{B}} \ArRB\right) \times_{\ArRB}
    \Ar_{R}(\mathcal{A}), \] which is an equivalence. The functors and
  transformations here are also compatible with the $\CatI$-module
  structures, by the same argument as in \cref{obs:freeLitftr} , so
  for $\mathcal{E},\mathcal{F} \to \ArRB$ we have a natural
  commutative square in which the vertical maps are equivalences:
  \begin{equation}\label{eq:ftrsquare}
    \begin{tikzcd}
      \Fun_{/\mathcal{B}}^{L\dcoc}(\mathcal{E},\mathcal{F}) \arrow{r}
      \arrow{d}{\sim} &
      \Fun_{/\mathcal{A}}^{L\dcoc}(f^{*}\mathcal{E},f^{*}\mathcal{E})
      \arrow{d}{\sim} \\
      \Fun_{/\ArRB}^{\mathcal{B}\dcoc}(\rmE_{\mathcal{B}}\mathcal{E},
      \rmE_{\mathcal{B}}\mathcal{F}) \arrow{r} &
      \Fun_{/\Ar_{R}(\mathcal{A})}^{\mathcal{A}\dcoc}(\rmE_{\mathcal{A}}f^{*}\mathcal{E},
      \rmE_{\mathcal{A}}f^{*}\mathcal{F}).
    \end{tikzcd}    
    \end{equation}
\end{observation}

After these preliminaries we are finally ready to consider fibrous
patterns and their envelopes. First, we want to show that the \icats{}
$\Fbrs(\mathcal{O})$ and
$\Seg^{/\mathcal{A}_{\mathcal{O}}}_{\mathcal{O}}(\CatI)$ have
$\CatI$-module structures inherited from those we have already
considered. This is slightly complicated by the fact that
$\Fbrs(\mathcal{O})$ may not be closed under tensors in
$\CatIOintcoc$, and similarly for the relative Segal objects. (For
example, for $\mathcal{O} \in \Fbrs(\xF_{*})$ and $\mathcal{C} \in
\CatI$, the \icat{} $\mathcal{C}
\times \mathcal{O}$ is not an object of $\Fbrs(\xF_{*})$ since its
fiber over $\angled{0}$ is $\mathcal{C}$, not $*$; on the other hand,
$\Fbrs(\xF_{*}^{\natural})$ \emph{is} closed under tensoring with
$\CatI$.) Luckily, cotensors are better behaved:

\begin{propn}\label{propn:cotensfbrs}
  Let $\mathcal{O}$ be an algebraic pattern.
  \begin{enumerate}[(i)]
  \item\label{it:fbrscotens} For $\mathcal{P} \in \Fbrs(\mathcal{O})$ and $\mathcal{C} \in
    \Cat$, the cotensor $\mathcal{P}^{\mathcal{C}}_{/\mathcal{O}}$ in
    $\CatIOintcoc$ is again fibrous.
  \item\label{it:relsegcotens} For $X \in
    \Seg^{/\mathcal{A}_{\mathcal{O}}}_{\mathcal{O}}(\CatI)$
    corresponding to $\mathcal{X} \in
    (\CatIcoc{\mathcal{O}})_{/\AractO}$ and $\mathcal{C} \in
    \Cat$, the cotensor $\mathcal{X}^{\mathcal{C}}_{/\AractO}$ in
    $(\CatIcoc{\mathcal{O}})_{/\AractO}$ again straightens to a
    relative Segal object.
  \end{enumerate}
\end{propn}
\begin{proof}
  To prove \ref{it:fbrscotens}, first observe that we can
  identify
  $\mathcal{P}^{\mathcal{C}}_{/\mathcal{O}} \times_{\mathcal{O}}
  \mathcal{O}^{\act}_{/O}$ as the fiber product $\Fun(\mathcal{C},\mathcal{P}
  \times_{\mathcal{O}} \mathcal{O}^{\act}_{/O})
  \times_{\Fun(\mathcal{C}, \mathcal{O}^{\act}_{/O})}
  \mathcal{O}^{\act}_{/O}$, so that we have a commutative cube
  \[
    \begin{tikzcd}[row sep=tiny, column sep=-1ex]
      \mathcal{P}^{\mathcal{C}}_{/\mathcal{O}} \times_{\mathcal{O}}
      \mathcal{O}^{\act}_{/O} \arrow{rr} \arrow{dr} \arrow{dd} & &
      \Fun(\mathcal{C}, \mathcal{P} \times_{\mathcal{O}}
      \mathcal{O}^{\act}_{/O}) \arrow{dd} \arrow{dr} \\
       & \lim_{E \in \mathcal{O}^{\el}_{O/}} \mathcal{P}^{\mathcal{C}}_{/\mathcal{O}} \times_{\mathcal{O}}
      \mathcal{O}^{\act}_{/E} \arrow[crossing over]{rr} & &   \Fun(\mathcal{C},  \lim_{E \in
        \mathcal{O}^{\el}_{O/}}\mathcal{P} \times_{\mathcal{O}}
      \mathcal{O}^{\act}_{/E})  \arrow{dd} \\
       \mathcal{O}^{\act}_{/O} \arrow{rr} \arrow{dr} & & \Fun(\mathcal{C},
      \mathcal{O}^{\act}_{/O}) \arrow{dr} \\
      & \lim_{E \in \mathcal{O}^{\el}_{O/}} \mathcal{O}^{\act}_{/E}
      \arrow{rr} \arrow[leftarrow,crossing over]{uu}& & \Fun(\mathcal{C}, \lim_{E \in \mathcal{O}^{\el}_{O/}} \mathcal{O}^{\act}_{/E}).
    \end{tikzcd}
  \]
  where the front and back faces are cartesian. Here the right vertical
  face is also cartesian since $\mathcal{P}$ is $\mathcal{O}$-fibrous.
  It then
  follows that the left vertical face is also cartesian, \ie{}
  $\mathcal{P}^{\mathcal{C}}_{/\mathcal{O}}$ is also
  $\mathcal{O}$-fibrous.

  For \ref{it:relsegcotens}, we extract the following commutative
  diagram from the cube
  \cref{eq:slicecotenscube} that describes
  $\mathcal{X}^{\mathcal{C}}_{/\AractO}$:
  \[
    \begin{tikzcd}[row sep=tiny, column sep=-1ex]
      (\mathcal{X}^{\mathcal{C}}_{/\AractO})_{O} \arrow{rr} \arrow{dr} \arrow{dd} & &
      (\mathcal{X}^{\mathcal{C}})_{O} \arrow{dd} \arrow{dr} \\
       & \lim_{E \in \mathcal{O}^{\el}_{O/}} (\mathcal{X}^{\mathcal{C}}_{/\AractO})_{E} \arrow[crossing over]{rr} & &    \lim_{E \in
        \mathcal{O}^{\el}_{O/}} (\mathcal{X}^{\mathcal{C}})_{E} \arrow{dd} \\
      \AractO_{O} \arrow{rr} \arrow{dr} & & (\AractO^{\mathcal{C}})_{O} \arrow{dr} \\
      & \lim_{E \in \mathcal{O}^{\el}_{O/}} \AractO_{E}
      \arrow{rr} \arrow[leftarrow,crossing over]{uu}& & \lim_{E \in \mathcal{O}^{\el}_{O/}} (\AractO^{\mathcal{C}})_{E}.
    \end{tikzcd}
  \]
  (Here we have also used $O$ for the constant functor $\mathcal{C} \to
  \mathcal{O}$ with this value.) The front and back vertical faces in
  this cube are cartesian by the definition of
  $\mathcal{X}^{\mathcal{C}}_{/\AractO}$, while the right vertical
  face is cartesian since $\mathcal{X}$ by assumption straightens to a
  relative Segal object (and we can identify
  $(\mathcal{X}^{\mathcal{C}}_{O}$ as $\Fun(\mathcal{C},
  \mathcal{X}_{O})$ etc.). Hence the left vertical face is also
  cartesian, and this is precisely the relative Segal condition for
  $\mathcal{X}^{\mathcal{C}}_{/\AractO}$.
\end{proof}

\begin{cor}\label{cor:localization-Cat-compatible}
  Let $\mathcal{O}$ be an algebraic pattern.
  \begin{enumerate}[(i)]
  \item The localization $\Lfbrs \colon \CatIOintcoc \to
    \Fbrs(\mathcal{O})$ is a localization of $\CatI$-modules.
  \item The localization $\Lrseg \colon
    (\CatIcoc{\mathcal{O}})_{/\AractO} \to
    \Seg^{/\mathcal{A}_{\mathcal{O}}}_{\mathcal{O}}(\CatI)$ is a
    localization of $\CatI$-modules.
  \end{enumerate}
\end{cor}
\begin{proof}
  We prove the first claim; the proof of the second is the same --- in
  particular, both
  follow from \cite{HA}*{Proposition 2.2.1.9}. In order to apply this
  to $\Lfbrs$, we must verify the required hypothesis, which amounts
  to checking that for $\mathcal{C} \in \CatI$ and $\mathcal{E} \in
  \CatIOintcoc$, the canonical map $\mathcal{C} \times \mathcal{E} \to
  \mathcal{C} \times \Lfbrs(\mathcal{E})$ is taken to an equivalence
  by $\Lfbrs$. Equivalently, we must show that for $\mathcal{P} \in
  \Fbrs(\mathcal{O})$, the induced map
  \[ \Map_{\CatIOintcoc}(\mathcal{C} \times \Lfbrs(\mathcal{E}), \mathcal{P}) \to
    \Map_{\CatIOintcoc}(\mathcal{C} \times \mathcal{E}, \mathcal{P}) \]
  is an equivalence. Using the cotensoring, this is the same as the map
  \[ \Map_{\CatIOintcoc}(\Lfbrs(\mathcal{E}),
    \mathcal{P}^{\mathcal{C}}_{/\mathcal{O}}) \to
    \Map_{\CatIOintcoc}(\mathcal{E},
    \mathcal{P}^{\mathcal{C}}_{/\mathcal{O}})\]
  given by composition with the localization map $\mathcal{E} \to
  \Lfbrs(\mathcal{E})$. This map is indeed an equivalence, since
  $\mathcal{P}^{\mathcal{C}}_{/\mathcal{O}}$ is fibrous by \cref{propn:cotensfbrs}.
\end{proof}

\begin{cor}\label{cor:fbrsenvadjcatmod}
  Let $\mathcal{O}$ be a sound pattern. Then we have a commutative
  square
  \[
    \begin{tikzcd}
    \CatIOintcoc \arrow{r}{\Lfbrs} \arrow{d}{\rmE} &
    \Fbrs(\mathcal{O}) \arrow{d}{\sliceEnv{\mathcal{O}}} \\
    (\CatIcoc{\mathcal{O}})_{/\AractO} \arrow{r}{\Lrseg} &
    \Seg^{/\mathcal{A}_{\mathcal{O}}}_{\mathcal{O}}(\CatI)
    \end{tikzcd}
  \]
  of $\CatI$-module functors. Moreover, the adjunction
      \[ \sliceEnv{\calO} \colon \Fbrs(\mathcal{O}) \adj
      \Seg^{/\mathcal{A}_{\mathcal{O}}}_{\mathcal{O}}(\CatI) :\!  \UnOint\]
  of \cref{prop:right-adjoint-to-envelope} is an adjunction of
  $\CatI$-modules, with the right adjoint being a \emph{lax}
  $\CatI$-module functor.
\end{cor}
\begin{proof}
  Let us use the universal property of $\Fbrs(\mathcal{O})$ as a
  $\CatI$-module localization to verify that the composite
  \[ \CatIOintcoc \xto{\rmE} (\CatIcoc{\mathcal{O}})_{/\AractO} \xto{\Lrseg}
    \Seg^{/\mathcal{A}_{\mathcal{O}}}_{\mathcal{O}}(\CatI) \] factors
  through $\Lfbrs$, as a functor of
  $\CatI$-modules. Thus we need to verify that if a morphism
  $\mathcal{E} \to \mathcal{F}$ in $\CatIOintcoc$ is taken to an
  equivalence by $\Lfbrs$, then $\rmE\mathcal{E} \to \rmE\mathcal{F}$
  is taken to an equivalence by $\Lrseg$. The latter condition is
  equivalent to the induced morphism
  \[ \Map(\rmE\mathcal{F}, \mathcal{X}) \to \Map(\rmE\mathcal{E},\mathcal{X})\]
  being an equivalence provided $\mathcal{X}$ is the unstraightening
  of an object in
  $\Seg^{/\mathcal{A}_{\mathcal{O}}}_{\mathcal{O}}(\CatI)$. By
  adjunction this holds \IFF{} the map
 \[ \Map(\mathcal{F}, \rmQ \mathcal{X}) \to
   \Map(\mathcal{E},\rmQ \mathcal{X}) \]
 is an equivalence for all such $\mathcal{X}$, but since $\mathcal{O}$
 is sound the object $\rmQ \mathcal{X}$ is fibrous, and hence this is
 indeed an equivalence as by assumption $\mathcal{E} \to \mathcal{F}$
 is taken to an equivalence by $\Lfbrs$. It follows that the right
 adjoint inherits a lax $\CatI$-module structure.
\end{proof}

\begin{remark}
  For any pattern $\mathcal{O}$ the Segal envelope
  \[ \sliceEnv{\mathcal{O}} \colon \Fbrs(\mathcal{O}) \to \Seg^{/\mathcal{A}_{\mathcal{O}}}_{\mathcal{O}}(\CatI)\]
  is a \emph{lax} $\CatI$-module functor, since it can be defined by restricting
  $\StOint$ to these full subcategories, the inclusions of which are lax $\CatI$-module functors. 
  This suffices to upgrade the envelope to a functor of \itcats{}, and 
  we can see that it is fully faithful since it is obtained by restricting the functor $\StOint \colon \CatIOintcoc \to \Fun(\mathcal{O},\CatI)_{/\calA_\calO}$, which is a fully faithful functor of \itcats{} by \cref{obs:envffit}.
\end{remark}

\begin{propn}\label{propn:pullback-fibrous-seg-Cat-compatible}
  Let $\mathcal{O}$ and $\mathcal{P}$ be algebraic patterns and $f \colon \mathcal{O} \to \mathcal{P}$ a strong Segal morphism.
  \begin{enumerate}[(i)]
  \item\label{it:f*fbrsmod} The functor $f^{*} \colon \Fbrs(\mathcal{P}) \to \Fbrs(\mathcal{O})$ is a lax $\CatI$-module functor and its left adjoint $f_{!}$ is a $\CatI$-module functor.
  \item\label{it:f*rsegmod} The functor $f^{\ostar} \colon \slSegC{\mathcal{P}} \to \slSegCO$ is a lax $\CatI$-module functor and its left adjoint $f_{!}$ is a $\CatI$-module functor.
  \end{enumerate}
\end{propn}
\begin{proof}
  To prove \ref{it:f*fbrsmod}, we observe that $f^{*}$ is obtained by
  restricting
  $f^{*} \colon \CatIOintcoc \to \CatIsl{\mathcal{P}}^{\intcoc}$,
  which is a $\CatI$-module functor by \cref{obs:f*factsyst}, to full subcategories;
  it is therefore a lax $\CatI$-module functor. The
  left adjoint $f_{!}$ is then automatically an oplax $\CatI$-module
  functor, and the oplax structure map is an equivalence \IFF{} the
  right adjoint $f^{*}$ preserves $\CatI$-cotensors, which we know
  from \cref{lem:f*cotens} and \cref{propn:cotensfbrs}. The proof of
  \ref{it:f*rsegmod} is the same.
\end{proof}

\begin{remark}
  It follows that for $\mathcal{Q} \in \Fbrs(\mathcal{O})$ and
  $\mathcal{R} \in \Fbrs(\mathcal{P})$ we have a natural equivalence
  \[ \Fun^{\intcoc}_{/\mathcal{P}}(f_{!}\mathcal{Q}, \mathcal{R})
    \simeq \Fun^{\intcoc}_{/\mathcal{O}}(\mathcal{Q}, f^{*}\mathcal{R}).\]
\end{remark}

\begin{cor}
  Let $f \colon \mathcal{O} \to \mathcal{P}$ be a strong Segal
  morphism between soundly extendable patterns that satisfies the
  hypotheses of \cref{fbrseq}. Then pullback along $f$ gives an
  equivalence
  \[ f^{*} \colon \Fbrs(\mathcal{P}) \isoto \Fbrs(\mathcal{O}) \]
  of $\CatI$-modules. In particular, for any $\mathcal{Q},\mathcal{Q}'$ in
  $\Fbrs(\mathcal{P})$, the induced functor
  \[ \Fun_{/\mathcal{P}}^{\intcoc}(\mathcal{Q},\mathcal{Q}') \to
    \Fun_{/\mathcal{O}}^{\intcoc}(f^{*}\mathcal{Q},
    f^{*}\mathcal{Q}')\]
  is an equivalence. \qed
\end{cor}

\end{document}